\documentclass{article}

\usepackage{mathtools}
\usepackage{graphicx}
\usepackage{etoolbox}
\usepackage{amsmath, amssymb, amsthm}

\newcommand{\dR}{\mathbb{R}}

\newcommand{\dN}{\mathbb{N}}

\newcommand{\dE}{\mathbb{E}}
\newcommand{\dP}{\mathbb{P}}

\newcommand{\cG}{\mathcal{G}}

\newcommand{\ind}{\mathbf{1}}

\providecommand{\given}{}
\DeclarePairedDelimiterXPP{\Pb}[1]{\mathbb{P}}{\lparen}{\rparen}{}{\renewcommand{\given}{\nonscript{}\:\delimsize\vert\nonscript{}\:\mathopen{}} #1}
\DeclarePairedDelimiterXPP{\E}[1]{\mathbb{E}}[]{}{\renewcommand{\given}{\nonscript{}\:\delimsize\vert\nonscript{}\:\mathopen{}} #1}
\DeclarePairedDelimiterX{\Set}[1]\lbrace\rbrace{\renewcommand{\given}{\nonscript{}\:\delimsize\vert\nonscript{}\:\mathopen{}} #1}

\DeclareMathOperator{\diag}{diag}
\DeclareMathOperator{\tr}{tr}

\DeclareMathOperator{\img}{im}

\DeclarePairedDelimiterX{\norm}[1]\lVert\rVert{\ifblank{#1}{\: \cdot \:}{#1}}

\DeclareMathOperator{\Poi}{Poi}

\newcommand{\bilingualcommand}[3]{%
	\newcommand{#1}[1][\ ]{%
		##1%
		\iflanguage{english}{\text{#2}}{%
			\iflanguage{french}{\text{#3}}{}%
		}%
		##1%
	}%
}

\bilingualcommand{\where}{where}{où}
\bilingualcommand{\textif}{if}{si}
\bilingualcommand{\textand}{and}{et}
\bilingualcommand{\textiff}{if and only if}{si et seulement si}
\bilingualcommand{\otherwise}{otherwise}{sinon}

\newcommand{\eps}{\varepsilon}

\newcommand{\quand}{\quad \textand{} \quad}
\newcommand{\qquand}{\qquad \textand{} \qquad}

\newtheorem{theorem}{Theorem}
\newtheorem{lemma}{Lemma}
\newtheorem{proposition}{Proposition}
\newtheorem{definition}{Definition}

\newtheorem{remark}{Remark}
\newtheorem{example}{Example}

\usepackage{xcolor}
\usepackage{plain}
\usepackage{algorithmicx, algpseudocode,float}
\usepackage{algorithm}
\usepackage{fullpage}
\usepackage[colorlinks, citecolor=blue, unicode]{hyperref}

\usepackage{graphicx}
\usepackage{etoolbox}
\usepackage{amsmath, amssymb, amsthm}
\usepackage{tikz}
\usetikzlibrary{shapes}
\usetikzlibrary{topaths,calc}
\tikzstyle{vertex} = [fill,shape=circle]
\tikzstyle{edge} = [fill,line cap=round, line join=round, line width=25pt]
\tikzstyle{elabel} =  [fill,shape=circle,node distance=30pt, fill opacity=0.5]
\pgfdeclarelayer{background}
\pgfsetlayers{background,main}

\algrenewcommand\algorithmicrequire{\textbf{Input:}}
\algrenewcommand\algorithmicensure{\textbf{Output:}}

\newtheorem{assumption}{Assumption}

\numberwithin{equation}{section}

\title{Sparse random hypergraphs: Non-backtracking spectra and community detection\footnote{A preliminary version of this article was presented at the 63rd IEEE Symposium on Foundations of Computer Science (FOCS 2022).}}
\date{\today}
\author{{
\sc Ludovic Stephan}\\[2pt]
Information, Learning and Physics (IdePHICS) Lab\\
École Polytechnique Fédérale de Lausanne (EPFL) \\
Route Cantonale, 1015 Lausanne, Switzerland \\
{\texttt{ludovic.stephan@epfl.ch}}\\[6pt]
{\sc and}\\[6pt]
{\sc Yizhe Zhu} \\[2pt]
Department of Mathematics\\
University of California, Irvine \\
Irvine, CA, USA 92697\\
\texttt{yizhe.zhu@uci.edu}}

\begin{document}
\maketitle 

\begin{abstract}
{We consider the community detection problem in a sparse $q$-uniform hypergraph $G$, assuming that $G$ is generated according to the   Hypergraph Stochastic Block Model (HSBM).  We prove that a spectral method based
on the non-backtracking operator for hypergraphs works with high probability down to the generalized Kesten-Stigum detection threshold conjectured by Angelini et al. (2015). We characterize the spectrum of the non-backtracking operator for the sparse HSBM and provide an efficient dimension reduction procedure using the Ihara-Bass formula for hypergraphs. As a result, community detection for the sparse HSBM on $n$ vertices can be reduced to an eigenvector problem of a $2n\times 2n$ non-normal matrix constructed from the adjacency matrix and the degree matrix of the hypergraph. To the best of our knowledge, this is the first provable and efficient spectral algorithm that achieves the conjectured threshold for HSBMs with $r$ blocks generated according to a general symmetric probability tensor.}
\end{abstract}




\section{Introduction}

The stochastic block model (SBM), first introduced in \cite{holland1983stochastic}, is a generative model for random graphs with a community structure. It serves as a useful benchmark for clustering algorithms on graph data. When the random graph generated by an SBM is sparse with bounded expected degrees, a phase transition has been observed around the so-called Kesten-Stigum threshold: in particular, above this threshold, a wealth of algorithms are known to achieve partial reconstruction \cite{mossel2018proof,abbe2018proof,coja2018information,hopkins2017efficient,ding2022robust}. Most relevant to this line of work are spectral algorithms that use the eigenvectors of a matrix associated with the graph $G$ to perform the reconstruction. In the sparse case, examples include the self-avoiding \cite{massoulie2014community}, non-backtracking \cite{decelle2011asymptotic, krzakala2013spectral, bordenave2018nonbacktracking}, graph powering \cite{abbe2020graph} or distance \cite{stephan2019robustness} matrices. We refer interested readers to the survey \cite{abbe2018community} for more references, including a more in-depth discussion of the Kesten-Stigum threshold.

As a generalization of graphs,  hypergraphs are well-studied objects in combinatorics and theoretical computer science. They are often used as a convenient way to represent higher-order relationships among objects, including co-authorship and citation networks, with more flexibility compared to graph representations \cite{zhou2007learning,benson2016higher,yin2018higher,aksoy2020hypernetwork,chodrow2021generative}. 
In this paper, we consider a higher-order SBM called the hypergraph stochastic block model (HSBM), which was first studied in \cite{ghoshdastidar2017consistency,ghoshdastidar2017uniform}. 
Many results in the SBM translate quite well to the HSBM: for example, the phase transition for exact recovery also happens in the regime of logarithmic degrees \cite{chien2018minimax}. The exact threshold was given in \cite{kim2018stochastic,zhang2021exact} by a generalization of the techniques in \cite{abbe2015community,abbe2016exact}. A faithful representation of a hypergraph is its adjacency tensor and several detection algorithms based on tensor methods do exist \cite{ghoshdastidar2017uniform,ke2020community,zhen2022community}. However, most computations involving tensors are NP-hard \cite{hillar2013most}, and they do not pair well with linear algebra techniques. Instead, many efficient algorithms were studied based on a certain matrix representation of the hypergraph.  These include spectral methods based on the weighted adjacency matrix \cite{chien2018minimax,ahn2019community,cole2020exact,zhang2021exact,dumitriu2021partial}, and semidefinite programming methods \cite{kim2018stochastic,lee2020robust}.

Although there have been rich results in recent years on  HSBMs with various settings, still not much is known about the bounded degree regime.
Angelini et al. \cite{angelini2015spectral} conjectured a detection threshold in the regime of constant average degree based on the performance of the belief propagation algorithm. Also, they conjectured that a spectral algorithm based on non-backtracking operators on hypergraphs could reach the same threshold. Partial recovery in this regime was studied in \cite{ahn2018hypergraph,dumitriu2021partial}, but the conjectured threshold can not be achieved by their methods. It was shown in \cite{Pal_2021} that a spectral algorithm based on the self-avoiding walks on hypergraphs could achieve the threshold for the 2-block case with equal size and a probability tensor given by two parameters.  However, constructing such a self-avoiding matrix is not very efficient and is hard to implement in practice. Below the conjectured threshold, unlike in the graph case \cite{mossel2015reconstruction,banks2016information,gulikers2018impossibility,coja2018information,gu2020non}, no impossibility result for HSBM is known.  In a different direction, hypothesis testing for HSBM in the bounded degree regime was considered in \cite{yuan2021testing}.

In this paper, for an $r$-block HSBM with not necessarily equal block sizes and general parameters given by a probability tensor, we prove that a spectral algorithm based on the non-backtracking operator of hypergraphs can achieve the conjectured  Kesten-Stigum threshold in \cite{angelini2015spectral} efficiently. The detailed setting is given in Section \ref{sec:def_HSBM}.

The non-backtracking operator and its connection to the Zeta function on hypergraphs were studied in \cite{storm2006zeta} for deterministic hypergraphs. The spectral gap of the operator for random regular hypergraphs was shown in \cite{dumitriu2021spectra}. 
We give a detailed analysis of the top eigenvalues and eigenvectors of the non-backtracking operator for inhomogeneous random hypergraphs, which are non-regular. Combined with the Ihara-Bass formula for hypergraphs, it yields the first provably efficient algorithm that works for the very sparse regime where each vertex has a bounded expected degree down to the conjectured computational threshold.
An advantage of the method is that
it works without any prior knowledge of the probability tensor, and it can be used to obtain an estimation of the parameters in the generative model \cite{angelini2015spectral}. 

Another notable feature of our spectral algorithm is that it works without the knowledge of the adjacency tensor ({see Definition \ref{def:adjacency_tensor} in Section \ref{sec:definition}}).  Algorithm \ref{alg:provable_algorithm} is based on the spectral analysis of a $2n\times 2n$ non-normal matrix $\tilde B$, which is closely related to the non-backtracking operator $B$ due to the Ihara-Bass formula. To construct $\tilde B$, all we need to know is the adjacency matrix (also called the similarity matrix, {see Definition \ref{def:adjacency_matrix}}) $A$ that counts the number of hyperedges between different pairs of vertices. A detailed dimension reduction procedure based on $A$ is given in Section \ref{sec:ihara_tildeB}. 

It has been shown that given the information of the adjacency tensor, the exact recovery threshold could be different from the one where only $A$  is observed \cite{kim2018stochastic, chien2018minimax,zhang2021exact}, based on the analysis of SDP and spectral algorithms in \cite{gaudio2022}. Surprisingly, it was shown in a subsequent work \cite{gu2023weak} that when $q=3,4$, our non-backtracking spectral method with only the information from $A$ is optimal: below the  Kesten-Stigum threshold, it is information-theoretically impossible to achieve detection. It was also shown in \cite{gu2023weak} that for $q\geq 7$, detection is still possible below the Kesten-Stigum threshold. It is an interesting open problem to see if any tensor method (even with exponential running time) can achieve detection in this regime.

 In recent years, we have seen tremendous development from random matrix theory in the study of spectra of random graphs, and the non-backtracking operator is a key ingredient to the proof of many new results  \cite{bordenave2018nonbacktracking,bordenave2020detection,stephan2020non, bordenave2020new,bordenave2019eigenvalues, brito2021spectral,benaych2020spectral,alt2021extremal,alt2021delocalization,dumitriu2022extreme}.
Sparse random hypergraphs are fundamental objects in probabilistic combinatorics, but their spectral properties (adjacency tensors or adjacency matrices) have not yet been well understood. Eigenvalues and the spectral norm for adjacency tensors were considered in \cite{friedman1995second,jain2014provable,zhou2021sparse,lei2020consistent,cooper2020adjacency}, and the eigenvalue distributions and concentration of the adjacency matrices were studied in \cite{feng1996spectra,lu2012loose,dumitriu2021spectra,dumitriu2020global,lee2020robust,dumitriu2021partial}. 
We believe the non-backtracking operator for hypergraphs could be a promising tool in the study of sparse random tensors and random hypergraphs, including tensor completion \cite{jain2014provable,montanari2018spectral,cai2021nonconvex,harris2021}, spiked tensor model \cite{richard2014statistical,arous2019landscape,ding2020estimating,auddy2021estimating}, and planted constraint satisfaction problems \cite{krzakala2009hiding,feldman2015subsampled,florescu2016spectral}.

\subsection{Definitions and notations for hypergraphs}\label{sec:definition}

Before a formal definition of HSBM, we introduce several definitions and notations for hypergraphs that will be used throughout the paper.  See Figure \ref{fig:hypergraph} for an example of a hypergraph. 

\begin{definition}[Hypergraph]
 A  \textit{hypergraph} 
 	$G$ is a pair $G=(V,H)$ where $V$ is a set of vertices and $H$ is the set of non-empty subsets of $V$ called \textit{hyperedges}. If any hyperedge $e\in H$ is a set of $q$ elements of $V$, we call $G$ \textit{$q$-uniform}.   In particular, a $2$-uniform hypergraph is an ordinary graph.
  
  The \textit{degree} of a vertex $x\in V$ is the number of hyperedges in $G$ that contains $x$. A $q$-uniform hypergraph is complete if any set of $q$ vertices is a hyperedge. 
\end{definition}

\begin{figure}
    \centering
\begin{tikzpicture}
    \node[vertex, label=right:$1$] (v1) at (0,2) {};
    \node[vertex, label=right:$2$] (v2) at (1.5,3) {};
    \node[vertex, label=left:$3$] (v3) at (4,2.5) {};
    \node[vertex, label=right:$4$] (v4) at (1,0.2) {};
    \node[vertex, label=below right:$5$] (v5) at (3.5,0) {};

    \begin{pgfonlayer}{background}
    \begin{scope}[transparency group,opacity=.5]
    \draw[edge,color=red] (v2) -- (v3) -- (v5) -- (v2);
    \fill[edge,color=red] (v2.center) -- (v3.center) -- (v5.center) -- (v2.center);
    \end{scope}
    \begin{scope}[transparency group,opacity=.5]
    \draw[edge,color=yellow] (v1) -- (v2) -- (v3) -- (v1);
    \fill[edge,color=yellow] (v1.center) -- (v2.center) -- (v3.center) -- (v1.center);
    \end{scope}
    \begin{scope}[transparency group,opacity=.5]
    \draw[edge,color=green] (v5) -- (v4) -- (v1) -- (v5);
    \fill[edge,color=green] (v5.center) -- (v4.center) -- (v1.center) -- (v5.center);
    \end{scope}
    \end{pgfonlayer}
    \node[elabel,color=yellow,label=right:\(e_1\)]  (e1) at (-3, 2.5) {};
    \node[elabel,below of=e1,color=red,label=right:\(e_2\)]  (e2) {};
    \node[elabel,below of=e2,color=green,label=right:\(e_3\)]  (e3) {};
\end{tikzpicture}
    \caption{A $3$-uniform hypergraph with $5$ vertices $\{1,\dots, 5\}$ and three hyperedges $e_1=\{ 1,2,3\}, e_2=\{ 2,3, 5\}, e_3=\{ 1,4,5\}$}
    \label{fig:hypergraph}
\end{figure}
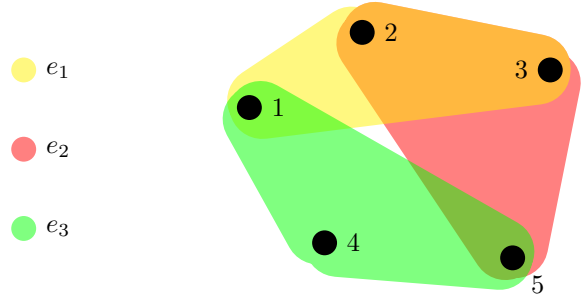

 \begin{definition}[Adjacency tensor of a uniform hypergraph]\label{def:adjacency_tensor}
   The \textit{adjacency tensor} $T$ of a $q$-uniform hypergraph $G=(V,H)$ on $n$ vertices is a $q$-th order symmetric tensor of size $n$ such that for $i_1,\dots, i_q\in V$,
 \[T_{i_1,\dots, i_q}=\mathbf{1} \{ \{i_1,\dots,i_q\} \in H\}.\]   Since  $T_{i_1,\dots, i_q}=T_{s(i_1),\dots, s(i_q)}$ for any permutation $s$ on $[q]$, we may abuse notation and write $T_{e}$ in place of $T_{i_1,\dots, i_q}$ for $e=\{i_1,\dots, i_q\}$.
 \end{definition}

\begin{definition}[Adjacency matrix of a hypergraph]\label{def:adjacency_matrix}
   Define the \textit{adjacency matrix} $A$ of a hypergraph $G=(V,H)$ as 
\begin{equation} \label{eq:DefA}
A_{xy}=\begin{cases}
\#\Set{e\in H \given \{x,y\}\subset e } & x\not=y\\
0 &x=y.
\end{cases}
\end{equation} 
\end{definition}
The definition of an adjacency matrix applies to any hypergraphs. In particular, when $G$ is a $q$-uniform hypergraph with adjacency tensor $T$, 
$ A_{xy}=\sum_{e\in H} \mathbf{1}\{x,y\in e\}T_e.$

\subsection{Non-backtracking operator for hypergraphs}

 For a given hypergraph $G=(V,H)$, let $\vec{H}$ be the \textit{oriented hyperedges} of $G$ such that 
\begin{align*}
    \vec{H}=\{ (x, e) : x\in e\cap V, e\in H \}.
\end{align*}
We shall use the notation $x\to e$ instead of  $(x,e)$ to emphasize that it is an oriented hyperedge. Now we can define the hypergraph non-backtracking operator as follows.
\begin{definition}[Non-backtracking operator for  hypergraphs] For a given  hypergraph $G=(V, H)$, let  $B$ be a matrix indexed by $\vec{H}$ such that
\begin{align} \label{eq:defB}
B_{(x \to e), (y \to f)} = \begin{cases}
    1 & \textif y \in e \setminus \Set{x}, f \neq e,  \\
    0 & \otherwise.
  \end{cases}
\end{align}
  For a $q$-uniform hypergraph, $B$ is of size $q|H|\times q |H|$.  This is a direct extension of the $q = 2$ case, identifying the index $x \to \Set{x, y}$ with the directed edge $(x, y)$.  
\end{definition}
  $B_{(x\to e), (y\to f)}=1$ in \eqref{eq:defB} corresponds to a non-backtracking walk of length $1$ in $G$. See Figure \ref{fig:NB_hypergraph} for an example of a non-backtracking walk corresponding to $(x,e,y,f)$.

\begin{figure} 
    \centering
    \includegraphics[width=0.4\linewidth]{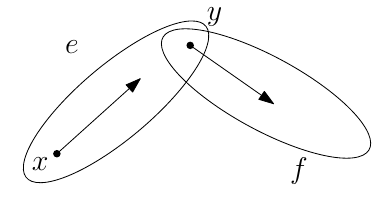}
    \caption{$(x,e,y,f)$  is a non-backtracking walk of length $1$}
    \label{fig:NB_hypergraph}
\end{figure}

 \subsection{Hypergraph stochastic block model}\label{sec:def_HSBM}
We place ourselves in the HSBM as defined in \cite{angelini2015spectral}:
\begin{definition}[Hypergraph stochastic block model]
Consider  an order-$q$  \textit{symmetric probability tensor} $\mathbf P\in \mathbb R^{r^q}$ such that
\[ p_{i_1,\dots,i_q}=p_{s(i_1),\dots, s(i_q)}
\]
for any permutation $s$ on $[q]$.  We sometimes use 
$p_{\underline i}$,  $p_{i, \underline j}$ to specify the index for an entry in $\mathbf P$.
Let the vertex set of a hypergraph be  $V = [n]$, and assign each vertex $x$ a \emph{type} $\sigma(x) \in [r]$. Then, each hyperedge $e$ of size $q$ is included in $H$ with probability
        \[ \Pb{e \in H} = \frac{p_{\underline\sigma(e)}}{\binom{n}{q - 1}}, \]
for any hyperedge $e=\{x_1,\dots, x_q \}$, where $\underline \sigma(e)=\underline\sigma(\Set{x_1, \dots, x_q}):= (\sigma(x_1), \dots, \sigma(x_q))$.
The  \textit{hypergraph stochastic block model} is defined as the distribution of a random hypergraph  $G=(V,H)$ generated according to $\mathbf{P}$ and $\sigma$, where all entries in $\mathbf P$ are constant independent of $n$.
\end{definition}

 The detection problem of HSBM is to find a  partition of the vertex set $V=[n]$ correlated with the true partition $\sigma$. To present our spectral algorithm in the next section, we introduce more assumptions and notations.  

We make the common assumption in the literature that each vertex has the same expected degree. Otherwise, a simple clustering based on vertex degrees correlates with the underlying communities. In our setting, this translates to the following assumption:

\begin{assumption}[Constant average degree]\label{assumption:degree}
    \begin{equation}\label{eq:constant_degree}
    \sum_{\underline j \in [r]^{q-1}} p_{i, \underline j} \prod_{\ell \in \underline j} \pi_{\ell} = d, \quad i \in [r],
\end{equation}
where the $\pi_i$ are the proportions of each type in $V$:
\[ \pi_{i} = \frac{\#\Set{x \in [n] \given \sigma(x) = i}}{n}. \]
\end{assumption}

We similarly define the two-type average degree matrix
\begin{align} \label{eq:defD2}
D_{ij}^{(2)} = \sum_{\underline k \in [r]^{q-2}} p_{ij, \underline k}\prod_{\ell \in \underline k}\pi_\ell.
\end{align}
By our symmetry assumption on $\mathbf P$, $D^{(2)}$ is a symmetric matrix, and we have
\begin{equation}\label{eq:perron_eigenvalue_Q}
    d = \sum_{j \in [r]} D_{ij}^{(2)} \pi_j \quad \text{for all } i \in [r].
\end{equation}

The \emph{signal matrix} $Q$ is given by 
\begin{align}\label{eq:Q=DPi}
 Q_{ij} =  D^{(2)}_{ij}\pi_j    
\end{align}
so $Q=D^{(2)}\Pi$ with $\Pi = \diag(\pi)$. Since $Q$ is similar to the symmetric matrix $ \Pi^{1/2} D\Pi^{1/2}$, its eigenvalues $(\mu_i)_{1\leq i \leq r}$ are all real, and we order them by absolute value:
\[|\mu_r|\leq \cdots \leq |\mu_2|\leq \mu_1 = d.\]
Additionally, the eigenvectors $\phi_i$ of $Q$ are equal to $\Pi^{-1/2}\psi_i$, where $\{\psi_i\}_{1\leq i\leq r}$ is a set of orthonormal eigenvectors of $ \Pi^{1/2} D\Pi^{1/2}$, and hence
\begin{align}\label{eq:pi_orthogonal}
    \langle \phi_i,\phi_j\rangle_{\pi}:=\sum_{k\in [r]}\pi_k\phi_i(k)\phi_j(k)=\langle \psi_i,\psi_j\rangle =\delta_{ij}.
\end{align}
From \eqref{eq:Q=DPi} and \eqref{eq:pi_orthogonal},
We have the following spectral decompositions:
\begin{align}\label{eq:eigen_decom}
\Pi^{1/2} D^{(2)}\Pi^{1/2}&=\sum_{i\in[r]}\mu_i  \psi_i\psi_i^*, \quad D^{(2)}=\sum_{i\in [r]}\mu_i \phi_i \phi_i^*, \quand Q=\sum_{i\in [r]}\mu_i  \phi_i(\Pi\phi_i)^*.
\end{align}

 The matrix $Q$ is significant for the following reason: the nonzero  eigenvalues of $\E A$ are exactly the $\mu_i$, with associated eigenvectors the $\tilde \phi_i \in \mathbb R^{n}$ given by
\begin{equation}\label{eq:ndimlifting}
    \tilde \phi_i(x) = \phi_i(\sigma(x)).
\end{equation}

The \emph{informative} eigenvalues (or \emph{outliers}) are defined as the $\mu_i$ satisfying
\begin{align}\label{eq:kesten_stigum}
    (q-1)\mu_i^2 > d.
\end{align}
The above threshold (and more specifically, the condition $(q-1)\mu_i^2 > d$) is a generalization of the  Kesten-Stigum threshold in the graph case \cite{krzakala2013spectral}. It was conjectured in \cite{angelini2015spectral} that this threshold marks the phase transition for the performance of belief propagation algorithms.
Equivalently, we can define the \textit{inverse  signal-to-noise ratios}
\begin{align}\label{eq:SNR}
    \tau_i=\frac{d}{(q-1)\mu_i^2},
\end{align}
so that equation \eqref{eq:kesten_stigum} is equivalent to $\tau_i < 1$. We shall denote by $r_0$ the number of informative eigenvalues, or equivalently the integer satisfying
\[ (q-1)\mu_{r_0+1}^2 \leq d < (q-1)\mu_{r_0}^2. \]
We always assume in this paper that $r_0 \geq 1$, which is equivalent to 
\begin{align}\label{eq:d_assumption}
    (q-1)d > 1.
\end{align} Otherwise, with high probability, the random hypergraph has no giant component  \cite{schmidt1985component}, and detection is impossible.

 \begin{remark}
 Since we defined $\pi_i = n_i / n$, hypothesis \eqref{eq:constant_degree} can seem unreasonable; for example, it is not satisfied when the vertex types are distributed randomly from a distribution $\hat \pi$ satisfying \eqref{eq:constant_degree}. However, in this case, we have with high probability,
\[ \mu_1 = d (1 + O(n^{-\delta})) \quad \textand \quad \norm*{Q \ind - d \ind}_{\infty} = O(n^{-\delta}), \]
for any $\delta < 1/2$, and since $\mu_1$ is raised to a power of $O(\log(n))$, the approximation errors are negligible in the proof. This case was carefully treated in \cite{bordenave2018nonbacktracking}. To simplify our presentation, we carry out our calculation under Assumption \ref{assumption:degree}.
\end{remark}

\begin{example}[Symmetric HSBM]\label{ex:symmetric_HSBM}
Take $\pi$ uniform over $[r]$, and 
\[ p_{i_1, \dots, i_q} = \begin{cases} c_{\mathrm{in}} &\textif i_1 = \dots = i_q, \\
c_{\mathrm{out}} &\otherwise. \end{cases} \]
This is the model considered in \cite{Pal_2021} for $r = 2$, and in \cite{angelini2015spectral} for general $r\geq 2$. In this case, the eigenvalues of $Q$ are
\[ \mu_1=d = \frac{1}{r^{q-1}} c_{\mathrm{in}} + \left(1 - \frac{1}{r^{q-1}}\right)c_{\mathrm{out}} \quand \mu_2 = \frac{c_{\mathrm{in}} - c_{\mathrm{out}}}{r^{q-1}}, \]
with $\mu_2$ having multiplicity $r-1$. Hence, the Kesten-Stigum threshold is given by
\[ (q-1)(c_{\mathrm{in}} - c_{\mathrm{out}})^2 > r^{q-1}(c_{\mathrm{in}} + (r^{q-1}-1)c_{\mathrm{out}}). \]
\end{example}

\section{Main results}
\subsection{Spectrum of the non-backtracking matrix}
Our first main result is a precise description of the spectrum of $B$, that generalizes the classical results for random graphs \cite{bordenave2018nonbacktracking,stephan2020non,bordenave2020new,brito2021spectral,coste2021simpler,abbe2020learning}:

\begin{theorem}[Spectrum of $B$]\label{thm:main}
Let $G$ be a hypergraph generated according to the HSBM with $m$ hyperedges, and $B$ be its non-backtracking matrix. Denote 
$|\lambda_1(B)|\geq |\lambda_2(B)|\geq \cdots \geq |\lambda_{qm}(B)|$ the eigenvalues of $B$ in   decreasing order of modules. Under Assumption \ref{assumption:degree}, when $n$ is large enough, the following holds with probability at least $1 - n^{-c}$ for some constant $c$:
\begin{enumerate}
    \item For any $i\in [r_0]$,  there exists a constant $c'>0$ such that 
    \[ \lambda_{i}(B) = (q-1)\mu_i + O(n^{-c'}). \]
    \item For all $r_0<i\leq qm$, 
    \[|\lambda_i(B)|\leq (1+o(1)) \sqrt{(q-1)d}.\]  
\end{enumerate}
\end{theorem}

Theorem \ref{thm:main} rigorously confirms the eigenvalue properties of $B$ described in \cite{angelini2015spectral}: if there are $r_0$ informative eigenvalues of $Q$ above the generalized Kesten-Stigum threshold in \eqref{eq:kesten_stigum}, then there are exactly $r_0$ outliers of $B$ outside the disk of radius $\sqrt{(q-1)d}$. See Figure \ref{fig:spectrum} for a simulation of the spectrum of $B$.

\begin{remark}[Erd\H{o}s-R\'{e}nyi hypergraphs]
If each entry in the probability tensor  $\mathbf P$ is $d$, from Theorem \ref{thm:main}, we obtain the result for   a $q$-uniform Erd\H{o}s-R\'{e}nyi hypergraph $G\left(n,q,d/\binom{n}{q-1}\right)$ with $d>(q-1)^{-1}$: with high probability, 
\[\lambda_1(B)=(q-1)d+o(1) \quand |\lambda_2(B)|\leq \sqrt{(q-1)d}+o(1).\] This can be seen as a \textit{Ramanujan property} for non-regular hypergraphs, which generalizes the result for Erd\H{o}s-R\'{e}nyi graphs proved in \cite{bordenave2018nonbacktracking}. Regular Ramanujan hypergraphs were studied in \cite{sole1996spectra,li2004ramanujan}, and it was shown in \cite{dumitriu2021spectra} that a random regular hypergraph is almost Ramanujan with high probability. 
\end{remark}

\begin{figure}    
\centering
\includegraphics[width=0.9\linewidth]{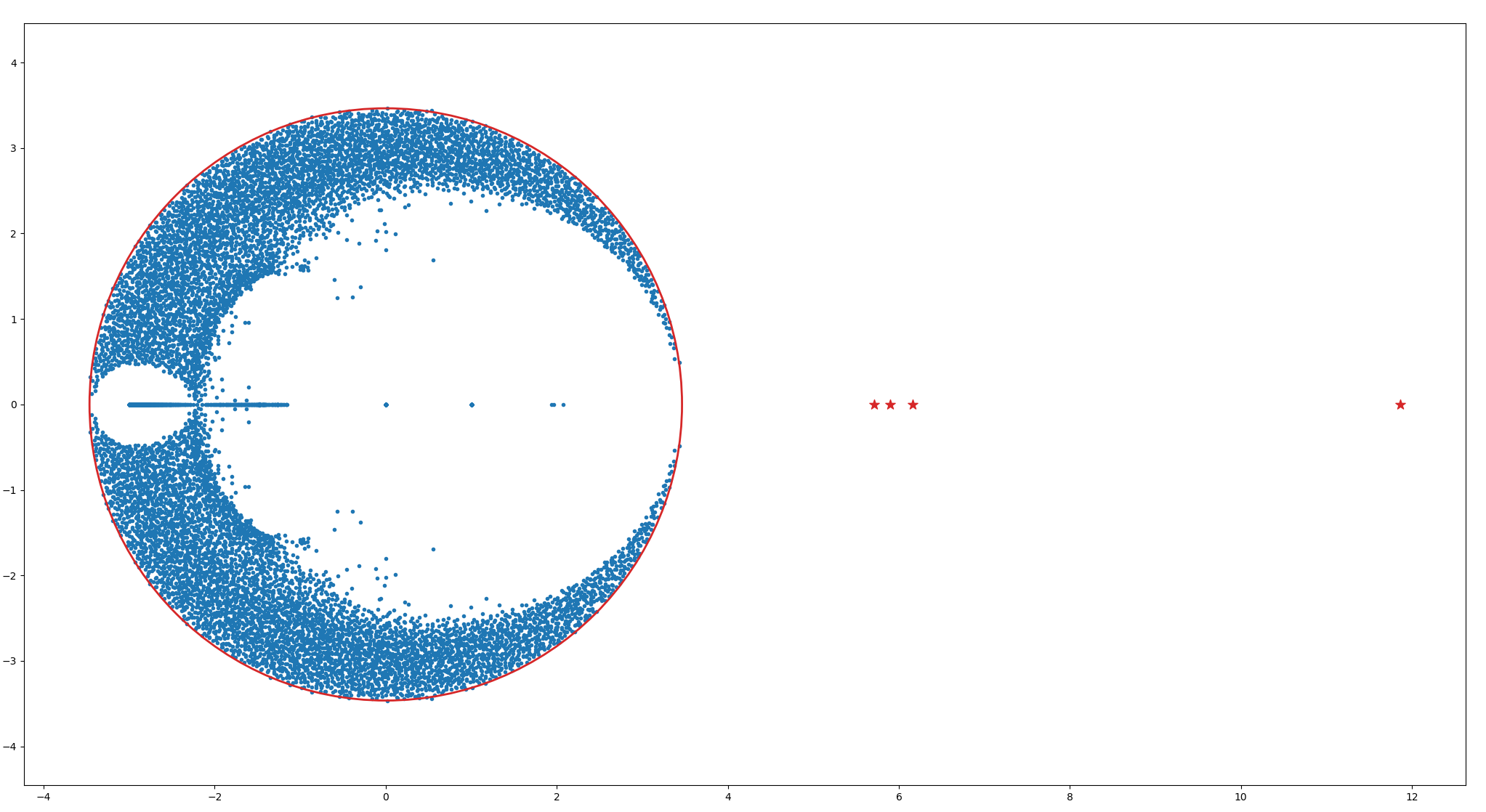}
\caption{Spectrum of a symmetric HSBM (see Example \ref{ex:symmetric_HSBM}) with $n = 6000$, $q = r = 4$. The parameters $c_\mathrm{in}$ and $c_{\mathrm{out}}$ have been chosen so that $d = 4$ and $\mu_2 = 2$. The outlier eigenvalues and the bulk circle have been outlined in red. The single eigenvalue close to $(q-1)d = 12$ and the three eigenvalues near $(q-1)\mu_2 = 6$ are clearly visible.}\label{fig:spectrum}  
\end{figure} 

\subsection{Dimension reduction via the Ihara-Bass formula}\label{sec:ihara_tildeB}

The spectral decomposition of $B$ is not well-suited for a community detection algorithm. First, $B$ has size $q |H| \sim q dn$, which becomes prohibitive compared to a matrix of size $n$. Second, we need to specify a procedure to embed the eigenvectors of $B$ from the oriented hyperedge space into the vertex space. In this section, we provide a result on a lower dimensional matrix.
Define the $2n \times 2n$ matrix $\tilde B$ as
\begin{equation}\label{eq:def_tilde_B}
\tilde B = \begin{pmatrix} 0 &  (D-I)\\  -(q-1)I & A-(q-2)I
\end{pmatrix},
\end{equation}
where $A$ is defined in \eqref{eq:DefA}, and $D$ is the diagonal \textit{degree matrix} with
\begin{align} \label{eq:def_degree_matrix}
D_{xx} = \#\{e\in H:x\in e \}, \quad \text{ or equivalently, }\quad  D_{xx}=\frac{1}{q-1}\sum_{y\in [n]}A_{xy}.
\end{align}
We first obtain an extension of the Ihara-Bass formula \cite{bass_iharaselberg_1992} to $q$-uniform hypergraphs, and show the following relation between $B$ and $\tilde B$. Similar formulas were obtained in \cite{storm2006zeta} for regular hypergraphs, and we provide a simple linear algebra proof for general $q$-uniform hypergraphs, following the strategy from \cite{kempton2016non}.

\begin{lemma}[Ihara-Bass formula for uniform hypergraphs]\label{lem:Ihara-Bass}
The following identity holds for any $z\in \mathbb C$:
\begin{align*}
    \det(B-zI)=(z-1)^{(q-1)m-n}(z+(q-1))^{m-n}\det \left( z^2+(q-2)z-zA+(q-1)(D-I)\right).
\end{align*}
 In particular,  the spectrum of $\tilde B$ is identical to that of $B$, except for possible trivial eigenvalues at $-1$ and $-(q-1)$.
  If additionally $m\geq n$, then the trivial eigenvalue of $B$ located at  $1$ has multiplicity $(q-1)m-n$, and the eigenvalue $-(q-1)$ has multiplicity $m-n$. 
\end{lemma}
\begin{remark}
In the sparse HSBM setting, we have $m=\overline{d}n/q$, where $\overline{d} = d+o(1)$ is the average degree in $G$; therefore it's not always true that $m\geq n$ and $-(q-1)$ is not always an eigenvalue of $B$. In particular, if $q\geq 3$ and $(q-1)^{-1}< d < q-1$, with  probability $1-n^{-c}$ for sufficiently large $n$, $B$ does not have an eigenvalue at $-(q-1)$. This is a consequence of Theorem \ref{thm:main}, since otherwise, the trivial eigenvalue $-(q-1)$ would be outside the bulk of radius $\sqrt{(q-1)d}$. 
\end{remark}

With Lemma \ref{lem:Ihara-Bass}, we  provide an eigenvector overlap result for $\tilde B$.  This implies that to get spectral information of $Q$, computing eigenvalues and eigenvectors of $\tilde{B}$ (instead of $B$) is sufficient.

\begin{theorem}[Eigenvector overlaps]\label{th:main_reduced}
Under Assumption \ref{assumption:degree}, for $i \in [r_0]$, let $\tilde u_i$ be the last $n$ entries of the $i$-th   eigenvector of $\tilde B$. Then with probability $1-n^{c'}$ for some constant $c'>0$, there exists an eigenvector $\phi_i'$ of $Q$ associated to $\mu_i$ such that
    \begin{equation}\label{eq:scalar_characterization}
    \frac{\langle \tilde u_i, \tilde \phi_i' \rangle}{\norm{\tilde u_i}\,\norm{\tilde \phi_i'}} = \sqrt{\frac{1 - \tau_i}{1 + \frac{q-2}{(q-1)\mu_i}}} + O(n^{-c})\quad \text{where}\quad \tilde \phi_i'(x) = \phi_i'(\sigma(x)), ~\forall x\in [n].
    \end{equation}
\end{theorem}

\begin{remark}[Removing the simple eigenvalues condition in \cite{bordenave2018nonbacktracking}]\label{remark:Remove}
When $q=2$, the eigenvector overlap result in \cite[Theorem 4]{bordenave2018nonbacktracking} only applies to eigenvectors associated with simple eigenvalues of $Q$, and such restriction was lifted in \cite{stephan2019robustness,stephan2020non}. Equation \eqref{eq:scalar_characterization} gives a more precise eigenvector overlap estimate compared to \cite{bordenave2018nonbacktracking}. Our proof of Theorem \ref{th:main_reduced} relies on the more refined eigenvector perturbation analysis in \cite[Theorem 9]{stephan2020non}, and it works without the simple eigenvalue assumption on $Q$.  
\end{remark}

\begin{remark}
The  quantity $\frac{(q-2)}{(q-1)\mu_i}$ in \eqref{eq:scalar_characterization} is nonzero only when $q\geq 3$. This is because two distinct hyperedges in a random hypergraph can share more than $1$ vertices only when $q\geq 3$, and such higher-order correlation appears in the analysis of eigenvector overlaps. See Section \ref{sec:functional} for further details. 
\end{remark}

 Theorem \ref{th:main_reduced} reduces the eigenvector problem of $B$ to a $2n \times 2n$ matrix $\tilde B$ and provides a natural embedding into the vertex space. This characterization is given in the graph case in \cite{stephan2020non}. The dimension reduction procedure was proposed for HSBM in \cite{angelini2015spectral}, and we rigorously proved it. See Figure \ref{fig:eigenvec} for a simulation on $\tilde u_i$\footnote{The code for simulating Figures \ref{fig:spectrum} and \ref{fig:eigenvec} can be found at \url{https://github.com/Aufinal/hsbm/}.}.

 \begin{figure} 
     \centering
\includegraphics[width=0.9\textwidth]{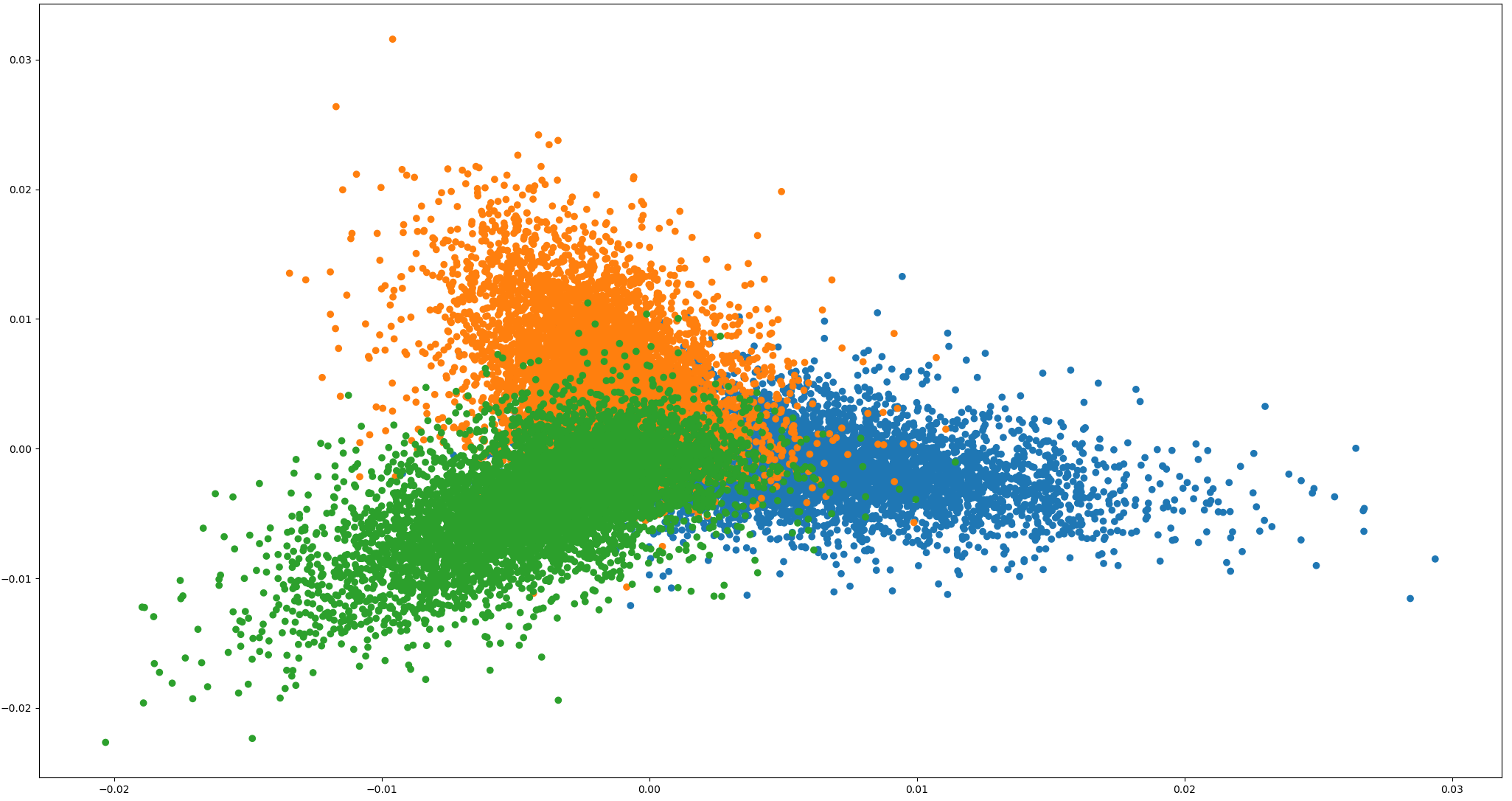}
     \caption{Scatter plot of the last $n$ entries for the second and third eigenvector of $\tilde B$ under the symmetric HSBM (see Example \ref{ex:symmetric_HSBM}) with $q=4$, $r=3$ and $n=20000$. The parameters $c_\mathrm{in}$ and $c_{\mathrm{out}}$ have been chosen so that $d = 4$ and $\mu_2 = 2$. The colors correspond to the actual label of each vertex. Despite overlap near $(0, 0)$ (which is expected due to the presence of isolated vertices and small connected components), the three communities are easily recognizable.}\label{fig:eigenvec}
 \end{figure}

\subsection{Spectral reconstruction algorithm}\label{sec:alg}

A classical measure of the performance of an estimator $\hat\sigma$ is the \emph{overlap}, defined as
\begin{align} \label{eq:def_ov}
\mathrm{ov}(\sigma, \hat\sigma) = \max_{\tau \in \mathfrak S_r} \frac1n \sum_{x\in[n]} \ind_{\tau \circ \sigma(x) = \hat \sigma(x)}, 
\end{align}
where $\mathfrak S_r$ is the set of permutations on $r$ elements.
We say that an estimator achieves \emph{weak reconstruction} if it does better than the ``dummy'' estimator that assigns everyone to the largest community or, equivalently if
\[ \mathrm{ov}(\sigma, \hat \sigma) > \max_{i \in [r]} \pi_i  + \eps\]
with high probability, for some $\eps > 0$.

In direct line with \cite{bordenave2018nonbacktracking,stephan2019robustness}, we provide an algorithm that provably achieves weak reconstruction in the HSBM above the Kesten-Stigum threshold. This procedure is described in Algorithm \ref{alg:provable_algorithm}; we stress that this is not an optimal algorithm by any means, but it is more amenable to theoretical study than other commonly used methods. The theoretical guarantees for Algorithm \ref{alg:provable_algorithm} are summarized in the following theorem:

\begin{algorithm}
\caption{Provably efficient reconstruction algorithm when $\pi_i=\frac{1}{r}, i\in [r]$}\label{alg:provable_algorithm}

\begin{algorithmic}[1]
\Require hypergraph $G=([n], H)$, thresholding parameter $K$

\State Form the $2n\times 2n$ matrix $\tilde B$ as in \eqref{eq:def_tilde_B}.
\State Compute the eigenvector $\xi$ associated with $\lambda_2(\tilde B)$.
\State Define $\tilde u$ the last $n$ entries of $\xi$, normalized so that $\norm{\tilde u}^2 = n$.
\State Partition the vertex set $[n]$ in two sets $(V^+, V^-)$ randomly, such that
\[ \Pb{x \in V^+ \given u} = \frac{1}{2} + \frac1{2K} u(x) \ind_{|u(x)| \leq K}. \]
\State Assign vertices in $V^+$ to the community $1$ and vertices in $V^-$ to the community $2$.
\end{algorithmic}
\end{algorithm}

\begin{theorem}[Weak reconstruction algorithm]\label{thm:weakreconstruction} Under Assumption \ref{assumption:degree} and the extra condition that the communities are balanced, i.e., $\pi_i = 1/r$ for all $i\in [r]$, there exists a deterministic threshold $K=2\sqrt{2}r\gamma_2^{1/2}$  such that for some constants $c,c'>0$, with probability at least $1-n^{-c}$,
the estimator $\hat\sigma$ output by Algorithm \ref{alg:provable_algorithm} satisfies
\[ \mathrm{ov}(\sigma, \hat\sigma) \geq  \frac 1r + \frac1{8r\gamma_2} + O(n^{-c'}), \]
where $\gamma_2$ is defined in \eqref{eq:def_gamma_i}.
\end{theorem}

Our Theorems \ref{thm:main} and \ref{th:main_reduced} work for general unbalanced $\pi_i$. However, translating the correlated eigenvector estimators in Theorem \ref{th:main_reduced} to a weak recovery algorithm with guarantee is still a challenging open question for general $\pi_i$. 

{Theorem \ref{thm:weakreconstruction} generalizes results for the 2-block case of equal size in \cite{Pal_2021} to $r$ blocks with general model parameters given by a probability tensor.  Even in the 2-block case, Theorem \ref{thm:weakreconstruction}  is more efficient compared to the algorithm in \cite{Pal_2021}, which counts the number of self-avoiding walks in a hypergraph. In addition, compared to \cite{Pal_2021}, the eigenvector overlap in Theorem \ref{th:main_reduced} and weak recovery guarantee now depend on the signal-to-noise ratio explicitly.}

The proof of Theorem \ref{thm:weakreconstruction} is deferred to Appendix \ref{sec:weakreconstruction}. It follows from a form of weak convergence of the eigenvectors of $\tilde B$ to a limiting random variable. The proof is similar to \cite{stephan2019robustness}, but with stronger probability estimates and a simpler overall bound. As expected, the difference in performance between our estimator and the naive one goes to zero as we approach the Kesten-Stigum threshold.

\subsection{Discussion}

\paragraph{Reconstruction in the HSBM}

We believe that our results are of interest for real-world algorithms, especially for $k$-means. Indeed, as in \cite{lei2015consistency}, the performance of $k$-means can be related to the distance between the $\tilde u_i$ and a matrix of appropriately clustered vectors (here, the vector $\tilde \phi_i$). Hence, \eqref{eq:scalar_characterization} directly influences the performance of reconstruction: the further away from the Kesten-Stigum threshold we are, the better the performance. A practical reconstruction algorithm above the detection threshold is given in Algorithm \ref{alg:practical_algorithm}. {It is an interesting open problem to provide theoretical guarantees for Algorithm~\ref{alg:practical_algorithm}.}

\begin{algorithm}
\caption{Practical reconstruction algorithm}\label{alg:practical_algorithm}

\begin{algorithmic}[1]
\Require hypergraph $G=([n], H)$

\State Form the $2n\times 2n$ matrix $\tilde B$ as in \eqref{eq:def_tilde_B}.
\State Compute the $r_0$ eigenvectors $\xi_1, \dots \xi_{r_0}$ outside the ``bulk" of radius $\sqrt{(q-1)d}$.
\State Define $\tilde u_i$ the last $n$ entries of $\xi_i$, normalized so that $\norm{\tilde u_i}^2 = n$.
\State Run your favorite clustering algorithm (e.g. $k$-means) on $\tilde u_1, \dots, \tilde u_{r_0}$ to get $r_0$ clusters.
\end{algorithmic}
\end{algorithm}

\paragraph{Beyond the sparse setting} We intentionally choose the simplest setting not to complicate further an already cumbersome proof. However, as in \cite{stephan2020non,bordenave2020detection}, it is possible to keep track of how the bounds in Theorem \ref{thm:main} depend on the problem parameters $r, d, p, \tau$. In particular, Theorem \ref{thm:main} still holds whenever
\[ r, d, p = \mathrm{polylog}(n)\quand (1-\tau_i)^{-1} = \mathrm{polylog}(n). \]
An important special case happens whenever $P = \alpha P_0$, with $P_0$ constant and $\alpha = \omega(1)$ is a scaling parameter. In this case, the RHS of \eqref{eq:scalar_characterization} goes to $1$ as $n \to \infty$, and hence we can conjecture that Algorithm \ref{alg:practical_algorithm} achieves asymptotically perfect reconstruction.

\paragraph{Non-uniform hypergraphs} Another natural extension would be to consider models beyond the uniform hypergraph case and allow for different edge sizes. Although the proof of Theorem \ref{thm:main} does not immediately generalize, we believe that our results still hold in the regime where the edge size is upper bounded by a fixed $K > 0$, and the probability of the presence of a size-$q$ edge scales like $c n^{-(q-1)}$ for any $q \leq K$. In this case, the theoretical eigenvalues $(q-1)\mu_i$ are replaced by the eigenvalues of the signal matrix
$$ Q_{ij} = \sum_{q = 2}^K (q-1) \sum_{\underline k \in [r]^{q-2}} \frac{p_{ij,\underline k}}{\binom{n}{q-1}} \prod_{\ell \in \underline k} \pi_\ell.  $$
Some progress has been made in \cite{dumitriu2021partial} based on the spectral analysis of the adjacency matrix for non-uniform hypergraphs with bounded expected degrees, but the detection threshold was not achieved via their method.
After our work appeared on arXiv, the authors in \cite{chodrow2022nonbacktracking} studied spectral clustering algorithms for non-uniform hypergraphs using the non-backtracking operator, and an Ihara-Bass formula for non-uniform hypergraphs was derived.
We believe our analysis based on the non-backtracking operator can be applied to the non-uniform HSBM and reach the threshold conjectured in \cite{chodrow2022nonbacktracking}.

\paragraph{Possible extensions} The definition of $\tilde B$ in \eqref{eq:def_tilde_B} and $\tilde u_i$ in Theorem \ref{th:main_reduced} imply that for any $\lambda\in \mathbb R$, $\tilde u_i$ is an eigenvector of the Hermitian matrix
\begin{equation}\label{eq:def_bethe_hessian}
    \Delta(\lambda) = \lambda(q-2+\lambda)I - \lambda A + (q-1)(D - I). 
\end{equation} 
This is reminiscent of the arguments justifying the use of the Bethe-Hessian \cite{saade2014spectral,dall2019revisiting,dall2020community,dall2021nishimori} in graph community detection; we conjecture that an appropriate form of a Bethe-Hessian algorithm can be of use in the hypergraph case.

\paragraph{Other spectral properties}
It remains open to characterize the limiting spectral distribution of $B$ even when $q=2$ \cite{bordenave2018nonbacktracking}. When $d\to\infty$ and $q=2$, it is known that the real part of the bulk spectral converges to the semicircle law \cite{wang2017limiting,coste2021eigenvalues}.  For Erd\H{o}s-R\'{e}nyi graphs with fixed $d$, an Alon–Boppana-type bound was conjectured in \cite{bordenave2018nonbacktracking} that $|\lambda_2(B)|\geq  \sqrt{d}+o(1)$, and we see the same phenomenon from our simulation for Erd\H{o}s-R\'{e}nyi hypergraphs. We can also see isolated eigenvalues inside the disk of radius $\sqrt{(q-1)d}$ in our HSBM simulation, which corresponds to the ratio between $\lambda_1(B)$ and $\lambda_r(B)$ for $2\leq r\leq r_0$. This ``eigenvalue insider" phenomenon for graph SBMs was first observed in \cite{dall2019revisiting} when $d=O(1)$, and proved in \cite{coste2021eigenvalues} when $d=\omega(\log n)$. We conjecture similar results hold for the HSBM.

\subsection{Proof overview and technical novelties}
We now provide an overview of the proof and highlight several technical novelties. We start with a few generalities on the non-backtracking matrix and non-backtracking walks in Section \ref{sec:prelim}.

In Section \ref{sec:structure}, we summarize all the technical results about the structure of the matrix $B^{\ell}$, where $\ell=c\log n$ for some suitably chosen constant $c$. Proposition \ref{prop_main} gives us approximate eigenvalues and eigenvectors of $B^{\ell}$ and separations of approximate eigenvalues from the bulk eigenvalues.  In Section \ref{sec:main_proof_perturb},
based on the structure of $B^{\ell}$, we prove Theorems \ref{thm:main} and \ref{th:main_reduced} by a perturbation analysis for non-normal matrices used in \cite{stephan2020non}.

Theorem \ref{th:main_reduced} rigorously justifies a computationally efficient dimension reduction procedure based on the new Ihara-Bass formula (Lemma \ref{lem:Ihara-Bass}), which simplifies the non-backtracking spectral clustering to an eigenvector problem of a $2n\times 2n$ matrix. This is the first result for $\tilde B$ in the literature, even for the graph case: the eigenvector overlap results from \cite{bordenave2018nonbacktracking} are for eigenvectors of $B$, and the asymptotic limits were not given. In addition, the simple eigenvalues condition from \cite{bordenave2018nonbacktracking} is removed, see Remark \ref{remark:Remove}.

The rest of the paper is devoted to proving Proposition \ref{prop_main}. In Section \ref{sec:decomposition}, we generalize the tangle-free matrix decomposition of $B^{\ell}$ in the graph case from \cite{bordenave2018nonbacktracking} to our hypergraph setting (Lemma \ref{lem:expansionBl}), and summarize the operator norm bounds we need in Proposition \ref{prop:trace}. This is essential to proving the separation of the outlier eigenvalues from the bulk.

In Section \ref{sec:local}, we describe the local structure of an HSBM-generated hypergraph; in particular, we prove that it is \emph{tangle-free} and \emph{locally tree-like}, two essential properties in analyzing the spectral algorithms. We also provide a coupling between local neighborhoods and an object called a \textit{Galton-Watson hypertree}, which is important for studying the pseudo-eigenvectors of $B^\ell$. The local analysis of HSBM generalizes previous results in \cite{Pal_2021} from the balanced 2-block HSBM to the more general case with $r$ blocks whose model parameters are from a probability tensor.

Section \ref{sec:functional} is devoted to the study of this hypertree. We construct a family of hypergraph functionals that will be used to approximate eigenvectors of $B$ and analyze their properties, particularly their first and second moments. In particular, compared to the graph case in \cite{bordenave2018nonbacktracking}, a new higher order correlation term appears when $q\geq 3$ (see Proposition \ref{prop:martingale_correlation}).

We then use concentration arguments in Section \ref{sec:spatial_avg} to relate those hypertree eigenvectors to pseudo-eigenvectors of the non-backtracking matrix.
Finally, Sections \ref{sec:near_eigenvec} and \ref{sec:prop_proof} wrap up all the previous arguments to finish the proof of Proposition \ref{prop_main}.

In Appendix \ref{sec:trace_bound}, we prove Proposition \ref{prop:trace} using the high trace method. Such a method was first applied to the non-backtracking operator in \cite{bordenave2018nonbacktracking,bordenave2020new}. For hypergraphs, it was first developed in \cite{Pal_2021} for the self-avoiding matrix, and we generalize it to the non-backtracking matrix. Such a generalization leads to a much more efficient algorithm, while \cite{Pal_2021} is based on counting self-avoiding walks in an HSBM of length $\ell=c\log n$.
The moment method is based on a bipartite representation of the non-backtracking walk on hypergraphs, and such a connection between hypergraphs and bipartite graphs was also used in \cite{dumitriu2021spectra,dumitriu2020global} for random regular hypergraphs. 

Different from the non-backtracking moment proofs in \cite{bordenave2018nonbacktracking,stephan2020non,bordenave2020detection} applied to a random weighted graph, our path-counting combinatorial estimates are done on an associated bipartite graph whose right vertices corresponds to hyperedges in a hypergraph spanned by the concatenations of non-backtracking walks. In addition, handling the dependence among overlapped hyperedges is one of the major challenges in the proof.

Appendix \ref{sec:Misc} contains omitted proofs for several lemmas. In Appendix \ref{sec:weakreconstruction}, we prove Theorem \ref{thm:weakreconstruction}.

\section{Preliminaries}
\label{sec:prelim}
We first introduce more definitions of hypergraphs.
\begin{definition}[Walk and distance]
	A \emph{walk} of length $\ell$ on a hypergraph $G$ is a sequence \[(x_0,e_1,x_1,\cdots ,e_{\ell}, x_{\ell})\] such that $x_{j-1}\not=x_{j}$ and $\{x_{j-1},x_j\}\subset e_j$ for all $1\leq j\leq \ell$. A walk is closed if $x_0=x_{\ell}$. 
		The distance between two vertices $x, y$ in $G$ is the minimal length of all walks between $x,y$. Let $(G,o)_{t}$ be  the set of vertices in $G$ within distance $t$ from $o$, and $\partial (G,o)_t$ be the set of vertices in $G$ at distance $t$ from $o$.
	\end{definition}	

\begin{definition}[Cycle and hypertree]\label{hypertree}
	A \textit{cycle} of length $\ell$ with $\ell\geq 2$ in a hypergraph $H$  is a walk \[(x_0,e_1,\dots, x_{\ell-1},e_{\ell},x_0)\] such that $x_0,\dots x_{\ell-1}$ are distinct vertices and $e_1\dots e_{\ell}$ are distinct hyperedges. A \emph{hypertree} is a hypergraph that contains no cycles.
\end{definition}

Define the \textit{edge reversal operator} $J$ such that
\begin{equation}\label{eq:def_P}
    J_{(x\to e), (y\to f)} = \ind_{e = f, x \neq y}
\end{equation}
Equivalently, for any $u \in \dR^{\vec H}$,
\[ [{J}u](x\to e) = \sum_{x\in e, y\neq x} u_{y \to e}. \]
When $q=2$, {$J$} corresponds to the edge involution operator  in \cite{bordenave2018nonbacktracking} which satisfies $J^2 = I$. {For  $q\geq 3$}, this is not the case, but the following holds:

\begin{lemma}\label{lem:formula_P}
Let {$J$} be the matrix defined in \eqref{eq:def_P}. The following identity holds:
\[ {J}^2 = (q-2){J} + (q-1)I. \]
\end{lemma}
The matrix {$J$} is useful in the study of the non-backtracking matrix $B$ in part due to the following formula, known as the \emph{parity-time} symmetry:
\begin{lemma}\label{lem:BPPB}
  For any $k\geq 0$,
  \[ B^kJ=J{B^*}^k.\]
\end{lemma}
This formula is fairly easy to prove by itself, but it is also the consequence of some deeper structure of $B$. Define the \emph{start} and \emph{terminal} matrices {$S,T\in \mathbb R^{V\times \vec H}$} such that
\begin{equation}\label{eq:def_S_T}
    S_{x, y\to e} = \ind_{x = y} \quand T_{x, y\to e} = \ind_{x\in e, x\neq y} .
\end{equation} 
 The following matrix identities are then easy to check:
\begin{equation}\label{eq:T_S_P_relations}
\begin{aligned}
    SS^* &= D,  &\qquad  TT^*&=(q-2)A+(q-1)D, &\qquad ST^*&= A, \\
    SJ &= T, &\qquad T^*S &= B + J . & 
\end{aligned}
\end{equation}
Equation \eqref{eq:T_S_P_relations} emphasizes the deep connections between $B$ and the matrices $A$ and $D$, which are at the core of Lemma \ref{lem:Ihara-Bass} and Theorem \ref{th:main_reduced}.

\section{Spectral structure of the matrix $B^\ell$}\label{sec:structure}

In this section, we state the main proposition about the structure for $B^{\ell}$ used to show Theorem \ref{thm:main}.  Denote 
        \[p_{\max}=\max_{e} p_{\underline\sigma(e)}> (q-1)^{-1},\]
       { where the inequality is due to \eqref{eq:d_assumption}.}
And we take 
    \[ \ell = \frac1{25}\log_{(q-1)p_{\max}}(n). \]
Recall $\tilde \phi_i$ from \eqref{eq:ndimlifting}. The corresponding lifted eigenvectors of $\chi_i\in \mathbb R^{\vec{H}}$ are defined {for $i \in [r]$} as
\begin{align}\label{eq:defchi}
  {\chi_i = S^*\tilde \phi_i}, \quad \text{or equivalently} \quad \chi_i(x \to e) =  \phi_i(\sigma(x)).
\end{align} 
For $i\in [r_0]$, we define the pseudo-eigenvectors of $B$ as 
\begin{align}\label{def:u_i} u_i = \frac{B^\ell J \chi_i}{[(q-1)\mu_i]^{\ell+1}\sqrt{n}} \qquand v_i = \frac{(B^*)^\ell \chi_i}{[(q-1)\mu_i]^{\ell}\sqrt{n}}.
\end{align}
Denote $U$ (resp. $V$) the $m \times r_0$ matrix whose columns are the $u_i$ (resp $v_i$).
Define
     \begin{align}\label{eq:def_gamma_i_l}
  \gamma_i^{(\ell)} = \frac{1 - \tau_i^{\ell+1}}{1 - \tau_i} + \frac{q-2}{(q-1)\mu_i}\frac{1-\tau_i^{{ \ell}}}{1-\tau_i} = \gamma_i + O(n^{-c}),
    \end{align}
where 
\begin{equation}\label{eq:def_gamma_i}
    \gamma_i := \frac{1 + \frac{(q-2)}{(q-1)\mu_i}}{1-\tau_i}.
\end{equation}
We will also need a slightly different parameter for $U$:
\begin{equation}\label{eq:def_GammaiU}
        \gamma_{i, U}^{(\ell)} = \frac{(d\gamma_i^{(\ell)} + (q-2)\mu_i)}{(q-1)\mu_i^2}.
\end{equation}

Define $\Sigma$   the $r_0\times r_0$ diagonal matrix such that
\[\Sigma_{ij}=(q-1)\mu_i\delta_{ij}.
\] 
Let $P_{\img(U)^\bot}$and $P_{\img(V)^\bot}$ be the projection onto the orthogonal complement of the linear space spanned by the column vectors of $U,V$, respectively. The following proposition holds.
\begin{proposition} \label{prop_main}
There exists a  constant $C = C(r, p_{\max}, q)$ such  with probability at least $1 - n^{-c}$ for a constant $c>0$, the following inequalities hold:
 \begin{align}
     \norm{U^*U - \diag(\gamma_{i, U}^{(\ell)})} &\leq C   n^{-1/4}, \label{eq:norm_bound1}\\
     \norm{V^*V - \diag(d \gamma_{i}^{(\ell)})} &\leq C   n^{-1/4} , \label{eq:norm_bound2}\\
          \norm{U^*V - I_{r_0}} &\leq C   n^{-1/4} ,\label{eq:norm_bound3}\\
     \norm{ V^* B^\ell U  - \Sigma^\ell} &\leq C   n^{-1/4}, \label{eq:norm_bound4}\\
          \norm{B^\ell P_{\img(V)^\bot}} &\leq C\log(n)^{9}[(q-1)d]^{\ell/2},\label{eq:norm_bound6}\\
     \norm{P_{\img(U)^\bot}B^\ell} &\leq C\log(n)^{9}[(q-1)d]^{\ell/2}\label{eq:norm_bound5}\\ 
     \norm{B^\ell} &\leq C \log(n)^{10} [(q-1)d]^\ell \label{eq:norm_bound7}.
 \end{align}
\end{proposition}

{The crucial inequalities for estimating the top $r_0$ eigenvalues are \eqref{eq:norm_bound3} and \eqref{eq:norm_bound4}. Indeed, assuming that $U$ and $V$ are exactly orthogonal (i.e. the RHS of \eqref{eq:norm_bound3} is zero), then the matrix
$B' = U \Sigma V^*$ exactly satisfies \eqref{eq:norm_bound4}, and the eigenvalues of $B'$ are exactly the $(q-1)\mu_i, i\in [r_0]$. \eqref{eq:norm_bound6} and \eqref{eq:norm_bound5} essentially imply part (ii) of Theorem \ref{thm:main} that all other eigenvalues are confined in the disk of radius $\sqrt{(q-1)d}$ asymptotically.}

{Theorems \ref{thm:main}  and \ref{th:main_reduced} ensue from a rigorous proof of these relations through tools from perturbation theory for non-symmetric matrices.}
The proof of this proposition occupies the remainder of this paper. We prove Proposition \ref{prop_main} in Section \ref{sec:prop_proof}.
In the next section, we first show how Proposition \ref{prop_main} implies Theorems \ref{thm:main} and \ref{th:main_reduced}.

\section{Proof of Theorems \ref{thm:main} and \ref{th:main_reduced}}\label{sec:main_proof_perturb}

Our main result follows from Proposition  \ref{prop_main} via perturbation theory arguments. We use the following result from \cite{stephan2020non}:
\begin{lemma}[Theorem  9 from \cite{stephan2020non}]\label{th:bauer_powers}
Let \( \Sigma = \diag(\theta_1, \dots, \theta_r) \) with
\[ 1 = |\theta_1| \geq \cdots \geq |\theta_r|, \]
$A \in \dR^{m \times m}$ and \( U, U', V, V' \in \dR^{m \times r} \). We set
\[ {M} = U\Sigma^\ell V^* \quand {M'} = U'\Sigma^{\ell'}{(V')}^*, \]
for two integers \( \ell, \ell' \).
  Assume the following:
  \begin{enumerate}
    \item the integers \( \ell, \ell' \) are relatively prime,\label{item:bauer_power_cond_i}

    \item the matrices \( U, U', V, V' \) are well-conditioned:\label{item:bauer_power_cond_ii}
          \begin{itemize}
            \item they all are of rank \( r \),
            \item for some \( \alpha, \beta \geq 1 \), for \( X \) in \( \{U, V, U', V' \} \),
                  \[ \norm{X^*X} \leq \alpha \quand \norm{{(X^*X)}^{-1}} \leq \beta, \]
            \item for some small \( \delta < 1 \),
                  \[ \norm{U^*V - I_r} \leq \delta \quand  \norm{{(U')}^*V' - I_r} \leq \delta, \]
          \end{itemize}
    \item there exists a small constant \( \eps > 0 \) such that\label{item:bauer_power_cond_iii}
          \[ \norm{A^\ell - {M}} \leq \eps \quand \norm{A^{\ell'} - {M'} } \leq \eps, \]
    \item if we let\label{item:bauer_power_cond_iv}
          \begin{equation}\label{eq:def_sigma0}
               \sigma_0 := 84 r^3\alpha^{7/2}\beta(\eps + 5r \alpha^2 \beta \delta),
          \end{equation}
          then
          \begin{equation}\label{eq:power_perturbation_bound}
            \sigma_0 < \ell \, |\theta_r|^\ell \quand \sigma_0 < \ell'\, |\theta_r|^{\ell'}.
          \end{equation}
  \end{enumerate}
  Assume without loss of generality that \( \ell \) is odd, and let
  \[ \sigma := \frac{\sigma_0}{\ell |\theta_r|^\ell}. \]
  Then, the \( r \) largest eigenvalues of \( A \) are close to the \( \theta_i \) in the following sense: there exists a permutation $\pi$ of $[r]$ such that for \( i\in [r] \),
  \[ \left| \lambda_{\pi(i)} - \theta_{i} \right| \leq 4\sigma, \]
  and all other eigenvalues of \( A \) have modulus less that \( \sigma_0^{1/\ell} \).
  
  Additionally, for $i \in [r_0]$, let $\xi$ be a unit eigenvector associated to \( \lambda_{\pi(i)} \). Let $\mathrm{Vect}(\Set{u_j \given \theta_j = \theta_i})$ be the vector space spanned by the pseudo-eigenvectors $u_j$ such that $\theta_j=\theta_i$. Then there exists a unit vector  $\zeta \in \mathrm{Vect}(\Set{u_j \given \theta_j = \theta_i})$  such that
  \begin{equation}\label{eq:bauer_eigenvector_bound}
      \norm{\xi - \zeta} \leq \frac{3\sigma}{\delta_i - \sigma},
  \end{equation}
  where $\delta_i:=\min_{\theta_j\neq\theta_i}|\theta_j-\theta_i|$  is the smallest gap between distinct eigenvalues. 
\end{lemma}

    Lemma \ref{th:bauer_powers} can be viewed as an extension of the Bauer-Fike Theorem \cite{bauer1960norms} to powers of matrices. Simply put, having a perturbation bound on the powers of $A$ allows for an improved bound on the eigenvalues of $A$, as long as $\sigma_0 \asymp c^\ell$ for some $c < |\theta_r|$.

\begin{proof}[Proof of Theorem \ref{thm:main}]
Our goal is to check all the assumptions of Lemma \ref{th:bauer_powers}. Define $\ell' = \ell + 1$, and $U', V'$ accordingly. Proposition \ref{prop_main} also applies to $\ell'$, hence there exists an event with probability $1 - n^{-c}$ such that \eqref{eq:norm_bound1} - \eqref{eq:norm_bound6} hold for both $\ell$ and $\ell'$.

\paragraph{Condition~\ref{item:bauer_power_cond_i}} Since \( \ell' = \ell + 1 \), \( \ell \) and \( \ell' \) are relatively prime.

\paragraph{Condition~\ref{item:bauer_power_cond_ii}} {With \eqref{eq:kesten_stigum}  and \eqref{eq:d_assumption}, we have $(q-1)\mu_i>1$ for $i\in [r_0]$. Therefore, for $\ell \geq 1$,  $\gamma_{i}^{(\ell)}$ {defined in \eqref{eq:def_gamma_i_l}} satisfies
\begin{align}
  d \leq d\gamma_i^{(\ell)} \leq d\left( 1+\frac{q-2}{(q-1)\mu_i}\right)\frac{1}{1-\tau_i}\leq \frac{(q-1)d}{1 - \tau_i}, 
\end{align}
and $\gamma_{i,U}^{(\ell)}$ defined in \eqref{eq:def_GammaiU} satisfies \[ \frac{d}{(q-1)\mu_i^2} \leq \gamma_{i, U}^{(\ell)} \leq \frac{2(q-1)}{1 - \tau_i}.\]}
Hence condition \ref{item:bauer_power_cond_ii} is implied by \eqref{eq:norm_bound1} - \eqref{eq:norm_bound3} when $n$ is large enough. In particular, {let $\tau=\max_{i\in [r_0]} \tau_i$}, we can take
\[\alpha =\frac{{3d(q-1)^2}}{1 - \tau}, \quad {\beta^{-1} =\frac{1}{2} \min \left\{ d, \frac{d}{(q-1)\mu_i^2}\right\}}, \quand \delta =  C n^{-1/4} .\]

\paragraph{Condition~\ref{item:bauer_power_cond_iii}} Following \cite{bordenave2020detection}, we write
\[ I_m = P_{\img(U)} + P_{\img(U)^\bot} = P_{\img(V)} + P_{\img(V)^\bot},\]
and since {$MP_{\img(V)^\bot} = P_{\img(U)^\bot}M = 0$}, 
\begin{equation}\label{eq:decomp_Bl_S}
\norm{B^\ell - M} \leq \norm{P_{\img(U)}B^\ell P_{\img(V)}- M} + \norm{P_{\img(U)^\bot}B^\ell} + \norm{B^\ell P_{\img(V)^\bot}}.
\end{equation}
{Since $U, V$ are of rank $r$ by condition \ref{item:bauer_power_cond_ii}}, projections on $\img(U)$ and $\img(V)$ have the expressions
\[ P_{\img(U)} = U(U^*U)^{-1}U^* \quand  P_{\img(V)} = V(V^*V)^{-1}V^*.\]
Hence,
\[\norm{P_{\img(U)}B^\ell P_{\img(V)}- M} \leq \norm{U} \norm{(U^*U)^{-1}U^* B^\ell V(V^*V)^{-1} - \Sigma^\ell}\norm{V}. \]
Define
    \[ \tilde U = V(V^*V)^{-1} \qquand \tilde V = U(U^*U)^{-1}; \]
we expect from \eqref{eq:norm_bound3} that $U$ and $\tilde U$ (resp. $V$ and $\tilde V$) are close. Indeed, write
\[ U = P_{\img(V)}U + P_{\img(V)^\bot} U = \tilde U + E_1 + P_{\img(V)^\bot} U . \]
The second term $E_1$ is equal to {$V(V^*V)^{-1}(V^*U - I_{r_0})$}, therefore
\[ \norm{E_1} \leq \sqrt{\alpha}\beta\delta. \]
We similarly write
\[ V =  \tilde V + E_2 + P_{\img(U)^\bot} V \quad \text{with} \quad \norm{E_2} \leq \sqrt{\alpha}\beta\delta \]
As a result,
\begin{align*}
    \norm{\tilde V^* B^\ell \tilde U - \Sigma^\ell} &\leq \norm{V^* B^\ell U - \Sigma^\ell} + \norm{V^* B^\ell U - \tilde V^* B^\ell \tilde U} \\
    &\leq \norm{V^* B^\ell U - \Sigma^\ell} + \norm{\tilde V} \left(\norm{B^\ell} \norm{E_1} + \norm{B^\ell P_{\img(U)^\bot}} \norm{V} \right) \\
    &\phantom{\leq \norm{V^* B^\ell U - \Sigma^\ell}}\  + \norm{\tilde U} \left(\norm{B^\ell} \norm{E_2} + \norm{P_{\img(V)^\bot}B^\ell } \norm{U} \right) \\
\end{align*}
Combining with \eqref{eq:decomp_Bl_S} and using \eqref{eq:norm_bound6}, \eqref{eq:norm_bound5},  and {\eqref{eq:norm_bound7}}, we finally find
\[ \norm{B^\ell - M} \leq c_1 \log(n)^9[(q-1)d]^{\ell/2} := \eps \]

\paragraph{Condition~\ref{item:bauer_power_cond_iv}}

From \eqref{eq:def_sigma0}, we find
\[ \sigma_0 = c_2 \log(n)^9 [(q-1)d]^{\ell/2}, \]
so \eqref{eq:power_perturbation_bound} reduces to
\[ \tau^{\ell/2} \leq c_3 \log(n)^{-10}, \]
which happens whenever $n$ is large enough.
Since $\ell=\kappa'\log(n)$ for some $\kappa'>0$, $\sigma = O(n^{-c})$ for some constant $c$, this implies Theorem \ref{thm:main}. 
\end{proof}

For Theorem \ref{th:main_reduced}, we first characterize the following relation between eigenvectors of $B$ and $\tilde B$ from \cite{angelini2015spectral} with a complete proof given in Appendix \ref{sec:Misc}. 
\begin{lemma}\label{lem:bprime_spec} 
Let $v$ be an eigenvector of $B$ with eigenvalue $\lambda\not\in \{1, -(q-1)\}$, and define
 \begin{align}
     v_{\mathrm{in}}(x) &= \sum_{e \ni x} [J^{-1}v](x \to e)\notag \\
     v_{\mathrm{out}}(x) &= \sum_{e\ni x} v(x \to e),\label{eq:outransform}
 \end{align}
 where the sum is over all hyperedges $e$ containing $x$.
Then the vector $v' = \begin{pmatrix} v_{\mathrm{in}}\\ v_{\mathrm{out}} \end{pmatrix}$ is an eigenvector of $\tilde B$ with the same eigenvalue $\lambda$.
 Moreover, if $v'$ is an eigenvector of $\tilde B$ associated with an eigenvalue $\lambda\not\in \{1,-(q-1)\}$, then there is an eigenvector $v$ of $B$ associated with $\lambda$ such that $v'=\begin{pmatrix} v_{\mathrm{in}}\\ v_{\mathrm{out}}\end{pmatrix}$.
\end{lemma}

\begin{proof}[Proof of Theorem \ref{th:main_reduced}]
Recall $\tilde{u}_i$ in the statement of Theorem \ref{th:main_reduced}.
From Lemma \ref{lem:bprime_spec}, with a proper scaling of $\tilde{u}_i$, we can assume  \[\tilde u_i=(\xi_i)_{\mathrm{out}}=S\xi_i\] for some unit eigenvector $\xi$ of $B$ with eigenvalue $\lambda_i$, where  $S$ is defined in \eqref{eq:def_S_T}. Recall the definition of $u_i$  from  \eqref{def:u_i}.
From \eqref{eq:bauer_eigenvector_bound} in Lemma \ref{th:bauer_powers}, there exists a unit vector $\zeta_i$  in $\mathrm{Vect}(\Set{u_j \given \mu_j = \mu_i})$ such that
\[ \norm{\xi_i - \zeta_i} = O(n^{-c}). \]
Since the unit vector $\zeta_i$ is a linear combination of vectors in $\{u_j: \mu_j=\mu_i\}$, for  some coefficients $\eta_j$ such that $\sum_{j:\mu_j=\mu_i} \eta_j^2=1$, and a normalizing constant $Z>0$, we have
\begin{align}\label{eq:norm1}
\zeta_i= \frac{1}{Z} B^\ell J\sum_{j: \mu_j=\mu_i}   \eta_j\chi_j=\frac{1}{Z}B^\ell J\chi_i',
\end{align}
where $\chi_i'=\sum_{j: \mu_j=\mu_i}   \eta_j\chi_j$. By linearity,  $\chi_i'$ is the lifted vector of   $\phi_i'=\sum_{j:\mu_j=\mu_i} \eta_j \phi_j$ using the lifting defined in \eqref{eq:defchi}.  
On the other hand, from \eqref{eq:pi_orthogonal} and \eqref{eq:ndimlifting}, we have
\begin{align}
\norm{ \tilde\phi_i}^2 &= n\sum_{j\in[r]} \pi_i \phi_i(j)^2 = n. \notag \\
  \norm{\tilde \phi_i'}^2&=\sum_{j:\mu_j=\mu_i} \eta_j^2 \norm{\tilde{\phi}'_j}^2=n. \notag  
\end{align}

    Since the vector $\zeta_i$ is a unit vector, from \eqref{eq:norm_bound1}, the normalizing constant $Z$ in \eqref{eq:norm1} satisfies
    \begin{align}
       Z^2&= \norm{B^\ell J \chi_i'}^2 =  (1+O(n^{-c}))[(q-1)\mu_i]^{2\ell} n(q-1)(d\gamma_i^{(\ell)}+(q-1)\mu_i).\label{eq:etanorm}
    \end{align}

Lemma \ref{lem:bprime_spec}  implies that  
\begin{align}\label{eq:approx_tilde_u}
\left\|\tilde{u}_i- \frac{1}{Z}SB^\ell J\chi_i'  \right\|=\| (\xi_i)_{\mathrm{out}}-(\zeta_i)_{\mathrm{out}} \|= \| S(\xi_i-\zeta_i)\|\leq \|S\| \|\xi_i-\zeta_i\|=O(n^{-c'})
\end{align}
for a constant $c'>0$ with probability $1-n^{-c'}$, where we used the fact that $SS^*=D$ from \eqref{eq:T_S_P_relations},  and $\|D\|=O(\log n)$ with probability $1-n^{-\omega (1)}$.
As a result, by triangle inequality, the scalar product of interest satisfies
\begin{align} \label{eq:innner}
\frac{\langle\tilde u_i, \tilde \phi_i' \rangle}{\norm{\tilde u_i} \norm{\tilde \phi_i'}} = \frac{\langle\tilde \phi_i', SB^\ell J\chi_i' \rangle}{\sqrt{n} Z\norm{\tilde{u}_i}} + O(n^{-c'})\frac{1}{\norm{\tilde{u}_i}}.
\end{align}

We show in Appendix \ref{sec:lemma6} the following lemma, proven the same way as \eqref{eq:norm_bound1} and \eqref{eq:norm_bound3}:
\begin{lemma} \label{lem:eigen_reduction}
With probability $1-n^{-c'}$ for some constant $c'>0$, for any $i,j\in [r_0]$,
\begin{align}\label{eq:scal_zeta_phi}
    \langle \tilde \phi_i, SB^\ell J\chi_j \rangle &= n[(q-1)\mu_i]^{\ell+1} \delta_{ij}+ O(n^{1-c})\\ \langle SB^\ell J\chi_i,  SB^\ell J\chi_j \rangle  &= n[(q-1)\mu_i]^{2\ell+2}\gamma_i^{(\ell)}\delta_{ij} + O(n^{1-c}).\label{eq:scale_ij_term}
\end{align}
\end{lemma}
As a result, with \eqref{eq:scal_zeta_phi},  \eqref{eq:scale_ij_term}, expanding $\tilde{\phi}_i', \chi_i'$ as  linear combinations of $\tilde\phi_j,\chi_j$, respectively,  we obtain
\begin{align}
\langle\tilde \phi_i', SB^\ell J\chi_i'\rangle &=  n[(q-1)\mu_i]^{\ell+1} +O( n^{1-c})  \label{eq:tildeproduct}\\
\|SB^\ell J\chi_i'\|^2&=  n[(q-1)\mu_i]^{2\ell+2} \gamma_i^{(\ell)} +O(n^{1-c})  .\label{eq:1nc}
\end{align} 
Hence from \eqref{eq:etanorm}, \eqref{eq:approx_tilde_u} and \eqref{eq:1nc},
\begin{align}\label{eq:better_approx_tildeu}
    \|\tilde{u}_i\|=\frac{1}{Z}\sqrt{n} [(q-1)\mu_i]^{\ell+1}\sqrt{\gamma_i^{(\ell)}}+O(n^{-c}).
\end{align}
Therefore, from   \eqref{eq:innner}, \eqref{eq:tildeproduct}, and \eqref{eq:better_approx_tildeu}, we conclude that
\[   \frac{\langle\tilde u_i, \tilde \phi_i' \rangle}{\norm{\tilde u_i} \norm{\tilde \phi_i'}} = \frac{1}{\sqrt{\gamma_i^{(\ell)}}} + O(n^{-c'}) = \frac{1}{\sqrt{\gamma_i}} + O(n^{-c'}), \]
where $\gamma_i$ is defined in \eqref{eq:def_gamma_i}. This completes the proof of Theorem \ref{th:main_reduced}.
\end{proof}

\section{Matrix decomposition and operator norm bounds}\label{sec:decomposition}

We define the hypergraph non-backtracking walk, which can be seen as a vertex-hyperedge sequence starting from a vertex and ending at a hyperedge.
	\begin{definition}[Non-backtracking walk]
	A \textit{non-backtracking walk} of length $\ell$ in $G=(V,H)$ is a walk $\gamma=(x_0,e_0,x_1, e_1,\cdots, x_{\ell},e_{\ell})$ such that 
	\begin{enumerate}
	    \item $x_0\in e_0$, $x_j\in (e_{j}\setminus \{x_{j-1}\})\cap e_{j-1}$ for $1\leq j\leq \ell$,
	    \item  $e_j\not=e_{j+1}$ for $0\leq j\leq \ell-1$.
	\end{enumerate}
\end{definition}
	$\gamma$ can be seen as a walk $(\gamma_0, \gamma_1,\dots,\gamma_{\ell})$ of length $\ell$ on the space of $\vec{H}$ with $\gamma_j=(x_j,e_j)\in \vec{H}, 0\leq j\leq \ell$.
Different from the graph case, our non-backtracking walks on hypergraphs do not have the interpretation as walks on the space of $V=[n]$. However, this is the correct definition to match the combinatorial interpretation for the matrix  power $B^{\ell}$.

We develop a useful way to expand $B^{\ell}$ and control the bulk eigenvalues of $B$ using the moment method.  These are the main ingredients to show \eqref{eq:norm_bound6} and \eqref{eq:norm_bound5}  in Proposition \ref{prop_main}.

For a $q$-uniform hypergraph $G=(V,H)$, let $H(V)$ and $\vec{H}(V)$ be the set of all hyperedges and oriented hyperedges, respectively, in the complete hypergraph indexed by $V$. Define the \textit{incidence matrix} of $G$ in $\mathbb R^{q\times |H(V)|}$ as 
\begin{align*}
    \mathcal A_{x\to e}=\begin{cases}
     1  & \text{if }   e\in H,\\
     0 & \otherwise
    \end{cases}
\end{align*}
for any $x\in e, e\in H(V)$.  Note that in our random hypergraph model $G$,   $\mathcal A_{x\to e}$ are independent for different hyperedges $e$. 
For convenience, we extend matrix $B$ and vector $\chi_i$ defined in \eqref{eq:defchi} to $\mathbb C^{\vec{H}(V)}$. Now we can also write $B$ as 
\begin{align}\label{eq:BtoA}
    B_{(x\to e), (y\to f)}
    &= \mathcal  A_{x\to e}  \mathcal A_{y\to f} \mathbf{1}\{x\in e, f\not=e, y\not=x\}.
\end{align}
  
Define $\Gamma_{(x\to e), (y\to f)}^k$ to be the set of non-backtracking walks of length $k$ denoted by $\gamma=(\gamma_0,\dots,\gamma_{k})$ with $\gamma_0=(x\to e)$ and $\gamma_{k}=(y\to f)$ in $\vec{H}(V)$. According to \eqref{eq:BtoA}, we have for $k\geq 1$,
\begin{align}
    B^k_{(x\to e), (y\to f)}=\sum_{\gamma \in \Gamma_{(x\to e), (y\to f)}^{k}} \prod_{s=0}^{k}\mathcal A_{\gamma_s}.\notag
\end{align}
A hypergraph spanned by $\gamma$ is given by all vertices and hyperedges from  $\gamma$. We will need the following definition of tangle-free property.

\begin{definition}[Tangle-free property]\label{def:tangle_free}
We say $\gamma$ is a \textit{tangle-free} path if the hypergraph spanned by $\gamma$ contains at most one cycle. Otherwise, we call $\gamma$ a tangled path. A hypergraph $G$   is called  $\ell$-tangle-free  if  for any $x\in [n]$, there is at most one cycle in $(G,x)_{\ell}$.
\end{definition}

Let $F_{(x\to e), (y\to f)}^k\subseteq \Gamma_{(x\to e), (y\to f)}^k$ be the subset of all tangle-free paths. 
If $G$ is $\ell$-tangle free, then for all $1\leq k\leq \ell$, we must have $B^k=B^{(k)}$, with  
\begin{align}
    B^{(k)}_{(x\to e), (y\to f)}=\sum_{\gamma \in F_{(x\to e), (y\to f)}^{k}} \prod_{s=0}^{k}\mathcal A_{\gamma_s}. \notag
\end{align}
For any oriented hyperedge $(x\to e)\in \vec{H}(V)$ , define the centered random variable
\[\underline{\mathcal A}_{x\to e}=\mathcal A_{x\to e}-\frac{p_{\underline\sigma(e)}}{\binom{n}{q-1}}. \]
We then define $\Delta^{(k)}$, a centered version of $B^{(k)}$, as 
\begin{align}
  \Delta^{(k)}_{(x\to e), (y\to f)}=  \sum_{\gamma \in F_{(x\to e), (y\to f)}^{k}} \prod_{s=0}^{k}\underline {\mathcal A}_{\gamma_s}.\notag
\end{align}
From the telescoping sum formula
\begin{align}
    \prod_{s=0}^{\ell} a_s=\prod_{s=0}^{\ell} b_s+\sum_{t=0}^{\ell} \prod_{s=0}^{t-1}b_s(a_t-b_t)\prod_{s=t+1}^{\ell} a_s,\notag
\end{align}
we can decompose $B^{(\ell)}$ as 
\begin{align}\label{eq:Bexpansion}
    B^{(\ell)}_{(x\to e), (y\to f)}=\Delta^{(\ell)}_{(x\to e), (y\to f)}+ \sum_{\gamma \in F_{(x\to e), (y\to f)}^{\ell}}\sum_{t=0}^{\ell} \prod_{s=0}^{t-1}\underline{\mathcal A}_{\gamma_s}\frac{p_{\underline\sigma(\gamma_t)}}{\binom{n}{q-1}}\prod_{s=t+1}^{\ell} \mathcal A_{\gamma_s}.
\end{align}
For $0\leq t\leq \ell$, define $R_t^{(\ell)}$ as
\begin{align}
    (R_t^{(\ell)})_{(x\to e), (y\to f)}=\sum_{\gamma \in F_{t,(x\to e), (y\to f)}^{\ell}}\prod_{s=0}^{t-1}\underline {\mathcal A}_{\gamma_s} p_{\underline\sigma(\gamma_t)}\prod_{s=t+1}^{\ell}  \mathcal A_{\gamma_s}.\notag
\end{align}
 
Here the set $F_{t,(x\to e), (y\to f)}^{\ell}\subseteq \Gamma_{(x\to e), (y\to f)}^{\ell}$ is the set of non-backtracking tangled paths $\gamma=(\gamma_0,\dots,\gamma_{\ell})$ such that:
\begin{itemize}
\item if $1 \leq t \leq \ell -1$, both $(\gamma_0,\dots,\gamma_{t-1})$ and $(\gamma_{t+1},\dots,\gamma_{\ell})$ are tangle-free,
\item if $t = 0$ (resp. $t = \ell$),  $(\gamma_1,\dots,\gamma_{\ell})$ (resp. $(\gamma_0, \dots, \gamma_{\ell-1})$) is tangle-free
\end{itemize}

 For $(x\to e), (y\to f) \in \vec{H}$, define
\begin{align}\label{eq:defK}
    K_{(x\to e), (y\to f)}=p_{\underline{\sigma}(e)}~\mathbf{1}((x,e)\to (y,f)) ,
\end{align} 
where $(x,e)\to (y,f)$ represents the non-backtracking property $y\in e, f\not=e$.  We also define  $ K^{(2)}$ as  
\begin{align}\label{eq:defK2}
K^{(2)}_{(x\to e), (y\to f)} =\frac{1}{\binom{n}{q-2}} \sum_{(x, e) \to (w, g) \to (y, f)} p_{\underline \sigma(g)},
\end{align}
where the sum is over all $(w,g)\in \vec{H}$ such that $(x,e)\to(w,g)\to(y,f)$ is a non-backtracking walk of length $2$. The scaling $\frac{1}{\binom{n}{q-2}}$ is to make each entry in $K^{(2)}$  of  order $1$.

By adding and subtracting $\frac{1}{\binom{n}{q-1}}  R_t^{(\ell)}$, to the $t$-th term of the summand in \eqref{eq:Bexpansion} for $0\leq t\leq \ell$, we then have  the following expansion for $B^{(\ell)}$. 
\begin{lemma} \label{lem:expansionBl}
For any $\ell\geq 1$, $B^{(\ell)}$ can be expanded as 
 \begin{align}
   & \Delta^{(\ell)}+\frac{1}{\binom{n}{q-1}}KB^{(\ell-1)}+\frac{q-1}{n-q+2}\sum_{t=1}^{\ell-1} \Delta^{(t-1)}K^{(2)} B^{(\ell-t-1)}+\frac{1}{\binom{n}{q-1}} \Delta^{(\ell-1)}K -\frac{1}{\binom{n}{q-1}}\sum_{t=0}^{\ell} R_t^{(\ell)}.  \notag
\end{align} 
\end{lemma}

{We shall see in the appendix that $K^{(2)}$ is close to the matrix
\begin{equation}\label{eq:def:barD}
    \overline D := T^* D^{(2)} S = \sum_{i=1}^r \mu_i (J\chi_i)\chi_i^*.
\end{equation}}
{Accordingly, }we define
\begin{align}\label{eq:def:LL}
    {L=K^{(2)}-\overline D}
\end{align}
and for $1\leq t\leq \ell-1$,
\begin{align}\label{eq:def_Sk_l}
    S_t^{(\ell)}=\Delta^{(t-1)}LB^{(\ell-t-1)}.
\end{align}
Then the following estimate holds.
\begin{lemma}\label{lem:Blwnorm}
  For any unit vector $w\in \mathbb C^{\vec H (V)}$,
  \begin{align}
      \|B^{(\ell)} w\|\leq &\| \Delta^{(\ell)}\| +\frac{1}{\binom{n}{q-1}}\|KB^{(\ell-1)}\|+\frac{q-1}{n-q+2} \sum_{j=1}^r  \sum_{t=1}^{\ell-1} \mu_j\| \Delta^{(t-1)}J \chi_j\|  |\langle   \chi_j, B^{(\ell-t-1)}w\rangle | \notag \\
      &+\frac{q-1}{n-q+2}\sum_{t=1}^{\ell-1} \| S_t^{(\ell)}\| +p_{\max}(q-1)\|\Delta^{(\ell-1)}\| +\frac{1}{\binom{n}{q-1}}\sum_{t=0}^{\ell} \|R_t^{(\ell)}\|. \notag 
  \end{align}
\end{lemma}
\begin{proof}
  It follows from Lemma \ref{lem:expansionBl}, decomposing $K^{(2)}$ into $L+{\overline D}$ given in \eqref{eq:def:LL}, and the fact that $\|K\|\leq \sqrt{\|K\|_{1} \|K\|_{\infty}}\leq  p_{\max}(q-1) \binom{n}{q-1}.$
\end{proof}

 {It remains to bound all terms present in Lemma \ref{lem:Blwnorm}. For most of them, this is done in the following proposition:}
 \begin{proposition}\label{prop:trace}
 Let $\chi$ be any vector among $\chi_1,\dots, \chi_r\in \mathbb C^{\vec{H}(V)}$.
  Let $\ell=\kappa \log_{p_{\max}(q-1)}(n)$ with $\kappa \in (0,1/6)$. For sufficiently large $n$, with  probability at least $1-n^{-c}$ for some constant $c>0$, the following norm bounds hold for all $1\leq k\leq \ell$:
  \begin{align}
      \| \Delta^{(k)}\| &\leq (\log n)^{10} [(q-1)d]^{k/2}, \label{eq:Delta}\\ 
      \| \Delta^{(k)}J \chi\| &\leq \sqrt n(\log n)^5 [(q-1)d]^{k/2},\label{eq:Delta_chi}  \\
      \|R_k^{(\ell)}\| &\leq \binom{n}{q-2}(\log n)^{25} n^{\kappa}[(q-1)p_{\max}]^{-k/2}, \label{eq:Rk}\\
      \| KB^{(k)}\| &\leq  \binom{n}{q-1}^{1/2} (\log n)^{10} [(q-1)p_{\max}]^{k}, \quad  \|B^{(k)}\|\leq (\log n)^{10} [(q-1)p_{\max}]^k. \label{eq:KBk}
  \end{align}
  And for all $1\leq k\leq \ell-1$, 
  \begin{align}
      \|S_{k}^{(\ell)}\|\leq  (\log n)^{30} n^{\kappa+1/2}[(q-1)p_{\max}]^{-k/2}.\label{eq:Sk}
  \end{align}
\end{proposition}
The proof of Proposition \ref{prop:trace} is based on the moment method generalized for random hypergraphs, and we present it in Appendix \ref{sec:trace_bound}.

\section{Local analysis}\label{sec:local}

In this section,  we will mainly modify the results in \cite{bordenave2018nonbacktracking, bordenave2020detection} to our hypergraph setting and characterize the local structure of sparse random hypergraphs. We construct a hypertree growth process in the following way. The hypertree is used to approximate the local neighborhoods in the random hypergraph $G$ generated by the HSBM.

\begin{definition}[Galton-Watson hypertree]\label{def:gw_tree}
We define the random Galton-Watson hypertree as follows:
\begin{itemize}
  \item Start from a \emph{root} $\rho$ with a given spin $\sigma(\rho)$;
  \item Generate $k = \Poi(d)$ hyperedges intersecting only at $\rho$, yielding $k(q-1)$ \emph{children} of $\rho$;
  \item For each hyperedge, we assume a fixed ordering of the $(q-1)$ associated children $v_1, \dots, v_{q-1}$, we denote the ordered children by $v=(v_1,\dots, v_{q-1})$. Then, we assign a type to each $v_i$ randomly such that
        \begin{align}\label{eq:prob_add}
        \Pb{\underline\sigma(v) = \underline j} = \frac 1 d \cdot  p_{\sigma(\rho), \underline j} \cdot \prod_{\ell \in \underline j} \pi_\ell.
        \end{align}
        The definition of $d$ in \eqref{eq:constant_degree} implies that the probabilities defined in \eqref{eq:prob_add} add up  to 1;
  \item Repeat the process for each child of $\rho$, treating as the root of an i.i.d Galton-Watson hypertree.
\end{itemize}
\end{definition}

{See Figure \ref{fig:GWH} for an illustration of a  Galton-Watson hypertree.}
A broader class of the growth process on hypertrees, including the Galton-Watson hypertree defined above, was described as ``broadcasting through hypergraphs" in \cite[Section 7]{mezard2006reconstruction}.
\begin{figure}
    \centering
\includegraphics[width=0.5\linewidth]{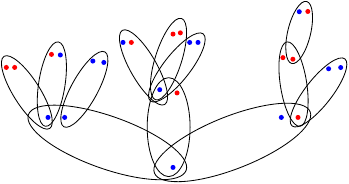}
    \caption{A Galton-Watson hypertree with $q=3$}
    \label{fig:GWH}
\end{figure}



\subsection{Growth property for Galton-Watson hypertrees}

Let $S_t$ be the size of the $t$-th generation in the Galton-Watson hypertree branching process.
\begin{lemma}\label{lem:GWupperbound}
 Assume $S_0=1$. There exist constants $c_1, c_2>0$ such that for all $s\geq 0$,
 \begin{align}\label{eq:GWupperbound}
     \mathbb P \left( \forall t\geq 1, S_t\leq s [(q-1)p_{\max}]^t\right)\geq  1-c_2e^{-c_1s} .
 \end{align}

 Moreover, there exists a universal constant $c_3>0$ such that for every $p\geq 1$,
 \begin{align}\label{eq:Stpmoment}
 \mathbb E \max_{ t\geq 0} \left( \frac{S_t}{[(q-1)p_{\max}]^t}\right)^p&\leq (c_3p)^p.
 \end{align}
 We also have
 \begin{align}
    (\mathbb E |(T,x)_t |^p)^{1/p}&\leq 2c_3p [(q-1)p_{\max}]^t.\notag
\end{align}

\end{lemma}
\begin{proof}
  For $t\geq 1$, set $\epsilon_t=[(q-1)p_{\max}]^{-t/2}\sqrt{t}$, $f_t=\prod_{l=1}^t (1+\epsilon_l)$. Let $Y_i$ be i.i.d. $\textnormal{Poi}(p_{\max})$ random variables. There exists constants $c_0,c_1>0$ such that
  \[ c_0\leq f_t\leq c_1, \quad \epsilon_t\leq c_1.
  \]By Chernoff bound, for any integer $\ell\geq 1$, and real number $s>1$,
  \begin{align*}
      \mathbb P \left(  \sum_{i=1}^{\ell} Y_i \geq \ell p_{\max}s\right)\leq \exp(-\ell p_{\max} h(s)),
  \end{align*}
  where $h(s)=s\log s-s+1$.    On the event that $\{S_k\leq sf_k[(q-1)p_{\max}]^k\}$, we have for any $s>c_0^{-1}$,
  \begin{align*}
   &\mathbb P \left( S_{k+1}\geq sf_{k+1}[(q-1)p_{\max}]^{k+1}\mid \{S_k\leq sf_k[(q-1)p_{\max}]^k\} \right)  \\
   &\leq\mathbb P\left( (q-1) \sum_{i=1}^{sf_k[(q-1)p_{\max}]^k} Y_i\geq  sf_{k+1}[(q-1)p_{\max}]^{k+1}\right)\\
   &\leq \exp(-sp_{\max}f_k[(q-1)p_{\max}]^kh(1+\epsilon_{k+1}))\leq \exp(-c_0'(k+1)s),
  \end{align*}
for some constant $c_0'>0$,  where we use the inequality that for $x\in [0,c_1], h(1+x)\geq \theta x^2$ for some constant $\theta>0$.
 Hence for any $s>\max\{(c_0')^{-1},c_0^{-1}\}$,
  \begin{align}
      \mathbb P \left ( \exists t\geq 1, S_t>sc_1[(q-1)p_{\max}]^t  \right)\leq \sum_{\ell=1}^{\infty}\exp(-c_0'\ell s)\leq \frac{e^{-c_0's}}{1-e^{-c_0's}}\leq c_1'e^{-c_0's}\notag
  \end{align}
  for some constant $c_1'>0$, this implies \eqref{eq:GWupperbound}. For $s\leq  \max\{(c_0')^{-1},c_0^{-1}\}$, by taking $c_2$ sufficiently large, \eqref{eq:GWupperbound} also holds.

  The second claim follows by using $\mathbb E |X|^p=p\int_{0}^{\infty} x^{p-1} \mathbb P(|X|\geq x) dx$.
 Since
\begin{align}
    |(T,x)_t|=1+S_1+\cdots S_{t}\leq 2\max_{k\geq 0}([(q-1)p_{\max}]^{-k}S_k) [(q-1)p_{\max}]^t,\notag
\end{align}
taking expectation, we obtain from  \eqref{eq:Stpmoment} that
\begin{align}\label{eq:Etree}
    (\mathbb E |(T,x)_t |^p)^{1/p}\leq 2c_3p [(q-1)p_{\max}]^t.
\end{align}

\end{proof}

\subsection{Growth property for  random hypergraphs}

\begin{definition}[Exploration process]\label{def:exploration_process}
Denote $S_t(v)$ the set of vertices at distance $t$ from $v$. We consider the exploration process of the neighborhood of $v$ which starts with $A_0=\{v\}$ and at stage $t\geq 0$, if $A_t$ is not empty, take a vertex in $A_t$ at minimal distance from $v$, denoted by $v_t$, reveal its neighborhood $N_{t+1}$ in $[n]\setminus A_t$, and update $A_{t+1}= (A_t\cup N_{t+1})\setminus \{v_t\}$. Denote $\mathcal F_t$ the filtration generated by $(A_0,\dots, A_t)$. The set of discovered vertices at time $t$ is denoted by $D_t=\cup_{0\leq s\leq t}A_s$.
\end{definition}

\begin{lemma}\label{lem:growthhypergraph}
 There exist $c_0,c_1>0$ such that for all $s\geq 0$ and for any $v\in [n]$,
 \begin{align}\label{eq:ERupperbound}
     \mathbb P(\forall t\geq 0, S_t(v)\leq s[(q-1)p_{\max}]^t)\geq 1-c_1e^{-c_0s}.
 \end{align}
 Consequently, for any $p\geq 1$, there exists $c_3>0$ such that
 \begin{align}
      \mathbb E \max_{ t\geq 0} \left( \frac{S_t(v)}{[(q-1)p_{\max}]^t}\right)^p&\leq (c_3p)^p,
     \label{eq:ERupperboundconsequence1} \\
     \mathbb E \max_{v\in [n], t\geq 0} \left( \frac{S_t(v)}{[(q-1)p_{\max}]^t}\right)^p&\leq (c_3\log n)^p+(c_3p)^p,\label{eq:ERupperboundconsequence}\\
    \mathbb E [|(G,x)_t|^p]^{1/p}&\leq 2c_3p [(q-1)p_{\max}]^t,\\
     \label{eq:neighbourhood_max_expectation}\mathbb E [\max_{x\in [n]}|(G,x)_t|^p]^{1/p}&\leq 2c_3(\log n+p) [(q-1)p_{\max}]^t.
 \end{align}
\end{lemma}
\begin{proof}
  For the first statement, consider the exploration process. Given $\mathcal F_t$, the number of neighbors of $v_t$ in $[n]\setminus D_t$ is stochastically dominated by $(q-1)V$, where
  \begin{align}
      V\sim \text{Bin}\left(\binom{n}{q-1}, \frac{p_{\max}}{\binom{n}{q-1}}\right).\notag
  \end{align}
  
 By applying Chernoff inequality for Binomial variables and repeating the proof of Lemma \ref{lem:GWupperbound}, \eqref{eq:ERupperbound} holds. \eqref{eq:ERupperboundconsequence1}  follows by using $\mathbb E |X|^p=p\int_{0}^{\infty} x^{p-1} \mathbb P(|X|\geq x) dx$. 
  \eqref{eq:ERupperboundconsequence} follows from the inequality that 
\begin{align*}
     \mathbb E \max_{v\in [n], t\geq 0} \left( \frac{S_t(v)}{[(q-1)p_{\max}]^t}\right)^p &\leq \int_{0}^{\infty} 
      s^{p-1} \left( 1\wedge nc_1 e^{-c_0s}\right) ds\\
      &= \int_{0}^{\frac{1}{c_0}\log(c_1n)} s^{p-1} ds +\int_{\frac{1}{c_0}\log(c_1n)}^{\infty} s^{p-1} c_1 ne^{-c_0s} ds.
\end{align*}
Similarly to \eqref{eq:Etree}, we get the following estimates for $(G,x)_t$:
\begin{align}
    \mathbb E [|(G,x)_t|^p]^{1/p}&\leq 2c_3p [(q-1)p_{\max}]^t\notag\\
    \mathbb E [\max_{x\in [n]}|(G,x)_t|^p]^{1/p}&\leq 2c_3(\log n+p) [(q-1)p_{\max}]^t.\notag
\end{align}
\end{proof}

In the remaining part of this section, we take $\ell=\kappa \log_{(q-1)p_{\max}} n.$ for some constant $\kappa$. Recall the tangle-free property from Definition \ref{def:tangle_free}.

\begin{lemma}[tangle-freeness]\label{lem:tangle-free}
  For some constants $c_1,c_2$, the following holds:
  \begin{enumerate}
      \item {$G$} is $\ell$-tangle-free with probability at least $1-\log^4(n)c_1n^{4\kappa-1}.$
      \item The probability that a given vertex $v$ has a cycle in its $\ell$-neighborhood is at most 
      $ c_2\log^2(n) n^{2\kappa-1}.$
  \end{enumerate}
\end{lemma}
\begin{proof}

  Let $\tau$ be the time at which all vertices of $(G,v)_{\ell}$ have been revealed. In the exploration process, we only discover hyperedges that bring in $(q-1)$ new vertices at each step. Given $\mathcal F_{\tau}$, the number of undiscovered hyperedges containing at least 2 vertices in $(G,v)_{\ell}$ is stochastically dominated my $\textnormal{Bin}\left(m,\frac{p_{\max}}{\binom{n}{q-1}}\right)$ with $m=|(G,v)_{\ell}|^2\binom{n}{q-2}$. Then, by Markov inequality,
  \begin{align}
      \mathbb P\left( (G,v)_{\ell} \text{ is not a hypertree}\right)\leq \frac{p_{\max}}{\binom{n}{q-1}}\mathbb E|(G,v)_{\ell}|^2\binom{n}{q-2}\leq \frac{c_2[(q-1)p_{\max}]^{2\ell}\log^2 n}{n},\notag
  \end{align}
  which proves the second claim.

  There are two possibilities that a given vertex $v$ has at least 2 cycles in the $\ell$-neighborhood:
  (1) there are at least two undiscovered hyperedges;
  (2) there is at least one undiscovered hyperedge containing 3 vertices in $(G,v)_{\ell}$.
  Since for binomial random variables, the probability that $\text{Bin}(m,p)\not\in \{0,1\}$ is bounded by $p^2m^2$, the first case happens with probability at most \[\frac{p_{\max}^2}{\binom{n}{q-1}^2} \binom{n}{q-2}^2 \mathbb E|(G,v)_{\ell}|^4\leq \frac{c_1[(q-1)p_{\max}]^{4\ell}\log^4 n}{n^2},\]
  and the second case happens with probability at most
  \[\frac{p_{\max}}{\binom{n}{q-1}} \binom{n}{q-3} \mathbb E|(G,v)_{\ell}|^3\leq \frac{c_2[(q-1)p_{\max}]^{3\ell}\log^3 n}{n^2},\]
   Taking a union bound over $v\in [n]$ implies that the first claim holds. 
\end{proof}

\subsection{Coupling between random hypertrees and  hypergraphs}

Using Lemma \ref{lem:tangle-free}, we are now equipped to show one of the main intermediate results: bounding the total variation distance between the $2\ell$-neighbourhoods of $(G, x)$ and $(T, x)$. We recall that this distance $d_\mathrm{var}$ is defined as
\[ d_\mathrm{var}(\mathbb P_1, \mathbb P_2) = \min_{X_1, X_2} \mathbb P(X_1 \neq X_2), \]
where the min ranges over all couplings $(X_1, X_2)$ such that $X_1 \sim \mathbb P_1$ and $X_2 \sim \mathbb P_2$.

The probability space we are interested in is the space $\mathcal G_{*}$  of \emph{rooted} labeled $q$-uniform hypergraphs, i.e., labeled hypergraphs with a distinguished vertex. Denoting by $\mathcal L(X)$ the law of a random variable $X$, we aim to show the following:

\begin{proposition}\label{prop:coupling}
  Let $\ell = \kappa \log_{(q-1)p_{\max}}(n)$ for some constant $\kappa > 0$. Then, for every $x \in V$, we have
  \[ d_\mathrm{var}(\mathcal L((G, x)_\ell), \mathcal L((T, x)_\ell)) \leq c_0  p_{\max} \log(n)\,  n^{2\kappa - 1}.\]
\end{proposition}

 \begin{proof}
  Let $E_\ell$ be the event under which $|(G, x)|_\ell \leq c_1 \log(n) [(q-1)p_{\max}]^\ell$; using Lemma \ref{lem:growthhypergraph}, we can choose $c_1$ such that \[\mathbb P(E_\ell) \geq 1 - \frac{[(q-1)p_{\max}]^\ell}n\geq 1-n^{\kappa-1}. \]
  
   We consider the exploration process of Definition \ref{def:exploration_process}; the goal is to couple the versions of this process on the hypergraph and the Galton-Watson hypertree. Assume that the coupling holds until time $t$, and let $i$ be the label of the vertex $v_t$. There are two ways that the coupling can fail at time $t$:
   
   \begin{enumerate}
       \item the graph spanned by the revealed hyperedges of $G$ at $t+1$ is not a hypertree,
       \item for some $\underline j \in [r]^{q-1}$, the coupling between the edges of type $(i, \underline j)$ adjacent to $v_t$ fails.
   \end{enumerate}
   
   Since condition (i) implies that $(G, x)_\ell$ is not a hypertree, the overall failure probability of this event is upper bounded using Lemma \ref{lem:tangle-free}. Our goal is thus to exhibit a coupling between the edges of $(G, x)$ and $(T, x)$ adjacent to $v_t$, or equivalently to bound the total variation distance between the two. Let
   \begin{equation}\label{eq:def_tau}
   \underline \tau = (\tau_1, \dots, \tau_r) \quad \text{with} \quad \tau_1 + \dots + \tau_r = q-1,
   \end{equation}
   and define $\mathbb P_{\underline \tau}$ (resp. $\mathbb Q_{\underline \tau}$) the distribution of the number of edges with $\tau_k$ vertices of type $k$ for all $k$, adjacent to $v_t$ in $G$ (resp. $T$). From Poisson thinning and the definition of $G$ and $T$, we find that
   \begin{align*}
       \mathbb P_{\underline \tau} &= \mathrm{Bin}\left(\prod_{k=1}^r \dbinom{n_k(t)}{\tau_k}, \frac{p_{i, \underline j}}{\dbinom{n}{q-1}} \right) \\
       \mathbb Q_{\underline \tau} &= \mathrm{Poi}\left(p_{i, \underline j} \cdot \dbinom{q-1}{\tau_1, \dots, \tau_k} \prod_{k=1}^r \pi_k^{\tau_k}\right),
   \end{align*}
   where $n_k(t)$ is the number of vertices of type $k$ in $[n] \setminus D_t$, and $\underline j$ is any $q-1$-tuple of $[r]$ such that $k$ is repeated $\tau_k$ times. The multinomial coefficient in $\mathbb Q_{\underline \tau}$ counts the number of permutations of $\underline j$ that yield the same hyperedge type. Bounding the total variation distance between those two distributions is a tedious but straightforward computation using classical Poisson approximation bounds:
   
   \begin{lemma}\label{lem:tvar_poisson_binomial}
    For any $\tau$ satisfying \eqref{eq:def_tau}, we have
    \[ d_\mathrm{var}(\mathbb P_{\underline \tau}, \mathbb Q_{\underline \tau}) \leq c_3   p_{\max} r(q-1)! n^{\kappa - 1}. \]
   \end{lemma}
   \noindent The proof of Lemma \ref{lem:tvar_poisson_binomial} is deferred to Section \ref{sec:Misc}.      Now, there are at most 
   \[\dbinom{r+q - 2}{q-1} \leq \frac{(r+q-2)^{q-1}}{(q-1)!} \]
   tuples $\tau$ satisfying \eqref{eq:def_tau}, and the coupling has to hold for all $t \leq c_1 \log(n) [(q-1)p_{\max}]^\ell$ under $E_\ell$. As a result, we find
   \begin{align*} 
   d_\mathrm{var}(\mathcal L((G, x)_\ell), \mathcal L((T, x)_\ell)) &\leq c_3 p_{\max}r(q-1)! n^{\kappa-1} \cdot \frac{(r+q-2)^{q-1}}{(q-1)!} \cdot c_1 \log(n) [(q-1)p_{\max}]^\ell \\
   &\leq c_4 p_{\max} \log(n) n^{2\kappa -1}.
   \end{align*}
   
 \end{proof}

\section{Functionals on Galton-Watson hypertrees}\label{sec:martingale}\label{sec:functional}

The goal of this section is to study specific functionals on random hypertrees, which will be the backbone of our pseudo-eigenvectors for $B$. We define, for a vector $\xi \in \dR^r$ and $t \geq 0$, the functional
\begin{equation}\label{eq:def_tree_functional}
   f_{\xi, t}(g, o) = \sum_{x_t \in \partial(g, o)_t} \xi_{\sigma(x_t)}.
\end{equation}
We establish estimates for the expectation, variance, and covariance of such functionals with two different points of view, which will both be useful in the later parts of the proof. An important quantity is the three-dimensional equivalent of $Q$, denoted $Q^{(3)}$:
\[ Q^{(3)}_{ijk} = \pi_j \pi_k \sum_{\underline \ell \in [r]^{q-3}} p_{ijk, \underline \ell} \prod_{m \in \underline \ell} \pi_m. \]
This is technically only defined for $q > 2$, but will always be prefaced by a factor of $(q-2)$ so this problem can be safely ignored.

\subsection{Martingale approach}
Let $(T, \rho)$ be a Galton-Watson hypertree, and $\mathcal F = (\mathcal F_t)_{t \geq 0}$ the filtration adapted to the random variables $(T, \rho)_t$.
Our first proposition is a martingale equation on the process $f_{\xi, t}$:
\begin{proposition}\label{prop:martingale_equation}
  Let $(\mu, \phi)$ an eigenpair of $Q$, and consider the random process
  \[ Z_t = [(q-1)\mu]^{-t}f_{\phi, t}(T, \rho). \]
  The process $(Z_t)_{t\geq 0}$ is an $\mathcal F$-martingale, with common expectation $\phi_{\sigma(\rho)}$.
\end{proposition}

 Throughout this section, we will denote by $\dE_t$ the conditional expectation $\E*{\cdot \given \mathcal F_t}$. We first reproduce a lemma from Lemma 13 in \cite{stephan2020non} about Poisson sums of random variables:
\begin{lemma}\label{lem:poisson_sums}
    Let \( N \) be a \( \Poi(d) \) random variable, and \( (X_i) \), \( (Y_i) \) two iid sequences of random variables, independent from \( N \), such that \( X_i \) and \( Y_j \) are independent whenever \( i \neq j \). Denote by \( A, B \) the random variables
    \[ A = \sum_{i=1}^N X_i \quand B = \sum_{i=1}^N Y_i. \]
Then the following identities hold:
    \begin{align}
       & \E A = d\cdot \E {X_1}, \quad \E B = d \cdot \E {Y_1}, \label{eq:poisson_sum_expectation}  \\
       & \mathrm{Cov}(A, B) = d\cdot \E {X_1Y_1}. \label{eq:poisson_sum_correlation}
    \end{align}
\end{lemma}

We will also need a result regarding specific functionals on $T$:

\begin{lemma}\label{lem:one_step_propagation}
 Let $x_t \in \partial (T, \rho)_t$, with $\sigma(x_t) = i$, and $\xi \in \dR^r$. Then
 \[ \dE_t\left[ \sum_{x_{t+1}: x_t \to x_{t+1}} \xi_{\sigma(x_{t+1})} \right] = (q-1)[Q\xi](i), \]
 where $x_{t}\to x_{t+1}$ means $x_{t+1}$ is a child of $x_{t}$ in the Galton-Watson hypertree.
\end{lemma}

\begin{proof}
    Let $e_1, \dots, e_N$ be the hyperedges adjacent to $x_t$, with $N \sim \Poi(d)$. For $1\leq k\leq N$, let
   \begin{equation} \label{eq:def_xk}
       X_k = \sum_{y \in e_k, y \neq x_t} \xi_{\sigma(y)}.
   \end{equation}
   It is clear the $X_i$ are i.i.d, so that using \eqref{eq:poisson_sum_expectation},
   \begin{align*} 
   \dE_t\left[ \sum_{x_{t+1}:x_t \to x_{t+1}} \xi_{\sigma(x_{t+1})} \right] &=\dE_t \left[\sum_{k=1}^N X_k\right]= d \dE[X_1] \\
   &= d \cdot \frac 1d \sum_{\underline j \in [r]^{q-1}} p_{i, \underline j} \prod_{\ell \in \underline j} \pi_\ell \sum_{\ell \in \underline j} \xi_\ell = \sum_{k =1}^{q-1} \sum_{\underline j \in [r]^{q-1}} p_{i, \underline j} \left(\prod_{\ell \in \underline j} \pi_\ell \right) \xi_{j_k}.
   \end{align*}
   Here we denote $\underline j=(j_1,\dots,j_{q-1})$.
 All terms in the last sum except for $\xi_{j_k}$ are invariant with respect to permutations of $\underline j$. As a result,
   \begin{align*}
       \dE_t\left[ \sum_{x_{t+1}:x_t \to x_{t+1}} \xi_{\sigma(x_{t+1})} \right] &= (q-1) \sum_{\underline j \in [r]^{q-1}} p_{i, \underline j} \left(\prod_{\ell \in \underline j} \pi_\ell \right) \xi_{j_1} \\
       &= (q-1) \sum_{j \in [r]} \, \left[\sum_{ \underline k \in [r]^{q-2}} p_{ij, \underline k} \prod_{\ell \in \underline k} \pi_\ell \right] \cdot \pi_j \xi_j= (q-1)[Q\xi](i).
   \end{align*}
\end{proof}

Now we are ready to finish the proof of Proposition \ref{prop:martingale_equation}.
\begin{proof}[Proof of Proposition \ref{prop:martingale_equation}]
   Let $t \geq 0$ be fixed. We have
   \[ Z_{t+1} - Z_t = [(q-1)\mu]^{-(t+1)}  \sum_{x_t \in \partial(T, \rho)_t}\left[ \sum_{x_t \to x_{t+1}} \phi_{\sigma(x_{t+1})} - (q-1)\mu \phi_{\sigma(x_t)} \right]. \]
    Lemma \ref{lem:one_step_propagation} implies that for any $x_t \in \partial (T, \rho)_t$, 
    \[  \dE_t\left[\sum_{x_t \to x_{t+1}} \phi_{\sigma(x_{t+1})}\right] = (q-1)[Q\phi](\sigma(x_t)) = (q-1)\mu \phi_{\sigma(x_t)}, \]
  where the last equation is from the eigenvector equation for $\phi$ and $\mu$. As a result, each term in the sum above has a zero expectation; therefore, $(Z_t)_{t\geq 0}$ is a martingale. The common expectation is  $Z_0 = \phi_{\sigma(\rho)}$.
\end{proof}

Our next results concern the correlation of two such martingales, which reduces to the study of their increments.

\begin{proposition}\label{prop:martingale_correlation}
  Let $(\mu, \phi)$ and $(\mu', \phi')$ be two (not necessarily distinct) eigenpairs of $Q$, and consider the martingales $(Z_t)_{t \geq 0}$ and $(Z'_t)_{t \geq 0}$ as in Proposition \ref{prop:martingale_equation}. Define the vector
  \begin{equation}\label{eq:def_y_phi_phip}
      y^{(\phi, \phi')} = Q(\phi \circ \phi') + (q-2) Q^{(3)}(\phi \otimes \phi'),
  \end{equation}
 where
 \[ (\phi\circ \phi')(i)= \phi_i\phi'_i,\quad   (\phi \otimes \phi')(i, j) = \phi_i\phi'_j,\quad [Q^{(3)}(\phi \otimes \phi')](i)=\sum_{j,k\in [r]}Q^{(3)}_{ijk} \phi_j\phi_k'.
 \]
  Then,
  \begin{equation} \label{eq:martingale_increments_correlation}
    \E*{(Z_{t+1} - Z_t)(Z'_{t+1} - Z'_t)} = \left[(q-1) \mu \mu'\right]^{-(t+1)} [Q^t y^{(\phi, \phi')}](\sigma(\rho)).
  \end{equation}
  As a result, the martingale $(Z_t)_{t \geq 0}$ converges in $\mathcal L^2$ whenever $(q-1)\mu^2 > d$. Further, for any $0 \leq t \leq t'$,
  \begin{align} \label{eq:EZtZt'}
  \dE[Z_t Z'_{t'}] = \phi_{\sigma(\rho)} \phi'_{\sigma(\rho)} + \sum_{s=0}^{t-1} \frac{\left[ Q^s y^{(\phi, \phi')} \right](\sigma(\rho))}{[(q-1)\mu \mu']^{s+1}}. 
  \end{align}
\end{proposition}

\begin{proof}
  We denote by $\Delta_t = Z_{t+1} - Z_t$ (resp. $\Delta'_t = Z'_{t+1} - Z'_t$) the martingale increments.  We first compute the conditional expectation $\dE_t[\Delta_t \Delta'_t]$:
  \begin{equation*}
      \dE_t[\Delta_t \Delta'_t] =[(q-1)^2 \mu \mu']^{-(t+1)}\sum_{x_t, x'_t \in \partial (T, \rho)_t} E(x_t, x_t'),
  \end{equation*}
  where the correlation term $E(x_t, x'_t)$ is given by
  \[ E(x_t, x'_t) = \dE \left[\left(\sum_{x_t \to x_{t+1}} \phi_{\sigma(x_{t+1})} - (q-1)\mu \phi_{\sigma(x_t)} \right) \left(\sum_{x'_t \to x'_{t+1}} \phi'_{\sigma(x'_{t+1})} - (q-1)\mu' \phi'_{\sigma(x'_t)} \right) \right] .\]
  By the independence property of the Galton-Watson hypertree, $E(x_t, x'_t)$ is zero except if $x_t = x'_t$. Consider a fixed $x_t$ in $(T, \rho)_t$, with type $\sigma(x_t) = i$ and adjacent edges $e_1, \dots, e_N$. Defining $X_k$ and $X'_k$ as in \eqref{eq:def_xk}, we can use \eqref{eq:poisson_sum_correlation} to find
  \begin{equation}\label{eq:e_x_x}
      E(x_t, x_t) = d \dE[X_1 X'_1],
  \end{equation} 
  so we only have to compute this last expectation. Expanding the product, we have
  \begin{align*}
      X_1X'_1 &= \left( \sum_{y \in e_1, y \neq x_t} \phi_{\sigma(y)} \right)\left( \sum_{y' \in e_1, y' \neq x_t} \phi'_{\sigma(y')} \right) \\
      &= \sum_{y \neq x_t} \phi_{\sigma(y)} \phi'_{\sigma(y)} + \sum_{y\neq y'\neq x_t} \phi_{\sigma(y)} \phi'_{\sigma(y')}.
  \end{align*}
  
  Using Lemma \ref{lem:one_step_propagation}, the expectation of the first sum is easily computed:
  \[ d \dE_t\left[ \sum_{y \neq x_t} \phi_{\sigma(y)} \phi'_{\sigma(y)} \right] = (q-1)[Q(\phi \circ \phi')](\sigma(x_t)). \]
  For the second one, we use the same method as in the proof of Lemma \ref{lem:one_step_propagation}. We have, letting $\sigma(x_t) = i$,
  \begin{align*}
      d\dE_t\left[\sum_{y\neq y'\neq x_t} \phi_{\sigma(y)} \phi'_{\sigma(y')} \right] &= d \cdot \frac 1d \sum_{\underline j \in [r]^{q-1}} p_{i, \underline j} \prod_{\ell \in \underline j} \pi_\ell \sum_{\ell \neq \ell' \in \underline j} \phi_\ell \phi'_{\ell'} \\
      &= \sum_{k \neq k'} \sum_{\underline j \in [r]^{q-1}} p_{i, \underline j} \left(\prod_{\ell \in \underline j} \pi_\ell \right) \phi_{j_k} \phi_{j_{k'}}.
  \end{align*}
  Again, each term except for the last two is invariant with respect to permutations of $\underline j$. Hence,
  \begin{align*}
      d\dE_t\left[\sum_{y\neq y'\neq x_t} \phi_{\sigma(y)} \phi'_{\sigma(y')} \right] &= (q-1)(q-2)\sum_{\underline j \in [r]^{q-1}} p_{i, \underline j} \left(\prod_{\ell \in \underline j} \pi_\ell \right) \phi_{j_1} \phi_{j_{2}} \\
      &= (q-1)(q-2) \sum_{j, k \in [r]} \sum_{\underline \ell \in [r]^{q-3}} p_{ijk, \underline \ell} \left(\prod_{m \in \underline \ell} \pi_m \right) \cdot \pi_j \pi_k \phi_j \phi_k \\
      &= (q-1)(q-2)\left[ Q^{(3)}(\phi \otimes \phi') \right](i).
  \end{align*}
  
  Putting all this together,
  \begin{equation}\label{eq:e_x_x_2}
      \begin{split}
          E(x_t, x_t) &= (q-1)[Q(\phi \circ \phi')](\sigma(x_t)) + (q-1)(q-2)\left[ Q^{(3)}(\phi \otimes \phi') \right](\sigma(x_t))\\
          &= (q-1) y^{(\phi, \phi')}(\sigma(x_t)).
      \end{split} 
    \end{equation}

  Now, since $\dE[\Delta_t \Delta'_t] = \dE_0[\dE_1[ \dots \dE_t[\Delta_t \Delta'_t]\dots]]$, repeated applications of Lemma \ref{lem:one_step_propagation} imply \eqref{eq:martingale_increments_correlation}.
  
  For the second part, since both $Z_t$ and $Z_t'$ are $\mathcal F$-martingales, if $t\leq t'$,
  \begin{align*} 
  \E*{Z_t Z'_{t'}} &= Z_0Z'_0 + \sum_{s = 0}^{t-1} \E*{\Delta_t \Delta'_t} = \phi_{\sigma(\rho)} \phi'_{\sigma(\rho)} + \sum_{s=0}^{t-1} \frac{\left[ Q^s y^{(\phi, \phi')} \right](\sigma(\rho))}{[(q-1)\mu \mu']^s}.
  \end{align*}
  By the Doob's second martingale convergence theorem, $(Z_t)_{t \geq 0}$ converges in $\mathcal L^2$ whenever $\dE[Z_t^2]$ is uniformly bounded. Since the spectral radius of $Q$ is $d$, this happens in particular when $(q-1) \mu^2 > d$.
\end{proof}

\subsection{A top-down approach: Galton-Watson transforms}

We now introduce another point of view on hypertree functionals, which recovers the same result as the previous section, but without being able to prove the martingale property in  Proposition \ref{prop:martingale_correlation}.

\paragraph{Galton-Watson transforms} From Definition \ref{def:gw_tree}, it is clear that the law of a Galton-Watson tree $(T, \rho)$ only depends on $\sigma(\rho)$. For a functional $f: \cG_* \to \dR$, we can therefore define its \emph{Galton-Watson transform} $\overline f: [r] \to \dR$ via
\begin{equation}\label{eq:GW_transform_definition}
    \overline f(k) = \dE_{\sigma(\rho) = k}[f(T, \rho)].
\end{equation}
Although this is not a reversible transformation, for a large class of functionals, $\overline f$ encapsulates all the relevant properties of $f$. The results proved in the above section lead to the following proposition:

\begin{proposition}\label{prop:gw_transform_expectation}
 Let $(\mu, \phi)$ and $(\mu', \phi')$ two eigenpairs of $Q$, and define $f_{\xi, t}=f_{\xi,t}(T,\rho)$ as in \eqref{eq:def_tree_functional}. Then
 \begin{align}
     \overline{f_{\phi, t}} &= [(q-1)\mu]^t \phi, \label{eq:fphit} \\
     \overline{f_{\phi, t}f_{\phi', t}} &= [(q-1)^2\mu\mu']^t \left( \phi \circ \phi' + \sum_{s=0}^{t-1} \frac{Q^s y^{(\phi, \phi')}}{[(q-1)\mu \mu']^{s+1}} \right),\label{eq:ffffphit}
 \end{align}
 and if $F_{\phi, t} = (f_{\phi, t+1} - (q-1)\mu f_{\phi, t})^2$, then  
 \begin{align}\label{eq:Fphit} 
 \overline{F_{\phi, t}} = (q-1)^{t+1}Q^t y^{(\phi, \phi)}.
 \end{align}
\end{proposition}
\begin{proof}
  For the first claim, from the martingale property of $Z_t$ in Proposition \ref{prop:martingale_equation},
  \begin{align*}
      \overline{f_{\phi, t}}(k)&=\dE_{\sigma(\rho)=k}[f_{\phi,t} (T,\rho)]=\dE_{\sigma(\rho)}\left[ \sum_{x_t\in \alpha (T,\rho)_t} \phi_{\sigma(x_t)}\right]=\dE_{\sigma(\rho)=k}\left[ [(q-1)\mu]^tZ_t\right]\\
      &=[(q-1)\mu]^{t+1}\dE_{\sigma(\rho)=k}\left[ Z_0\right]=[(q-1)\mu]^t\phi(k).
  \end{align*}
  The second claim follows from \eqref{eq:EZtZt'}.
  By definition,
  \[ F_{\phi,t}=[(q-1)\mu]^{2t+2}(Z_{t+1}-Z_t)^2.
  \]
  and the third claim is derived from \eqref{eq:martingale_increments_correlation} by taking $\phi=\phi'$.
\end{proof}
\paragraph{A simple transformation} Most properties of graph functionals on $G$ can be understood in terms of the following transformation $\partial$, defined for any functional $f: \cG_* \to \dR$:
\begin{equation}\label{eq:def_partial_functional}
    \partial f (g, o) = \sum_{o'\in \mathcal N(o)} f(g \setminus \Set{o} ,o'),
\end{equation}
where the sum ranges over all neighbors $o'$ of $o$ denoted by $\mathcal N (o)$. For a hypertree $(T, \rho)$, this operation corresponds to
\[ \partial f(T, \rho) = \sum_{i\in \mathcal N(\rho)} f(T_i, \rho_i) \]
where $(T_i, \rho_i)$ are the subtrees attached to the root $\rho$. The Galton-Watson transforms of $f$ and $\partial f$ are related as follows:
\begin{lemma}\label{lem:gw_transform_recurrence}
 We have, for any functional $f: \mathcal G_{*} \to \mathbb R$,
$ \overline{\partial f}= (q-1) Q \overline f.
 $
\end{lemma}
\begin{proof}
  This lemma is an easy consequence of Lemma \ref{lem:one_step_propagation} for $t = 0$.
\end{proof}

Lemma \ref{lem:gw_transform_recurrence} actually allows us to recover all the results of Proposition \ref{prop:gw_transform_expectation} using the recurrence formula
$ f_{\xi, t} = \partial f_{\xi, t-1}, $
which holds for all $\xi$ and $t \geq 1$. However, it is not sufficient to obtain the martingale convergence property in Proposition \ref{prop:martingale_correlation}.

\section{Spatial averaging of hypergraph functionals}\label{sec:spatial_avg}

We say a function $f$ from $\mathcal G_*$ to $\mathbb R$ is $t$-local if $f(G,o)$ is only a function of $(G,o)_t$. We first provide a moment inequality for hypergraph functionals using the Efron-Stein inequality. This following lemma is adapted from Proposition 12.3 in \cite{bordenave2020detection}.

\begin{lemma}\label{lem:Effron}
 Let $f,\psi: \mathcal G_* \to R$ be two $t$-local functions such that $|f(g,o)|\leq \psi(g,o)$ for all $(g,o)\in \mathcal G_*$ and $\psi$ is non-decreasing by the addition of hyperedges. Then there exists a universal constant $c_1>0$ such that for all $p\geq 2$,
 \begin{align}
   \left[  \mathbb E\left( \sum_{o\in [n]} f(G,o)-\mathbb E f(G,o)\right)^p\right]^{\frac{1}{p}}&\leq c_1\sqrt{n} p^{3/2} [(q-1)p_{\max}]^{t}\left(\mathbb E\left[\max_{o\in [n]}\psi(G,o)^{2p}\right]\right)^{\frac{1}{2p}},\label{eq:spatial_avg}
 \end{align}
\end{lemma}
\begin{proof}
  For $x\in [n]$, denote 
  \begin{align}
      H_x=\{\{y_1,\dots, y_{q-1}, x\}\in H \mid y_i\leq x, \text{ for all } i\in [q-1]  \}.\notag
  \end{align}
  Then the vector $(H_1,\dots, H_n)$ has independent entries, and  there is a measurable function $F$ such that
  \begin{align}
      Y:=\sum_{o\in [n]} f(G,o)=F(H_1,\dots, H_n).\notag
  \end{align}
  Define  $G_x$ the hypergraph with vertex set $V$ and hyperedge set $\cup_{y\not=x} H_y$, and
  \begin{align*}
      Y_x=\sum_{o\in V} f(G_{x},o)=F(H_1,\dots, H_{x-1},\emptyset, H_{x+1},\dots, H_n ).
  \end{align*}
  Then $Y_x$ is $\cup_{y\not=x} H_y$-measurable. By Efron-Stein inequality (see Theorem 15.5 in \cite{boucheron2013concentration}), for all $p\geq 2$,
 \begin{align}\label{eq:EF}
    ( \dE|Y-\E Y|^p) \leq (c_1\sqrt{p})^p\mathbb E  \left[\left(\sum_{x\in [n]}(Y-Y_x)^2\right)^{p/2}\right].
 \end{align}
 Since $f$ is $t$-local, for a given $x\in V$, the difference $f(G,o)-f(G_x,o)$ is always zero except possibly for $x\in (G,o)_t$. Consequently,
 \begin{align*}
     |Y-Y_x|&\leq \sum_{o\in (G,x)_t} |f(G,o)-f(G_x,o)|\\
      &\leq \sum_{o\in (G,x)_t}\psi(G,o)+\psi (G_x,o)\leq 2|(G,x)_t|\cdot \max_{o\in V}\psi(G,o),
 \end{align*}
where we used the assumption that  $\psi$ is non-decreasing by adding hyperedges. By H\H{o}lder's inequality 
\[ \left(\sum_{i=1}^n x_i^2\right)^{p/2}\leq n^{p/2-1} \sum_{i=1}^n |x_i|^p ,
\]
 we find 
\begin{align*}
   \mathbb E  \left[\left(\sum_{x\in [n]}(Y-Y_x)^2\right)^{p/2}\right]&\leq n^{p/2-1}2^{p}\mathbb E\left[ \sum_{x\in [n]}|(G,x)_t|^p\cdot \max_{o\in V}|\psi(G,o)|^p\right]\\
   &\leq n^{p/2}(\mathbb E[|(G,x)_t|^{2p}])^{1/2}  \left(\mathbb E \left[\max_{o\in V}|\psi(G,o)|^p\right]\right)^{1/2}.
\end{align*}
Then the desired bound follows from \eqref{eq:EF}.
 \end{proof}

Next, we compare the expectation of functionals on $G$ and the random hypertrees.
\begin{lemma}\label{lem:expdifference}
 Let $\ell=\kappa \log_{(q-1)p_{\max}} n$. Let $f:\mathcal G_{*}\to \mathbb R$ be a $\ell$-local function. Then there exists a universal constant $c_1>0$ such that for all $x\in [n]$,
 \begin{align}
    | \dE f(G,x)-\dE f(T,x)|\leq  c_1 \sqrt{ p_{\max}\log(n)\, n^{2\kappa -1} }\left(\sqrt{\dE |f(G,x)|^2}\vee \sqrt{\dE |f(T,x)|^2}\right).
 \end{align}
\end{lemma}
\begin{proof}
  Let $\mathcal E_{\ell}$ be the event that the coupling between $(G,x)_{\ell}$ and $(T,x)_{\ell}$ fails. Then on $\mathcal E_{\ell}^c$, we have $f(G,x)=f(T,x)$ because $f$ is $\ell$-local. Then
  \begin{align*}
      | \dE f(G,x)-\dE f(T,x)|&\leq  \dE [|f(G,x)- f(T,x)|\mathbf{1}_{\mathcal E_{\ell}}]\\
      &\leq \sqrt{\dP (\mathcal E_{\ell})}\sqrt{\dE |f(G,x)-f(T,x)|^2}\\
      &\leq \sqrt{\dP (\mathcal E_{\ell})}\left(\sqrt{\dE |f(G,x)|^2}+\sqrt{\dE |f(T,x)|^2}\right)\\
      &\leq c_0 \sqrt{ p_{\max} \log(n)n^{2\kappa -1} }\left(\sqrt{\dE |f(G,x)|^2}+\sqrt{\dE |f(T,x)|^2}\right),
  \end{align*}
  where the last inequality is from Proposition \ref{prop:coupling}.
\end{proof}

Finally, we prove a concentration bound between the functional defined on hypergraphs and its expectation on hypertrees. 
\begin{lemma}\label{lem:functional_concentration}
 Let $\ell = \kappa \log_{(q-1)p_{\max}}(n)$, and $f : \mathcal  G_* \to \mathbb R$ be an $\ell$-local function such that $f(g, o) \leq \alpha |(g, o)_\ell |^{\beta}$ for some $\alpha, \beta$ not depending on $g$. Then, for all $s>0$, with probability $1 - n^{-s}$,
 \[ \left| \sum_{x\in [n]}f(G,x)-\mathbb E  \sum_{x\in [n]} f(T,x) \right| \leq c_3 \alpha s^{3/2+\beta}\log(n)^{3/2 + \beta} n^{\kappa(1 + \beta) + 1/2}. \]
\end{lemma}
\begin{proof} We shall use the following version of the Chebyshev inequality: for a random variable $X$ and any $p \geq 1$,
\begin{equation}\label{eq:chebyshev_generalized}
    \Pb*{|X| > a\, \E{X^{2p}}^{\frac 1 {2p}}} \leq a^{-2p}.
\end{equation}
Let $\psi(g, o) = \alpha |(g, o)_\ell |^{\beta}$; it is easily checked that $\psi$ satisfies the conditions of Lemma \ref{lem:Effron}. Further, using \eqref{eq:neighbourhood_max_expectation}, we have
\[ \E*{\max_{o\in [n]}\psi(G,o)^{2p}}^{\frac{1}{2p}} \leq c_1 \alpha (\log(n) \vee p)^\beta [(q-1)p_{\max}]^{\beta \ell}. \]
We apply \eqref{eq:chebyshev_generalized} with $a = e$ and $2p \geq s \log(n)$: from Lemma \ref{lem:Effron} ,with probability at least $1 - n^{-s},$
\[ \left| \sum_{x\in [n]}f(G,x)-  \dE f (G,x) \right| \leq c_2 \alpha s^{3/2+\beta} \log(n)^{3/2 + \beta} n^{\kappa(1 + \beta) + 1/2} . \]
The exponent in $n$ above is more than the one in Lemma \ref{lem:expdifference}, and hence using the triangle inequality

\[ \left| \sum_{x\in [n]}f(G,x)-\mathbb E  \sum_{x\in [n]} f(T,x) \right| \leq c_3 \alpha s^{3/2+\beta}\log(n)^{3/2 + \beta} n^{\kappa(1 + \beta) + 1/2}. \]
\end{proof}

\begin{proposition}\label{prop:functional_concentration_final}
 Let $f$ be a $\ell$-local function satisfying the hypotheses of Lemma \ref{lem:functional_concentration}, and define $\overline f$ as in \eqref{eq:GW_transform_definition}.
Then for all $s>0$, with probability $1 - n^{-s}$,
 \[ \left | \frac1n \sum_{x \in [n]} f(G, x) - \sum_{i \in [r]} \pi_i \overline f(i) \right | \leq c_3 \alpha s^{3/2+\beta} \log(n)^{3/2 + \beta} n^{\kappa(1 + \beta) - 1/2}. \]
\end{proposition}

\begin{proof}
  From the definitions of $\overline f$ and $\pi$, we have
  \[ \frac1n \sum_{x \in [n]} \E{f(T, x)} = \sum_{i \in [r]} \pi_i \overline f(i), \]
  and the result ensues from an application of Lemma \ref{lem:functional_concentration}.
\end{proof}

\section{Structure of near eigenvectors}\label{sec:near_eigenvec}

Throughout this section, we will make use of the following functional, defined for $t \geq 0$ and $\xi \in \dR^r$:
 \begin{align}\label{eq:defh}
 h_{\xi, t}(g, o) = \ind_{(g, o)_t \text{ is tangle-free}} \sum_{x_0 = o, \dots, x_t} \xi(\sigma(x_t)), 
 \end{align}
  where the sum ranges over all non-backtracking paths of length $t$ starting at $x_0 = o$.
  \begin{lemma}\label{lem:non_backtracking_functional}
  The functional $h_{\xi, t}$ satisfies the following properties:
      \begin{enumerate}
          \item $|h_{\xi, t}(g, o)| \leq 2 \norm{\xi}_\infty |(g, o)_t|$ 
          \item $h_{\xi, t}(T, \rho) = f_{\xi, t}(T, \rho)$ if $(T, \rho)$ is a hypertree, where $f_{\xi, t}(T, \rho)$ was defined in \eqref{eq:def_tree_functional}.
      \end{enumerate}
  \end{lemma}
  
  \begin{proof}
    If $(g, o)_t$ is tangle-free, there are at most two paths from $o$ to $x_t$ for any vertex $x_t$ in $(g, o)_t$. Hence the first part of the lemma holds. For the second part, notice that if $(T, \rho)$ is a hypertree, then there is a path from $o$ to $x_t$ if and only if $x_t$ is at depth $t$ in $T$, and this path is unique.
  \end{proof}

\begin{lemma}\label{lem:scal_u_v}
  Let $\ell\leq \kappa \log_{(q-1)p_{\max}}n$. With probability at least $1 - c_1 n^{12\kappa - 1}$, the following inequality holds for any $i, j \in [r]$ and $0 \leq t \leq 3\ell$:
 \begin{equation}
      |\langle B^t J \chi_i, \chi_j\rangle - n [(q-1)\mu_i]^{t+1} \delta_{ij} |\leq c_2 \log(n)^{7/2} n^{6\kappa + 1/2} \label{eq:scal_u_v}.
 \end{equation}
\end{lemma}

\begin{proof}
  Define the functional
  \[ f(g, o) = \phi_j(\sigma(o))\partial h_{\phi_i, t}(g , o), \]
  where $\phi_i,\phi_j$ are  eigenvectors of $Q$, $\partial h$ is defined in \eqref{eq:def_partial_functional}, and $h_{\phi_i,t}$ is defined in \eqref{eq:defh}. Recall the definition of $\chi_i$ in \eqref{eq:defchi}.
  When $(G, x)_t$ is tangle-free, we have
  \begin{align*} 
  f(G, x) &= \phi_j(\sigma(x)) \sum_{x_0=x, \dots, x_{t+1}} \phi_i(x_{t+1}) \\
  &=\sum_{e: x\in e}\phi_j(\sigma(x))  [B^{t} J \chi_i](x \to e)=\sum_{e:x \in e} \chi_j(x \to e) [B^{t} J \chi_i](x \to e),
  \end{align*}
  and hence on the event described in Lemma \ref{lem:tangle-free} (i) that $H$ is $3\ell+1$-tangle free,
  \[ \sum_{x \in [n]} f(G, x) = \langle \chi_j, B^{t}J \chi_i \rangle. \]
  On the other hand, it is straightforward to check that $f$ is $(t+1)$-local and satisfies
  \[ |f(g, o)| \leq 2|(g, o)_{t+1}|,\]
  using the properties of $h$ outlined above.
  Additionally, from Lemma \ref{lem:gw_transform_recurrence}, 
  \[ \overline f(k) =\phi_j(k) \overline{\partial h_{\phi_i,t}} =\phi_j(k) (q-1) [Q\overline{h_{\phi_i, t}}](k), \]
  and Lemma \ref{lem:non_backtracking_functional} implies
  \[ \phi_j(k) (q-1) [Q\overline{h_{\phi_i, t}}](k)= \phi_j(k) (q-1) [Q\overline{f_{\phi_i, t}}](k).
  \]
  And we can apply \eqref{eq:fphit} in Proposition \ref{prop:gw_transform_expectation}:
  \begin{align*} 
    \sum_{k \in [r]} \pi(k) \overline f(k)&=(q-1)\sum_{k\in [r]} \pi(k) \phi_j(k)[Q\overline{h_{\phi_i, t}}](k) \\
    &= [(q-1)\mu_i]^{t+1} \sum_{k \in [r]} \pi(k) \phi_i(k)\phi_j(k)= [(q-1)\mu_i]^{t+1} \delta_{ij},
  \end{align*}
  where the last equation is from \eqref{eq:pi_orthogonal}.
  Then \eqref{eq:scal_u_v} follows from Proposition \ref{prop:functional_concentration_final}.
  \end{proof}

  \begin{lemma}\label{lem:scal_uu_vv}
    Let $\ell\leq \kappa \log_{(q-1)p_{\max}}n$, and define
    \begin{align}\label{eq:def_gamma_i_t}
  \gamma_i^{(t)} = \frac{1 - \tau_i^{t+1}}{1 - \tau_i} + \frac{q-2}{(q-1)\mu_i}\frac{1-\tau_i^t}{1-\tau_i}. 
    \end{align}
  Then for any $i, j \in [r]$ and $t\leq\ell$, the following holds:
    \begin{align}
        \left| \langle B^t J \chi_i, B^t J \chi_j \rangle - [(q-1)\mu_i]^{2t}n(q-1)(d\gamma_i^{(t)} + (q-2)\mu_i) \delta_{ij} \right| &\leq c_1 \log(n)^{9/2} n^{3\kappa - 1/2}, \label{eq:scal_uuvv1}\\
        \left| \langle (B^*)^t \chi_i, (B^*)^t \chi_j \rangle - [(q-1)\mu_i]^{2t}n d\gamma_i^{(t)} \delta_{ij} \right| &\leq  c_2 \log(n)^{9/2} n^{3\kappa - 1/2}.\label{eq:scal_uuvv2}
    \end{align}
  \end{lemma}

  \begin{proof}[Proof of Lemma \ref{lem:scal_uu_vv}]
  We begin with a lemma that will be proven in Appendix \ref{sec:Misc}:
  \begin{lemma}\label{lem:gw_transform_product}
   For any vectors $\phi, \phi'\in \mathbb R^{r}$,
   \begin{equation}\label{eq:general_vec_product}
       \langle \pi, Q^{(3)}(\phi \otimes \phi') \rangle = \phi^* \Pi Q \phi'.
   \end{equation}
   In particular, if $\phi_i, \phi_j$ are two eigenvectors of $Q$ for $\mu_i,\mu_j$,  we have 
   \begin{align}\label{eq:eigen_vec_product}
    \langle \pi, Q^{(3)}(\phi_i \otimes \phi_j) \rangle &= \mu_i\delta_{ij},  \\
   \langle \pi, Q^s y^{(\phi_i, \phi_j)} \rangle &=d^{s+1}\delta_{ij}+(q-2)d^s\mu_i\delta_{ij}.\label{eq:pi_yphi_phi}
   \end{align}
   As a result, for any $i, j \in [r]$ and $t \in \dN$, 
   \begin{align}\label{eq:pi_eigen_vec_product}
     \langle \pi, \overline{f_{\phi_i, t}f_{\phi_j, t}} \rangle = [(q-1)\mu_i]^{2t} \gamma_{i}^{(t)} \delta_{ij}.  
   \end{align} 
  \end{lemma}
  
  Now define
  \[ f(g, o) = \sum_{e:  o\in e} \left(\sum_{o' \in e} h_{\phi_i, t}(g, o')\right)\left(\sum_{o' \in e} h_{\phi_j, t}(g, o')\right). \]
  When $G$ is tangle-free, we have
  \[ \sum_{x\in [n]} f(G, x) = \sum_{x\in [n]} \sum_{x \to e} [B^t J \chi_i](x \to e) [B^t J \chi_j](x \to e) = \langle B^t J \chi_i, B^t J \chi_j \rangle. \]
  Additionally, using \eqref{eq:e_x_x} and \eqref{eq:e_x_x_2},
  \[ \overline f = (q-1)Q\overline{f_{\phi_i, t}f_{\phi_j, t}} + (q-1)(q-2) Q^{(3)}(\overline{f_{\phi_i, t}} \otimes \overline{f_{\phi_j, t}}). \]
Since $\pi^*Q=d\pi^*$,  we can now use Lemma \ref{lem:gw_transform_product} to get:
  \begin{align*}
  \langle \pi, \overline f \rangle &= (q-1)d[(q-1) \mu_i]^{2t} \gamma_i^{(t)} \delta_{ij}+ (q-1)(q-2)[(q-1)\mu_i]^{2t}  \mu_i\delta_{ij}\\
  &= [(q-1) \mu_i]^{2t} (q-1) (d \gamma_i^{(t)} + (q-2)\mu_i)\delta_{ij}.
  \end{align*}
  Proposition \ref{prop:functional_concentration_final} then implies the first result.   For the second one,  with Lemma \ref{lem:formula_P}, we can  write
  \begin{align*}
  \langle (B^*)^t \chi_i, (B^*)^t \chi_j \rangle &= \left\langle \frac{J^2 - (q-2)J}{q-1}(B^*)^t \chi_i, (B^*)^t \chi_j \right\rangle \\
  &= \frac1{q-1} \langle B^t J \chi_i, B^t J \chi_j \rangle - \frac{q-2}{q-1}\langle B^{2t} J \chi_i, \chi_j \rangle.
  \end{align*}
  \eqref{eq:scal_uuvv2} then results from \eqref{eq:scal_uuvv1}, as well as Lemma \ref{lem:scal_u_v}.
  \end{proof}

Next, we show a lemma that quantifies the pseudo-eigenvector scaling of the $J \chi_i$:
\begin{lemma}\label{lem:telescopediff}
 Let $\ell\leq \kappa \log_{(q-1)p_{\max}}n$. With probability at least $1 - c_1 n^{4\kappa - 1}$, the following inequality holds for any $i \in [r]$  and $0 \leq t \leq 2\ell$:
 \begin{equation}\label{eq:telescopediff}
     \norm{B^{t+1}J \chi_i - (q-1)\mu_i B^t J \chi_i}^2 \leq (q-1)[(q-1)d]^{t+3}n + c_1 \log(n)^{9/2} n^{6\kappa + 1/2}.
 \end{equation}
 {In particular when $\kappa \leq 1/12$,
 \begin{equation}\label{eq:telescopediff2}
     \norm{B^{t+1}J \chi_i - (q-1)\mu_i B^t J \chi_i} \leq c_2 [(q-1)d]^{t/2}\sqrt n.
 \end{equation}}
\end{lemma}

\begin{proof}
  Define
  \[ f(g, o) = \partial H(g, o) \qquad \text{where} \qquad H(g, o) = (h_{\phi_i, t+2}(g, o) - (q-1)\mu_i h_{\phi_i, t+1}(g, o))^2. \]
  As before, it is easy to check that when $(G, x)$ is tangle-free,
  \[ f(G, x) = \sum_{e: x \in e} \left([B^{t+1}J \chi_i](x \to e) - (q-1)\mu_i [B^t J \chi_i](x \to e)\right)^2.  \]
  On the other hand, Lemma \ref{lem:gw_transform_recurrence} and \eqref{eq:Fphit} in Proposition \ref{prop:gw_transform_expectation} imply that 
  \[ \overline f =(q-1)Q\overline{H}= (q-1)^{t+3}Q^{t+2}y^{(\phi_i, \phi_i)}. \]
  Using \eqref{eq:pi_yphi_phi} in Lemma \ref{lem:gw_transform_product} and the fact that $\mu_i \leq d$, we find
  \[ |\langle \pi, \overline f \rangle | \leq (q-1)[(q-1)d]^{t+3}. \]
  Using Proposition \ref{prop:functional_concentration_final}, as well as the  inequality that
  \[   \norm{B^{t+1}J \chi_i - (q-1)\mu_i B^t J \chi_i}^2=\sum_{x\in [n]} f(G,x)\leq n\left|\frac{1}{n}\sum_{x\in [n]} f(G,x)-\sum_{i}\pi_i \overline{f}(i) \right|+n |\langle \pi, \overline{ f}\rangle |,
  \]
  we finish the proof.
\end{proof}

\begin{lemma}\label{lem:chiBw}
Let $\ell\leq \kappa \log_{(q-1)p_{\max}}(n)$, $w$ be any unit vector orthogonal to all $(B^{*})^{\ell}\chi_i, i\in [r_0]$. With probability at least   $1 - c_1 n^{4\kappa - 1}$, for any $t\leq \ell$, there are constants $c_2,c_3>0$ such that 
\begin{align}\label{eq:robound}
    |\langle \chi_i, B^{t}w\rangle  |\leq c_2(q-1)^{1/2} [(q-1)d]^{t/2+1}\sqrt{n}+c_3\log^3(n) n^{3\kappa+1/4}.
\end{align}
\end{lemma}
\begin{proof}
For $i\in [r_0]$,
\begin{align}
   [ (q-1)\mu_i]^{-t}\langle (B^*)^{t}\chi_i, w\rangle = [ (q-1)\mu_i]^{-t}\langle (B^*)^{t}\chi_i, w\rangle-[ (q-1)\mu_i]^{-\ell}\langle B^{*\ell}\chi_i, w\rangle.\notag
\end{align}
Hence
\begin{align}\label{eq:sumB*}
[ (q-1)\mu_i]^{-t}| \langle (B^*)^{t}\chi_i, w\rangle|
\leq & \sum_{s=t}^{\ell-1} [ (q-1)\mu_i]^{-(s+1)}|\langle (B^*)^{s+1}\chi_i,w\rangle -[(q-1)\mu_i]\langle (B^*)^{s}\chi_i,w\rangle|.  
\end{align}
Using Lemma \ref{lem:formula_P}, we find
\begin{align*}
  &|\langle ({B^{*}})^{s+1}\chi_i,w\rangle -[(q-1)\mu_i]\langle (B^{*})^s\chi_i,w\rangle|\\
  \leq& \frac{1}{q-1} |\langle B^{s+1}J \chi_i -(q-1)\mu_i  B^s J \chi_i,Jw\rangle |+\frac{q-2}{q-1}|\langle B^{s+1} J \chi_i -(q-1)\mu_i B^s J \chi_i, w\rangle| \\
  \leq & 2 \|B^{s+1}J \chi_i-(q-1)\mu_i B^s J \chi_i\|
  \leq  2(q-1)^{1/2}[(q-1)d]^{(s+3)/2}\sqrt{n}+c_1 \log(n)^{9/4}n^{3\kappa+1/4},
\end{align*}
where we use the fact that $\|P\|\leq q-1$ and $\|w\|=1$, and the last inequality is from \eqref{eq:telescopediff2}.

Since $(q-1)\mu_i>\sqrt{(q-1)d}$ for $i\in [r_0]$,  from \eqref{eq:sumB*}, 
\begin{align*}
     |\langle \chi_i, B^{t}w\rangle|&=|\langle (B^*)^{t}\chi_i,w\rangle|\\
     &\leq \sum_{s=t}^{\ell-1}[(q-1)\mu_i]^{t-s-1}\left( 2(q-1)^{1/2}[(q-1)d]^{(s+3)/2}\sqrt{n}+c_1 \log(n)^{9/4}n^{3\kappa+1/4}\right)\\
     &\leq \sum_{s=t}^{\ell-1}[(q-1)d]^{(t-s-1)/2}\left( 2(q-1)^{1/2}[(q-1)d]^{(s+3)/2}\sqrt{n}+c_1 \log(n)^{9/4}n^{3\kappa+1/4}\right)\\
     &\leq  c_2(q-1)^{1/2} [(q-1)d]^{t/2+1}\sqrt{n}+c_3\log^3(n) n^{3\kappa+1/4}.
\end{align*}
\end{proof}

 \section{Proof of Proposition \ref{prop_main}}\label{sec:prop_proof}
 
 After proving all the lemmas in Section \ref{sec:near_eigenvec}, we summarize the estimates and finish the proof Proposition \ref{prop_main} in this section.
    
{\subsection{Proof of \eqref{eq:norm_bound1}-\eqref{eq:norm_bound4}}}

    For any $r_0\times r_0$ matrix $M$, we have 
   \begin{align}\label{eq:spectral_entrywise}
       \|M\|\leq \|M\|_F\leq r_0 \max_{ij}|M_{ij}|.
   \end{align} 
   Hence \eqref{eq:norm_bound1} and \eqref{eq:norm_bound2} follow from the entrywise bounds in Lemma \ref{lem:scal_uu_vv} by taking $\kappa=1/25$ and $t=\ell$.  Now for any $i,j\in [r_0]$,
 \begin{align}
     \langle u_i,v_j\rangle =\frac{\langle B^{\ell}J \chi_i, (B^*)^{\ell}\chi_j\rangle  }{n(q-1)^{2\ell+1}\mu_i^{\ell+1} \mu_j^{\ell} }=\frac{\langle B^{2\ell}J \chi_i, \chi_j\rangle  }{n(q-1)^{2\ell+1}\mu_i^{\ell+1} \mu_j^{\ell} }.\notag
 \end{align}
 Then \eqref{eq:norm_bound3} is due to  Lemma \ref{lem:scal_u_v} by taking $t=2\ell$.   To prove \eqref{eq:norm_bound4}, we use \eqref{eq:spectral_entrywise} again.
 \begin{align}
     (V^* B^{\ell} U)_{ij}=\langle v_i, B^{\ell} u_j\rangle=\frac{\langle B^{3\ell}J \chi_j, \chi_i\rangle  }{n(q-1)^{2\ell+1}\mu_i^{\ell} \mu_j^{\ell+1} }.\notag
 \end{align}
 By taking $t=3\ell$ in  Lemma \ref{lem:scal_u_v}, 
 \[ |(V^* B^{\ell} U)_{ij}-\Sigma_{ij}^{\ell}|\leq c_2\log (n)^{7/2} n^{6\kappa-1/2}.
 \]
 Hence \eqref{eq:norm_bound4} holds by taking $\kappa=1/25$. 

 {
 \subsection{Proof of \eqref{eq:norm_bound6}-\eqref{eq:norm_bound7}}
 }

Bound \eqref{eq:norm_bound6} is equivalent to the following lemma:
 \begin{lemma}\label{lem:bulk_B}
Let $\ell\leq \kappa \log_{(q-1)p_{\max}}(n)$ with $\kappa<1/24$, $w$ be any unit vector orthogonal to all $B^{*\ell}\chi_i, i\in [r]$. For sufficiently large $n$, with probability at least   $1 -n^{-c}$ for a constant $c>0$,
\begin{align}
    \|B^{\ell}w\|\leq (\log n)^{9}n^{7\kappa/2-1/4}.\notag
\end{align}
\end{lemma}

\begin{proof}
  By  Lemma \ref{lem:tangle-free} and Proposition \ref{prop:trace}, with probability at least $1-n^{-c}$, $H$ is $\ell$-tangle free and the bounds in Proposition \ref{prop:trace} holds.   Therefore by Lemma \ref{lem:Blwnorm} and Proposition \ref{prop:trace},
  \begin{align}
      \|B^{\ell}w\|&=\|B^{(\ell)}w\| \leq c_1(\log n)^{10} n^{\kappa/2}+c_1n^{-1/2}(\log n)^5\sum_{j=1}^r \sum_{t=1}^{\ell-1}[(q-1)d]^{t/2}|\langle   \chi_j, B^{(\ell-t-1)}w\rangle|.\label{eq:lasttermrobound}
  \end{align}
  For $j\in [r_0]$, we use \eqref{eq:robound} to get
  \begin{align*}
     c_1 (\log n)^5 n^{-1/2} \sum_{j=1}^{r_0} \sum_{t=1}^{\ell-1}[(q-1)d]^{t/2}|\langle   \chi_j, B^{(\ell-t-1)}w\rangle|
     &\leq c_2(\log n)^6n^{\kappa/2}+c_2(\log n)^{8} n^{7\kappa/2-1/4}.
  \end{align*}
For $j>r_0$ in \eqref{eq:lasttermrobound}, we have  $(q-1)\mu_j\leq \sqrt{(q-1)d}$, and conditioned on  $G$ being $\ell$-tangle free, we have
\[ |\langle   \chi_j, B^{(\ell-t-1)}w\rangle|\leq \|{(B^*)}^{\ell-t-1}\chi_j\|\leq  c_3 n^{\kappa/2+1/2}[(q-1)d]^{-t/2}
+c_3 \log^{9/4}n^{2\kappa-1/4}.\]
where the last inequality is given by \eqref{eq:scal_uuvv2} and the triangle inequality. This implies
\begin{align*}
   &   c_1 (\log n)^5 n^{-1/2} \sum_{j=r_0+1}^{r} \sum_{t=1}^{\ell-1}[(q-1)d]^{t/2}|\langle   \chi_j, B^{(\ell-t-1)}w\rangle|\leq c_4(\log n)^6 n^{\kappa/4} +c_4(\log n)^{9/4} n^{5\kappa/2-1/4}.
\end{align*}

Therefore \eqref{eq:lasttermrobound} is bounded by $(\log n)^{9}n^{7\kappa/2-1/4}$ for sufficiently large $n$.
\end{proof}

A similar bound holds for $(B^{\ell})^*$. The proof is similar to the proof of Lemma \ref{lem:bulk_B}, and we defer it to Appendix \ref{sec:Misc}.
\begin{lemma}\label{lem:bulk_B_left}
Let $\ell\leq \kappa \log_{(q-1)p_{\max}}(n)$ with $\kappa<1/24$, $w$ be any unit vector orthogonal to all $B^{\ell} J \chi_i, i\in [r]$.For sufficiently large $n$, with probability at least   $1 -n^{-c}$ for a constant $c>0$,
\begin{align}
    \| (B^\ell)^{*}w\|\leq  (\log n)^{9}n^{7\kappa/2-1/4}.\notag
\end{align}
\end{lemma}

{\eqref{eq:norm_bound6} follows  directly from Lemma \ref{lem:bulk_B} and \eqref{eq:norm_bound5} follows similarly from Lemma \ref{lem:bulk_B_left}. \eqref{eq:norm_bound7} is due to \eqref{eq:KBk} in Proposition \ref{prop:trace}, whose proof is in Section \ref{sec:KBk}. This completes the proof of Proposition \ref{prop_main}.}

 \subsection*{Acknowledgments}
Y.Z. was partially supported by  NSF-Simons Research Collaborations on the Mathematical and Scientific Foundations of Deep Learning. Y.Z. also acknowledges support from NSF  DMS-1928930 during his participation in the program “Universality and Integrability in Random Matrix
Theory and Interacting Particle Systems” hosted by the Mathematical Sciences Research Institute
in Berkeley, California, during the Fall semester of 2021. Both authors thank Simon Coste for helpful discussions.

\bibliographystyle{plain}
\bibliography{bib_all}

\appendix

\section{Proof of Proposition \ref{prop:trace}}\label{sec:trace_bound}

\subsection{Proof of \eqref{eq:Delta} on $\Delta^{(k)}$}\label{sec:Deltak}

Let $\vec{e}_1,\dots,\vec {e}_{2m}$ be oriented hyperedges. 
With the convention that $\vec e_{2m+1}=\vec e_{1}$, we have the following trace expansion bound:
\begin{align}
    \|\Delta^{(k)}\|^{2m}&\leq \tr \left(\Delta^{(k)} {\Delta^{(k)}}^*\right)^m  \notag \\
    &=\sum_{\vec{e}_1,\dots,\vec{e}_{2m}}\prod_{i=1}^m \Delta^{(k)}_{\vec{e}_{2i-1},\vec{e}_{2i}} \Delta^{(k)}_{\vec{e}_{2i+1},\vec{e}_{2i}} =\sum_{\gamma\in W_{k,m}}\prod_{i=1}^{2m}\prod_{s=0}^k \underline{\mathcal A}_{\gamma_{i,s}}, \label{eq:DA}
\end{align}
where $W_{k,m}$ is the set of sequence of paths $(\gamma_1,\dots,\gamma_{2m})$ such that 
$ \gamma_i=(\gamma_{i,0},\dots,\gamma_{i,k})$ is a non-backtracking tangle-free walk of length $k$ for all $i=1,\dots,2m$, and for all $1\leq i\leq m$,  
\begin{align}\label{eq:requivalent}
    \gamma_{2i-1,k}=\gamma_{2i,k}, \quad \gamma_{2i,0}=\gamma_{2i+1,0},
\end{align}
with the convention that $\gamma_{2m+1}=\gamma_1$. See Figure \ref{fig:gamma}.
\begin{figure} 
    \centering
\includegraphics[width=0.5\linewidth]{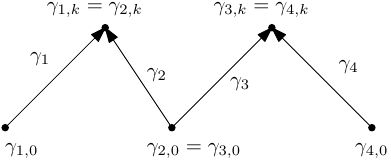}
    \caption{concatenation of non-backtracking walks}
    \label{fig:gamma}
\end{figure}
Taking the expectation yields,
\begin{align}\label{eq:EDelta}
    \mathbb E\|\Delta^{(k)}\|^{2m} \leq  \sum_{\gamma\in W_{k,m}'}\mathbb E\prod_{i=1}^{2m}\prod_{s=0}^k \underline{\mathcal A}_{\gamma_{i,s}},
\end{align}
where $W_{k,m}'$ is the subset of $W_{k,m}$ such that each distinct hyperedge is visited at least twice, and the rest of the terms in \eqref{eq:DA} are zero after taking the expectation.

For each $\gamma\in W_{k,m}$, we associate a bipartite graph $G(\gamma)=(V(\gamma), H(\gamma), E(\gamma))$ of visited vertices and hyperedges. We call $V(\gamma)$ the left vertex set and $H(\gamma)$ the right vertex set; see Figure \ref{fig:represent} for a bipartite representation of a hypergraph.

\begin{figure}
    \centering
\includegraphics[width=0.5\linewidth]{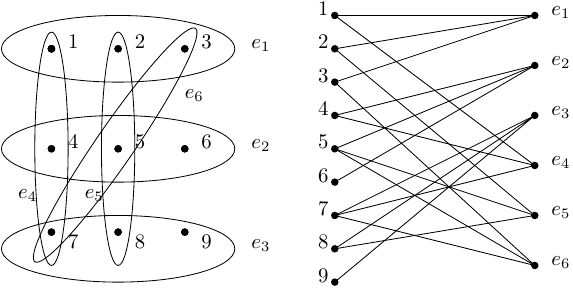}
    \caption{A bipartite representation (left) of a 3-uniform hypergraph (right)}
    \label{fig:represent}
\end{figure}

There is an edge between $(x,e)$ for $x\in V(\gamma), e\in H(\gamma)$ if and only if $(x,e)$ is an oriented hyperedge appearing in $\gamma$ or $x\in e$ and there exists a path $(w,e)\to (x,f)$ in $\gamma$ for some $w,f$.
See Figure \ref{fig:bipartite}. 

\begin{figure} 
    \centering
    \includegraphics[width=0.3\linewidth]{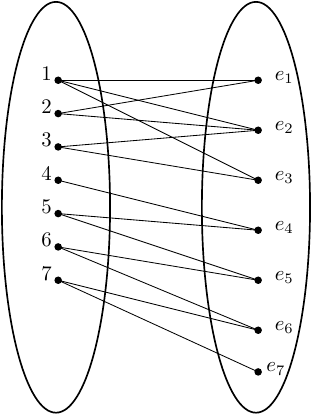}
    \caption{A bipartite representation of non-backtracking walks. $(1\to e_1), (2 \to e_2), (3 \to e_3), (1 \to e_2)$ is a non-backtracking walk of length $3$ on the hypergraph, which corresponds to the non-backtracking walk  $(1,e_1,2,e_2,3,e_3,1,e_2)$  of length $7$ in the bipartite graph. Similarly, $(4,e_4,5,e_5,6,e_6,7,e_7)$ is a non-backtracking walk of length $7$ in the bipartite graph, which corresponds to the non-backtracking walk  $(4\to e_4), (5\to e_5),(6\to e_6), (7\to e_7)$ of length $3$ in the hypergraph.}
    \label{fig:bipartite}
\end{figure}

Now a non-backtracking walk on the hypergraph is represented as a corresponding non-backtracking walk on the vertex set of the bipartite graph. Then $\gamma$ is a concatenation of non-backtracking tangle-free walks of length $(2k+1)$ in the corresponding bipartite graph. We call $G(\gamma)$ a \textit{factor graph representation} of $\gamma$.

 A path given by a vertex-hyperedge sequence is \textit{canonical} if $V(\gamma)=\{1,\dots,v \}$, $H(\gamma)=\{e_1,\dots, e_h \}$ and the vertices and hyperedges are first visited in the increasing order.  Let  $\mathcal W_{k,m}$ be the set of canonical paths in $W_{k,m}$.

\begin{lemma}\label{lem:bipartite_NB}
 Let $\mathcal W_{k,m}(v,h,e)$ be the set of canonical paths in $W_{k,m}$. We have 
 \begin{align}
     |\mathcal W_{k,m}(v,h,e)|\leq (3k)^{2m}(6km)^{6m(e-v-h+1)}.\notag
 \end{align}
\end{lemma}
The proof follows steps in Lemma 17 of \cite{bordenave2018nonbacktracking}, and we make some necessary changes and include the proof for completeness. 
\begin{proof}
  We need to find an injective way to encode canonical paths $\gamma\in W_{k,m}(v,h,e)$. We represent $\gamma$ as a concatenation of non-backtracking walks on the factor graph.
  Let 
  $\gamma=(x_{i,t})_{1\leq i\leq 2m, 0\leq t\leq 2k+1}\in \mathcal W_{k,m}(v,h,e)$. $x_{i,t}$ comes from the left vertex set when $t$ is even, and from the right vertex set when $t$ is odd.
  Let $y_{i,t}$ be the edge of $\gamma$.
  
  We explore the sequence $x_{i,t}$ in lexicographic order and consider each index $(i,t)$ as a time. For $0\leq t\leq 2k$, we say $(i,t)$ is a first time if $x_{i,t+1}$ has not been seen before.
  If $(i,t)$ is a first time we call the edge $y_{i,t}$ a tree edge. By construction, the set of tree edges forms a tree $T$ with vertex set $\{1,\dots,v\}\cup \{e_1,\dots,e_h \}$.
  The edges not in $T$ are called the excess edges of $\gamma$. It follows that the number of excess edges is $\epsilon=e-v-h+1$.

  We build a first encoding of $\mathcal W_{k,m}(v,e)$.   if $(i,t)$ is not a first time, we say $(i,t)$ is an important time and mark the time $(i,t)$ by the vector $(x_{{i,t}},x_{i,\tau})$,  where $(i,\tau)$ is the next time that $x_{i,\tau}$ will not be a tree edge (by convention $\tau=2k+1$ if $x_{i,s}$ remains on the tree for all $2\leq s\leq 2k+1$). 
  Since there is a unique non-backtracking path between two vertices of a tree, we can reconstruct $\gamma\in W_{k,m}$ from the position of the important times and their mark. 
  
  Now we use the fact that each path $x_i$ is tangle free. We partition important times into three categories, short cycling, long cycling, and superfluous times.
    First, consider the case where the $i$-th path $\gamma_i$ contains a cycle. For each $i$, the first time $(i,t)$ such that $x_{i,t+1}\in \{x_{i_0},\dots, x_{i,t}\}$ is called a short cycling time. Let $0\leq \sigma\leq t$ be such that $x_{i,t+1}=x_{i,\sigma}$. By the assumption of tangle-freeness, $C=(x_{i,\sigma},\dots, x_{i,t+1})$ is the only cycle visited by $x_i$. We denote by $(i,\tau)$ the first time after $(i,t)$ that $y_{i,\tau}$ is not an edge of $C$ (by convention $\tau=2k+1$) if $x_i$ remains on $C$). We add the extra mark $\tau$ to the short cycling time.
  Important times $(i,t)$ with $1\leq t<\sigma$ with $1\leq t<\sigma$ or $\tau<t\leq 2k+1$ are called  long cycling times. The other important times are called superfluous.
  The key observation is that for each $1\leq i\leq 2m$, the number of long cycling times $(i,t)$ is bounded by $\epsilon-1$. Since there is at most one cycle, no edges of $x$ can be seen twice outside those of $C$, the $-1$ coming from the fact that the short cycling time is an excess edge.

  Now consider the case where $x_i$ does not contain a cycle, then all important times are called long cycling times, and their number is bounded by $\epsilon$. Now we can reconstruct $x$ from the positions of the long cycling and the short cycling times and their marks. For each $1\leq i\leq 2m$, there are at most $1$ short cycling time and $\epsilon-1$ long cycling times within $x_i$ if $x_i$ contains a cycle, or $0$ short cycling time and $\epsilon$ long cycling times if $x_i$ does not contain a cycle. There are at most $(2k+1)^{2m\epsilon}$ ways to position them in time and at most $(v+h)^2$ different possible marks for a long cycling time and $(v+h)^2(2k+1)$ possible marks for a short cycling time. We find that 
  \begin{align}
      |\mathcal W_{k,m}(v,h,e)|\leq (2k+1)^{2m\epsilon} [(v+h)^2(2k+1)]^{2m} (v+h)^{4m(\epsilon-1)}.\notag
  \end{align}
    We use $v+h\leq 2m(2k+1)\leq 6km$ to obtain the desired bound.
\end{proof}

We now set  
\begin{align}\label{eq:set_m}
    m=\left\lfloor \frac{\log n}{13 \log(\log n)}\right\rfloor .
\end{align}

 For our choice of $m$, we have for sufficiently large $n$,
\begin{align}\label{eq:masymptotic}
    n^{1/2m}&=o(\log^7 n), \quad  \ell m=o(\log^2 n), \quad (2\ell m)^{6m}\leq n^{12/13}.
\end{align}

Let 
\[ v(\gamma)=|V(\gamma)|, \quad h(\gamma)=|H(\gamma)|, \quad e(\gamma)=|E(\gamma)|. 
\]

For given $v,e,h$,  the number of possible ways to label vertices in the bipartite graph $G(\gamma)$ is at most \[n^v\binom{n}{q-d_1}\binom{n}{q-d_2}\cdots \binom{n}{q-d_h},\] where $d_i$ is the degree of $e_i$ in $G(\gamma)$.  Note that 
\begin{align}\label{eq:degree_identity}
\sum_{i=1}^h d_i=\sum_{i=1}^v \deg(v_i)=e.
\end{align}
Since each vertex except the starting point of each $\gamma_i$ has degree at least $2$, we have 
\begin{equation}\label{eq:bound_e_v}
e\geq 2m+2(v-2m)=2v-2m.
\end{equation}

So  the number of possible ways of labeling is bounded by 
\begin{align}\label{eq:total_labeling}
   n^v\binom{n}{q-d_1}\binom{n}{q-d_2}\cdots \binom{n}{q-d_h}\leq n^v \left(\frac{q-1}{n-q+2}\right)^{e-h} \binom{n}{q-1}^h.
\end{align}

(1) Suppose the remaining vertices in all hyperedges are labeled with new labels.
 Since each vertex is labeled with all possibilities and there are no constraints to label hyperedges, using the assumption in \eqref{eq:constant_degree}, 
each hyperedge contributes a factor of $\frac{d}{\binom{n}{q-1}}$. We know $h\leq (k+1)m$ since each hyperedge must appear at least twice to have a nonzero contribution. 
So the total contribution is bounded by 
\begin{align*}
  S_1 &=  \sum_{h=1}^{(k+1)m}\sum_{v=1}^{m+e/2}\sum_{e=v+h-1}^{2m(2k-1)}(3k)^{2m}(6km)^{6m(e-v-h+1)} n^v \left(\frac{q-1}{n-q+2}\right)^{e-h} \binom{n}{q-1}^h\left(\frac{d}{\binom{n}{q-1}}\right)^h\\
    &\leq c_1 (3\ell)^{2m} \sum_{h=1}^{(k+1)m} \sum_{v=1}^{m+e/2}\sum_{e=v+h-1}^{2m(2k-1)} (6km)^{6m(e-v-h+1)} n^{v + h - e}d^h (q-1)^{e-h}
\end{align*}

We introduce $\eta = e - v - h + 1$; using \eqref{eq:bound_e_v}, we have
\[ v \leq \eta + 2m + h - 1. \]
Hence, with this change of variables
\begin{align}\label{eq:bound_on_S1}
    S_1 &\leq c_1 (3\ell)^{2m} \sum_{h = 1}^{(k+1)m}  \sum_{\eta = 0}^\infty \sum_{v = 1}^{\eta + 2m + h - 1} (6km)^{6m\eta} n^{1 - \eta} d^h (q-1)^{\eta + v - 1} \\
    &\leq c_2 n (3\ell)^{2m} \sum_{h = 1}^{(k+1)m}  \sum_{\eta = 0}^\infty (6km)^{6m\eta} n^{1 - \eta} d^h (q-1)^{2\eta + 2m + h - 2} \notag \\
    &\leq c_3 n (3(q-1)\ell)^{2m} \left(\sum_{h = 1}^{(k+1)m} [(q-1)d]^h \right)\left( \sum_{\eta = 0}^\infty \left(\frac{(q-1)^2(6km)^{6m}}{n}\right)^\eta\right) \notag \\
    &\leq c_4 n (3(q-1)\ell)^{2m} [(q-1)d]^{(k+1)m}=o(\log^{19m} n)[(q-1)d]^{km}, \notag 
\end{align}
since the sum in $\eta$ converges via \eqref{eq:masymptotic}.

(2)  Suppose we have $t$ hyperedges that are not free (containing existing labels) with $1\leq t\leq h$. Let $\delta_i$ be the indicator that $e_i$ is not new.  The total number of ways to choose free labels is bounded by 
\[ n^v\binom{n}{q-d_1-\delta_1}\binom{n}{q-d_2-\delta_2}\cdots \binom{n}{q-d_h-\delta_h}\leq n^v\left(\frac{q-1}{n-q+2}\right)^{e+t-h}\binom{n}{q-1}^h.
\]
We have $\binom{h}{t}$ many ways to choose $t$ hyperedges that are not free and at most $(v+qh)^t$ many ways to label them. Then the total contribution is bounded by  
\begin{align*}
 S_2
 =& \sum_{h=1}^{(k+1)m} \sum_{t=1}^h \sum_{v=1}^{m+e/2}\sum_{e=v+h-1}^{2m(2k-1)}(3k)^{2m}(6km)^{6m(e-v-h+1)} \binom{h}{t}(v+qh)^t\\ &\cdot n^v\left(\frac{q-1}{n-q+2}\right)^{e+t-h}\binom{n}{q-1}^h\left(\frac{d}{\binom{n}{q-1}}\right)^{h}\left(\frac{p_{\max}}{\binom{n}{q-1}}\right)^{t}\\
 \leq & c_1(3\ell)^{2m} \sum_{h=1}^{(k+1)m} \sum_{v=1}^{m+e/2}\sum_{e=v+h-1}^{2m(2k-1)} (6km)^{6m(e-v-h+1)} n^{v + h - e-t}d^{h}(q-1)^{e-h}\\
 &\cdot\sum_{t=1}^h\left(\frac{h(v+qh) p_{\max}(q-1)}{nd}\right)^t \\
 \leq & c_1 \frac{\log^4(n)}{n}(3\ell)^{2m}\sum_{h=1}^{(k+1)m} \sum_{v=1}^{m+e/2}\sum_{e=v+h-1}^{2m(2k-1)} (6km)^{6m(e-v-h+1)} n^{v + h - e-t}d^{h}(q-1)^{e-h},
\end{align*}
where we use \eqref{eq:masymptotic} to show the sum involving $t$ is bounded by $\log^4(n)/n$. The last line above can be controlled by the inequalities from \eqref{eq:bound_on_S1}. Therefore 
\begin{align}
S=S_1+S_2=o(\log^{19m}n)[(q-1)d]^{km}.\notag
\end{align}
 By Markov's inequality, with \eqref{eq:set_m}, for sufficiently large $n$,
 \begin{align}
     \mathbb P\left( \|\Delta^{(k)}\| \geq (\log n)^{10} [(q-1)d]^{k/2}\right)&\leq \frac{\mathbb E\|\Delta^{(k)}\|^{2m}}{(\log n)^{20m}[(q-1)d]^{km}}\leq \left(\frac{1}{\log n}\right)^m\leq n^{-c}\notag
 \end{align}
 for some constant $0<c<\frac{1}{13}$. Hence \eqref{eq:Delta} holds.

\subsection{Proof of \eqref{eq:Delta_chi} on $\Delta^{(k)}J \chi$} \label{sec:Delta_kPchi}
\begin{proof}
  From the definition of $\chi_i$ in \eqref{eq:defchi}, and the fact that $\phi_i$ is a normalized eigenvector, we have $\|J \chi_i\|_{\infty}\leq (q-1)$ for any $i\in [r]$. 
  We write 
  \begin{align}
     \| \Delta^{(k)}J \chi\|^{2}= \sum_{\vec{e},\vec{f},\vec{g}}\Delta^{(k)}_{\vec e,\vec f}~\Delta^{(k)}_{\vec e,\vec g}~(J \chi)(\vec f)(J \chi)(\vec g). \notag
  \end{align}
  Therefore
  \begin{align}
      \mathbb E \| \Delta^{(k)}J \chi\|^{2m}\leq (q-1)^{2m}\sum_{\gamma\in W_{k,m}''} \mathbb E \prod_{i=1}^2 \prod_{s=0}^k \underline{\mathcal A}_{\gamma_i,s}, \notag
  \end{align}
  where $W_{k,m}''$ is the set of paths $(\gamma_1,\dots, \gamma_{2m})$ such that $\gamma_i=(\gamma_{i,0},\dots,\gamma_{i,k})$ is a non-backtracking walk and $\gamma_{2i-1,0}=\gamma_{2i,0}$.  The only difference with $W_{k,m}$ defined in Section \ref{sec:Deltak} is the boundary condition. However, such a condition was not used in the proof of Lemma \ref{lem:bipartite_NB}. Therefore from Lemma \ref{lem:bipartite_NB}, the number of canonical paths in $W_{k,m}''$ is  at most $(3k)^{2m}(6km)^{6m(e-v-h+1)}.$
  Following the rest of the proof for \eqref{eq:Delta}  we obtain
  \[ \mathbb E \| \Delta^{(k)}J \chi\|^{2m}=o(n(q-1)^{2m}\log^{4m}(n)) [(q-1)d]^{k}.
  \]
  Then the result follows from Markov's inequality and a union bound over $1\leq k\leq \ell$. 
\end{proof}

\subsection{Proof of \eqref{eq:Rk} on $R_k^{(\ell)}$}\label{sec:proof_Rkl}
We set  
\begin{align}
    m=\left\lfloor \frac{\log n}{25 \log(\log n)}\right\rfloor .\notag
\end{align}
For $0\leq k\leq \ell$,
\begin{align}
    \|R_k^{(\ell)}\|^{2m}&\leq \textnormal{tr}\left[R_k^{(\ell)}{R_k^{(\ell)}}^*\right]^{m}=\sum_{\gamma \in T_{\ell,m,k}'} \prod_{i=1}^{2m}\prod_{s=0}^{k-1}\underline {\mathcal A}_{\gamma_{i,s}} p_{\underline\sigma(\gamma_{i,k})}\prod_{s=k+1}^{\ell}  \mathcal A_{\gamma_{i,s}}, \label{eq:Rkl}
\end{align}
where $T_{\ell,m,k}'$ is the set of sequences $\gamma=(\gamma_1,\dots, \gamma_{2m})$ such that $\gamma_i^1=(\gamma_{i,0},\dots,\gamma_{i,k-1}),  \gamma_i^2=(\gamma_{i,k+1},\dots, \gamma_{i,\ell})$ are non-backtracking tangle-free, $\gamma_i=(\gamma_i^1,\gamma_{i,k},\gamma_i^2)$ is non-backtracking tangled. And for all $1\leq i\leq m$,
$
    \gamma_{2i-1,\ell}=\gamma_{2i,\ell}, \gamma_{2i,0}=\gamma_{2i+1,0},
$
with the convention that $\gamma_{2m+1}=\gamma_1$.

Define the bipartite graph $G(\gamma)=(V(\gamma), H(\gamma), E(\gamma))$ as the union of the bipartite graph $G(\gamma_i^1), G(\gamma_i^2)$ for all $1\leq i\leq 2m$, where  $G(\gamma_i^1), G(\gamma_i^2)$ are constructed as in Section \ref{sec:Deltak}.  Note that the edges $\gamma_{i,k}$  are not taken into account in $G(\gamma)$. And for any fixed $\gamma_{i}^1,\gamma_{i}^2,  1\leq i\leq 2m,$ there are at most $(q-1)^{2m}\binom{n}{q-2}^{2m},$ many ways to choose $\gamma_{i,k}, 1\leq i\leq 2m$.

Set $v(\gamma)=|V(\gamma)|, h(\gamma)=|H(\gamma)|$, $e(\gamma)=|E(\gamma)|$. Since $\gamma_i$ is tangled, each connected component of $G(\gamma)$ contains a cycle. Therefore $v+h\leq e$. 
Taking the expectation in \eqref{eq:Rkl} implies
\begin{align}
     \mathbb E\|R_k^{(\ell)}\|^{2m}
    &\leq p_{\max}^{2m}\binom{n}{q-2}^{2m}\sum_{\gamma \in T_{\ell,m,k}}\mathbb E \prod_{i=1}^{2m} \prod_{s=0}^{k-1} \underline{\mathcal A}_{\gamma_{i,s}} \prod_{s=k+1}^{\ell}  \mathcal A_{\gamma_{i,s}}.\label{eq:RkE}
\end{align}
where $T_{\ell,m,k}$ is a subset of $T_{\ell,m,k}'$ such that
\begin{align}
    h\leq   km+ 2m(\ell-k)=m(2\ell-k).\notag
\end{align}
This is because to have non-zero contribution in \eqref{eq:RkE}, each hyperedge in $\gamma_{i,s}$  with $1\leq s\leq k-1$ should appear at least twice.
And using \eqref{eq:degree_identity} again, we have $e\geq 2v-2m$.

\begin{lemma}\label{lem:tangledenumeration}
 Let $\mathcal T_{\ell,m,k}(v,h, e)$ be the set of canonical paths in $T_{\ell,m,k}$ with $v(\gamma)=v, h(\gamma)=h, e(\gamma)=e$. We have
 \begin{align}
     |\mathcal T_{\ell,m,k}(v,e)|\leq (8\ell m)^{12m(e-v-h+1)+8m}.\notag
 \end{align}
\end{lemma}
\begin{proof}
 This is a bipartite graph version of Lemma 18 in \cite{bordenave2018nonbacktracking}. Let $\gamma=(\gamma_1,\dots, \gamma_{2m})\in \mathcal T_{\ell,m,k}$. The only difference compared to the setting in Lemma 18 of \cite{bordenave2018nonbacktracking} is that each $\gamma_{i}$ now corresponds to a walk on length $2\ell-2$ in the bipartite graph, and  $\gamma_i=(\gamma_i^1,\gamma_i^2)$ in $\gamma$ corresponds to two non-backtracking tangle-free paths of length $2k-1$ and $2\ell-2k-1$, respectively.

 We start by reordering $\gamma\in \mathcal T_{\ell,m,k}$ into a new sequence that preserves the connectivity of the path as much as possible. Reorder $\gamma$ into $\hat{\gamma}=(\hat{\gamma}_1,\dots, \hat{\gamma}_{2m})$ by setting $\hat{\gamma_i}=\gamma_i$ for odd $i$ and $\hat{\gamma}_{i,t}=\gamma_{i,\ell-t}$ for even $i$. Also for $i$ odd, we  set $k_i=k-1$ and for $i$ even set $k_i=\ell -k-1$. Write $\hat{\gamma_i}=(\hat{\gamma_i}',\hat{\gamma_i}'')$ with $\hat{\gamma_i}'=(\hat{\gamma}_{i,0},\dots, \hat{\gamma}_{i,k_i-1})$ and $\hat{\gamma_i}''=(\hat{\gamma}_{i,k_i+1},\dots, \hat{\gamma}_{i,\ell})$. To each $i$, we say $\gamma_i$ is connected if $G(\hat{\gamma}_i'')$ intersects the graph $H_i=\left(\bigcup_{j<i} G(\hat\gamma_j) \right)\cup G(\hat{\gamma}_i'')$, otherwise it's disconnected.
 Now as  bipartite graph walk sequences, we write each $\hat{\gamma}_i'=(\hat{x}_{i,0}, \hat{x}_{i,1},\dots, \hat{x}_{i,2k_i},\hat{x}_{i,2k_i-1})$,  and $\hat{\gamma}_i''=(\hat{x}_{i,2k_i+2}, \hat{x}_{i,2k_i+3},\dots, \hat{x}_{i,2\ell },\hat{x}_{i,2\ell+1})$.

 If $\gamma_i$ is disconnected, for $0\leq t\leq \ell$, we set $x_{i,t}=\hat{x}_{i,t}$.  If $\gamma_i$ is connected, we define for $0\leq t\leq 2k_i-1$, $x_{i,t}=\hat{x}_{i,t}$, and if  $q_i>2k_i-1$ is the first time such that $\hat{x}_{i,q_i}\in H_i$, we set for $2k_i+2\leq t\leq q_i$,  $x_{i,t}=\hat{x}_{i,q_i+2k_i+2-t}$, and for $q_i+1\leq t\leq 2\ell+1$, $x_{i,t}=\hat{x}_{i,t}$. We then explore the sequence $(x_{i,t})$ in lexicographic order and set $y_{i,t}=\{x_{i,t},x_{i,t+1} \}$. Recall the  definition of the first time, tree edge and the excess edge are the same  in Lemma \ref{lem:bipartite_NB}. $(i,t)\not=(i,2k_i-1)$ is a first time if the end vertex of $y_{i,t}$, $x_{i,t+1}$ has not been seen before. When $\gamma_i$ is connected, we add the extra mark $(q_i, \hat{x}_{i,q_i})$; if $\gamma_i$ is disconnected, the extra mark is set to $0$. With our ordering, all vertices of $V(\gamma)\setminus \{1\}$ will have an associated tree edge, with the exception of $x_{i,2k_i+2}$ when $\gamma_i$ is disconnected. Let $\delta$ be the number of disconnected $\gamma_i$'s, then there are $\delta+e-v+1$ many excess edges. However, there are at  most $\epsilon=e-v+1$ many excess edges in each connected component of $G(\gamma)$.

 We now repeat the proof of Lemma \ref{lem:bipartite_NB}. The main difference is for each $i$, since $\hat{\gamma}_i'$ and $\hat{\gamma}_i''$ are tangle free, there are at most $2$ short cycling times and $2(\epsilon-1)$ long cycling times. Since there are at most $(2\ell)^{4m\epsilon}$ ways to position these cycling times, we have 
 \[ |\mathcal T_{\ell,m,k}(v,h,e)|\leq (4\ell (v+h))^{2m}(2\ell)^{4m(e-v-h+1)}(2(v+h)^2\ell)^{4m}(v+h)^{8m(e-v-h)},
 \]
 where the fact $(4\ell(v+h))^{2m}$ is the bound on the choice of extra marks.
 Using $v+h\leq 4\ell m$, we obtain the bound.
  \end{proof}

By independence, we must have
  \begin{align}
      \mathbb E\prod_{s=0}^{k-1} \underline {\mathcal A}_{\gamma_{i,s}} \prod_{s=k+1}^{\ell}  \mathcal A_{\gamma_{i,s}}\leq \left( \frac{p_{\max}}{\binom{n}{q-1}}\right)^{h}.\notag
  \end{align}
  By our choice of $m$, 
\begin{align}\label{eq:choiceofm}
    n^{1/2m}&=o(\log^{13} n), \quad  \ell m=o(\log^2 n), \quad (8\ell m)^{12m}\leq n^{24/25}.
\end{align}
Let $S$ be the right-hand side of \eqref{eq:Rkl}.  Using the change of variable $\eta=e-v-h+1$, we find $v\leq \eta+2m+h-1$. Then from \eqref{eq:total_labeling} and \eqref{eq:RkE}, $S$ is bounded by 
\begin{align*}
    & \sum_{h=1}^{m(2\ell-k)}\sum_{v=1}^{m+e/2}\sum_{e=v+h-1}^{2m(2k-1)}p_{\max}^{2m}(q-1)^{2m}(8\ell m)^{12m(e-v-h+1)+8m} \\
    &\cdot n^v \left(\frac{q-1}{n-q+2}\right)^{e-h} \binom{n}{q-1}^h\binom{n}{q-2}^{2m}\left( \frac{p_{\max}}{\binom{n}{q-1}}\right)^{h}\\
    &\leq  c_1[(q-1)p_{\max}]^{2m}(8\ell m)^{8m}n\binom{n}{q-2}^{2m}\sum_{h=1}^{m(2\ell-k)}\sum_{\eta=0}^{\infty}\sum_{v=1}^{\eta+2m+h-1}(8\ell m)^{12m\eta } n^{-\eta}p_{\max}^h (q-1)^{\eta+v-1}\\
   & \leq c_2[(q-1)p_{\max}]^{2m}(8\ell m)^{8m}(q-1)^{2m-2}n\binom{n}{q-2}^{2m}\sum_{h=1}^{m(2\ell-k)}p_{\max}^h (q-1)^{h}\sum_{\eta=0}^{\infty}(8\ell m)^{12m\eta } n^{-\eta}(q-1)^{2\eta}\\
   &\leq c_3[(q-1)p_{\max}]^{2m}(8\ell m)^{8m}(q-1)^{2m-2}n\binom{n}{q-2}^{2m}[(q-1)p_{\max}]^{m(2\ell-k)}.
\end{align*}
Hence 
\begin{align}
      \mathbb E\|R_k^{(\ell)}\|^{2m}=o(\log^{42m}n)[(q-1)p_{\max}]^{(2\ell-k)m}\binom{n}{q-2}^{2m}.\notag
\end{align}
Then \eqref{eq:Rk} follows from \eqref{eq:choiceofm} and Markov's inequality.

\subsection{Proof of \eqref{eq:KBk} on  $KB^{(k)}$ and $B^{(k)}$}\label{sec:KBk}
Recall the definition of $K$ from \eqref{eq:defK}. Take $m$ as in \eqref{eq:set_m}.  Then
\begin{align}
    \|KB^{(k-1)}\|^{2m}&\leq \tr[(KB^{(k-1)}(KB^{(k-1)})^*]^m \notag\\
    &= \sum_{\gamma\in W_{k,m}} \prod_{i=1}^m K_{(\gamma_{2i-1,0},\gamma_{2i-1,1})}\prod_{s=1}^k \mathcal A_{\gamma_{2i-1,s}} \prod_{s=0}^{k-1} \mathcal A_{\gamma_{2i,s}}K_{(\gamma_{2i,k},\gamma_{2i,k+1})} \notag\\
    &\leq p_{\max}^{2m}\sum_{\gamma\in W_{k,m}}\prod_{i=1}^m \prod_{s=1}^k \mathcal A_{\gamma_{2i-1,s}} \prod_{s=0}^{k-1} \mathcal A_{\gamma_{2i,s}}.\label{eq:SKB}
\end{align}

For each canonical path $\gamma \in \mathcal W_{k,m}$ with $v,e,h$ defined as before, since there are at most $m$ distinct hyperedges are covered by the union of $\{\gamma_{2i-1,0}, \gamma_{2i,k}\}$, we obtain 
\begin{align}
    \mathbb E \prod_{i=1}^m \prod_{s=1}^k \mathcal A_{\gamma_{2i-1,s}} \prod_{s=0}^{k-1} \mathcal A_{\gamma_{2i,s}}\leq \left( \frac{p_{\max}}{\binom{n}{q-1}}\right)^{h-m}.\notag
\end{align}
Let $S$ be the right-hand side of \eqref{eq:SKB}. From Lemma \ref{lem:bipartite_NB},  $S$ is bounded by
\begin{align*}
  &p_{\max}^{-m}\binom{n}{q-1} ^m \sum_{h=1}^{km}\sum_{v=1}^{m+e/2}\sum_{e=v+h-1}^{2m(2k-1)}(3k)^{2m}(6km)^{6m(e-v-h+1)}  n^v \left(\frac{q-1}{n-q+2}\right)^{e-h} \binom{n}{q-1}^h\left(\frac{p_{\max}}{\binom{n}{q-1}}\right)^{h}.
\end{align*}
By our choice of $m$, and a change of variable $\eta=e-v-h+1$, following the same steps in the proof in Section \ref{sec:Deltak}, $S$ can be further bounded by 
\begin{align}
   o(\log ^{17m} n) \binom{n}{q-1}^m  [(q-1)p_{\max}]^{2km}.\notag
\end{align}
 Then the result holds by using  Markov's inequality. The bound on $B^{(k)}$ follows similarly.

\subsection{Proof of \eqref{eq:Sk} on  $S_k^{(\ell)}$}

 Recall the definitions of $S_t^{(\ell)}$, $L$, and $K^{(2)}$ from \eqref{eq:defK2}, \eqref{eq:def:LL} and \eqref{eq:def_Sk_l}. 
 From \eqref{eq:defK2}, we have 
\begin{align}
 K^{(2)}_{(x\to e), (y\to f)} &= \frac{1}{\binom{n}{q-2}} \sum_{(x, e) \to (w, g) \to (y, f)} p_{\underline \sigma(g)}=\frac{1}{\binom{n}{q-2}} \sum_{w\in e,w\not=x, y}~\sum_{\substack{g: \{ w,y\}\subset g,\\ g\not=e,f}}p_{\underline \sigma(g)}.\label{eq:K2_further}
\end{align}


Recall the definition of $D^{(2)}$ from \eqref{eq:defD2} and \eqref{eq:eigen_decom}.
We now define a matrix $D''$ whose $(i,j)$-entry is  the expected number of hyperedges 
 containing two vertices of type $i$ and $j$.  More specifically, for $i\not=j$, and any $x,y\in [n]$ such  that $\sigma(x)=i,\sigma(y)=j$, define
\begin{align}
 D''_{ij}=\frac{1}{\binom{n}{q-2}}\sum_{g: \{x,y\}\subset g} p_{\underline{\sigma}(g)},\notag
\end{align}
And for any $i\in [r]$, and   $x'\not=x$ with $\sigma(x')=\sigma(x)=i$, define
\begin{align}
 D''_{ii}=\frac{1}{\binom{n}{q-2}}\sum_{g: \{x,x'\}\subset g} p_{\underline{\sigma}(g)}.\notag
\end{align}
 By the symmetry of the probability tensor $\mathbf{P}$, the definition of $D''$ above is independent of the choice of $x,y$, and $x'$.   
 Recall the definition of $Q$ in \eqref{eq:Q=DPi}. Let \[Q''_{ij}=D''_{ij}\pi_j\] and  $\{\mu''_i \}_{i\in [r]}$ be the eigenvalues of $Q''$. Since $|D_{ij}-D''_{ij}|=O(n^{-1})$, we have \[|\mu_i''-\mu_i|\leq \| Q-Q''\|_F=O(n^{-1})\]
 for any $i\in [r]$. $Q''$ also  has a similar spectral decomposition as $Q$ from \eqref{eq:eigen_decom}:
 \[ Q''=\sum_{i\in [r]}\mu_i''\phi_i''(\Pi \phi_i'')^*. 
 \]
 
 Following the derivations in \eqref{eq:pi_orthogonal} and \eqref{eq:eigen_decom}, 
 $D''$ can be decomposed as 
 \[D''=\sum_{i\in [r]}\mu_i''\phi_i'' \phi_i''^*\]
 such that \[   \langle \phi_i'',\phi_j''\rangle_{\pi}:=\sum_{k\in [r]}\pi_k\phi_i''(k)\phi_j''(k)=\delta_{ij}.\]
 Denote \begin{align}\label{eq:defoverlineD}
 \overline{ D}:=\sum_{k=1}^r   \mu_k (J \chi_k)   \chi_k^*.
 \end{align}
 Correspondingly, we define 
\begin{align}\label{eq:defoverlineDp}
\overline{ D''}:=\sum_{k=1}^r   \mu_k'' (J \chi_k'')   \chi_k''^*,
\end{align}
where $\chi_k''$ is the lifted vector from $\phi_k''$ according to \eqref{eq:defchi}.
 Then
 \[ L=K^{(2)}-\overline{D}= K^{(2)}-\overline{D''} +(\overline{D''}-\overline D).\]
 
By \eqref{eq:def_P} and \eqref{eq:defchi}, we get  
\begin{align}
    \overline{D''}_{(x\to e),(y\to f)}&= \sum_{k=1}^r  \mu_k'' (P \chi_k'')(x\to e)  \chi_k''(y\to f) \notag\\
    &=\sum_{w\not=x, w\in e}  D''_{\sigma(w)\sigma(y)}=\sum_{ \substack{w\not=x, y,\\ w\in e}} D''_{\sigma(w)\sigma(y)}+ \mathbf{1} \{y\in e, x\not=y\}D''_{\sigma(y),\sigma(y)}.\notag
\end{align}

And $K^{(2)}_{(x\to e), (y\to f)}$ in \eqref{eq:K2_further} can be further expanded as  
\begin{align*}
 &\sum_{\substack{w\not=x, y\\ w\in e}}D''_{\sigma(w)\sigma(y)} -\frac{1}{\binom{n}{q-2}} \sum_{w\in e,w\not=x, y}\mathbf{1}\{w,y\in e\} p_{\underline \sigma(g)}-\frac{1}{\binom{n}{q-2}} \sum_{w\in e,w\not=x, y}\mathbf{1}\{w,y\in f\} p_{\underline \sigma(f)}\\
  &=\sum_{\substack{w\not=x, y\\ w\in e}}D''_{\sigma(w)\sigma(y)}-\frac{1}{\binom{n}{q-2}}p_{\underline \sigma(g)}\mathbf{1}\{y\in e\} \cdot |e\setminus \{x, y\} |-\frac{1}{\binom{n}{q-2}}p_{\underline \sigma(f)}\cdot  |e\cap f\setminus \{x, y\} |,
\end{align*}
where $\{x,y\}$ is a set of one or two elements depending on whether $x=y$. 

Therefore $(K^{(2)}-\overline{D''})_{(x\to e),(y\to f)}$ can be written as 
\begin{align*}
&-\frac{1}{\binom{n}{q-2}}p_{\underline \sigma(g)}\mathbf{1}\{y\in e\} \cdot |e\setminus \{x, y\} |-\frac{1}{\binom{n}{q-2}}p_{\underline \sigma(f)}\cdot  |e\cap f\setminus \{x, y\} |
 -\mathbf{1} \{y\in e, y\not=x\}D''_{\sigma(y),\sigma(y)}\\
    =&-\frac{1}{\binom{n}{q-2}}p_{\underline \sigma(g)}\mathbf{1}\{y\in e\} \cdot |e\setminus \{x, y\} |-\frac{1}{\binom{n}{q-2}}p_{\underline \sigma(f)}\cdot  |e\cap f\setminus \{x, y\} |\\
    &-\mathbf{1} \{y\in e, y\not=x,e\not=f\}D''_{\sigma(y),\sigma(y)}
    -\mathbf{1}\{ y\not=x,e=f\}D''_{\sigma(y),\sigma(y)}.
\end{align*}

Now we can decompose \[-L=\overline{D''}-K^{(2)}+(\overline{D}-\overline{D''})=L^{(1)}+L^{(2)}+L^{(3)}+L^{(4)}+L^{(5)}\]
with
\begin{align*}
   L^{(1)}_{(x\to e),(y\to f)}&= \frac{1}{\binom{n}{q-2}}p_{\underline \sigma(g)}\mathbf{1}\{y\in e\} \cdot |e\setminus \{x, y\} |\\
   L^{(2)}_{(x\to e),(y\to f)}&= \frac{1}{\binom{n}{q-2}}p_{\underline \sigma(f)}\cdot  |e\cap f\setminus \{x, y\} |\\
     L^{(3)}_{(x\to e),(y\to f)}&=\mathbf{1}\{ y\not=x,e=f\}D''_{\sigma(y),\sigma(y)}\\
     L^{(4)}_{(x\to e),(y\to f)}&=\mathbf{1} \{y\in e, y\not=x,e\not=f\}D''_{\sigma(y),\sigma(y)}\\
      L^{(5)}&=\overline{D}-\overline{D''}.
\end{align*}

Let 
\begin{align}\label{eq:K1}
K^{(1)}_{x\to e,y\to f}=\ind\{y\in e,y\not=x,e\not=f \}.
\end{align}
Recall $K$ is defined in \eqref{eq:defK}. It's easy to check the following two identities hold:
\begin{align*}
 & [(J+I)(K+p_{\max} I)]_{x\to e,y\to f} 
 =  \left(p_{\underline \sigma(e)}(q-1)\ind \{y\in e, e\not=f\} +p_{\max}\ind\{e=f\}\right),\\
 &[(K^{(1)}+ I)(J+I)]_{x\to e,y\to f} 
 =\left(\ind\{e\not=f\} |e\cap f\setminus \{x\}| +\ind\{ e=f\}\right).\notag
\end{align*}
Direct calculation implies the following entry-wise bounds:
\begin{align*}
 L^{(1)}_{x\to e,y\to f}&\leq \frac{q-1}{\binom{n}{q-2}}[(J+I)(K+p_{\max} I)]_{x\to e,y\to f},\\
    L^{(2)}_{x\to e,y\to f}&\leq \frac{(q-1)p_{\max}}{\binom{n}{q-2}}[(K^{(1)} +I)(J+I)]_{x\to e,y\to f}.
\end{align*}
Since $L^{(1)}$ and $B^{(\ell)}$ all have nonnegative entries, we find
\[ \norm*{\Delta^{(t-1)}L^{(1)}B^{(\ell-t-1)}}\leq \frac{(q-1)}{\binom{n}{q-2}}\norm{\Delta^{(t-1)}} \cdot \norm{J+I} \cdot \left(\norm{KB^{(\ell - t-1)}} + p_{\max}\norm{B^{(\ell-t)}} \right).\]
Note that  $J$ is Hermitian and from Lemma \ref{lem:formula_P}, $\|J\|\leq q-1$. Using \eqref{eq:Delta} and \eqref{eq:KBk} we obtain  with probability at least $1-n^{-c}$,
\begin{align}\label{eq:boundL1}
    \norm*{\Delta^{(t-1)}L^{(1)}B^{(\ell-t-1)}}=O(\sqrt{n} (\log^{20}n) [(q-1)p_{\max}]^{\ell-t/2}).
\end{align}
Similarly,

\[ \norm*{\Delta^{(t-1)}L^{(2)}B^{(\ell-t-1)}}\leq \frac{(q-1)p_{\max}}{\binom{n}{q-2}}\norm{\Delta^{(t-1)}}  \left( \norm{K^{(1)}(J+I)B^{(\ell-t-1)}}+q\norm{B^{(\ell-t-1)}}\right)\]
Note that from \eqref{eq:K1}, $K^{(1)}$ is the non-backtracking operator of the complete $q$-uniform hypergraph, by Lemma \ref{lem:BPPB}, $K^{(1)}P=P{K^{(1)}}^*$.
Since the definition of $K$ and $K^{(1)}$ only  differs by a multiplicative factor of at most $p_{\max}$, applying the same proof argument for $KB^{(k)}$ in Section \ref{sec:KBk}, we can show that with probability at least $1-n^{-c}$, 
\begin{align}
\norm{K^{(1)*} B^{(\ell-t-1)}}&\leq \binom{n}{q-1}^{1/2} (\log n)^{10} [(q-1)p_{\max}]^{\ell-t-1},\notag
\end{align}
which implies
\begin{align}\label{eq:boundL2}
    \norm*{\Delta^{(t-1)}L^{(2)}B^{(\ell-t-1)}}=O(\sqrt{n} (\log^{20}n) [(q-1)p_{\max}]^{\ell-t/2}).
\end{align}
Similarly, we have
\[L^{(3)}_{(x\to e),(y\to f)} \leq p_{\max} P_{(x\to e),(y\to f)}, \]
so we can bound 
\begin{align}\label{eq:boundL4}
    \norm*{\Delta^{(t-1)}L^{(3)}B^{(\ell-t-1)}} \leq p_{\max}\norm{\Delta^{(t-1)}}  \norm{P}  \norm{B^{(\ell-t-1)}}=O\left( \log^{20}(n) [(q-1)p_{\max}]^{\ell-t/2}\right). 
    \end{align}

We apply the moment method to bound
\begin{align}
   M= \Delta^{(t-1)}L^{(4)} B^{(\ell-t-1)}.
\end{align}
 We have
\begin{align}\label{eq:M_tracebound}
    \|M\|^{2m}&\leq \tr[MM^*]^m = \sum_{\gamma \in F_{\ell,m,t}} \prod_{i=1}^{2m}\prod_{s=0}^{t-1}\underline {\mathcal{A}}_{\gamma_{i,s}} L_{\gamma_{i,t-1},\gamma_{i,t}}^{(4)} \prod_{s=t}^{\ell}  \mathcal A_{\gamma_{i,s}}
\end{align}
where $F_{\ell,m,t}$ is the set of sequences $\gamma=(\gamma_1,\dots, \gamma_{2m})$ such that  $\gamma_i^1=(\gamma_{i,0},\dots,\gamma_{i,t-1}),  \gamma_i^2=(\gamma_{i,t},\dots, \gamma_{i,\ell})$ are non-backtracking tangle-free, $\gamma_i=(\gamma_i^1,\gamma_i^2)$ is a non-backtracking walk, and for all $1\leq i\leq m$,
$
    \gamma_{2i-1,\ell}=\gamma_{2i,\ell}, \gamma_{2i,0}=\gamma_{2i+1,0},
$
with the convention that $\gamma_{2m+1}=\gamma_1$. Here we use the definition of $L^{(4)}$ to guarantee $(\gamma_{i,t-1}, \gamma_{i,t})$ is a non-backtracking walk.

Define $G(\gamma)=(V(\gamma), H(\gamma), E(\gamma))$ to be the bipartite representation of $\gamma$ with parameters $v,e,h$ as before. 
Taking the expectation in \eqref{eq:M_tracebound} implies
\begin{align}
     \mathbb E\|M\|^{2m}
    &\leq p_{\max}^{2m}\sum_{\gamma \in F_{\ell,m,t}}\mathbb E \prod_{i=1}^{2m} \prod_{s=0}^{t-1} \underline{\mathcal A}_{\gamma_{i,s}} \prod_{s=t}^{\ell}  \mathcal A_{\gamma_{i,s}}.\label{eq:MkE}
\end{align}
where $F_{\ell,m,t}'$ is a subset of $F_{\ell,m,t}$ such that
\begin{align}
    h \leq   tm+ 2m(\ell-t+1)=m(2\ell-t+2).\notag
\end{align}
This is because to have non-zero contribution in \eqref{eq:RkE}, each hyperedge in $\gamma_{i,s}$ with $0\leq s\leq t-1$ should appear at least twice.

\begin{lemma}
 Let $\mathcal F_{\ell,m,t}(v,h,e)$ be the set of canonical paths in $ F_{\ell,m,t}$ with given parameters $v,h,e$. We have
 \begin{align}
     |\mathcal F_{\ell,m,t}(v,h, e)|\leq (8\ell m)^{12m(e-v-h+1)+8m}. \notag 
 \end{align}
\end{lemma}
\begin{proof}
In the proof of Lemma \ref{lem:tangledenumeration}, the assumption that each $\gamma_i$ is tangled was not used, and we only use the fact that $\gamma_{i}^1$ and $\gamma_{i}^2$ are non-backtracking tangle-free for $1\leq i\leq 2m$. So the same bound holds for $|\mathcal F_{\ell,m,t}(v,h, e)|$. \end{proof}
 
 The following expectation bound holds:
  \begin{align}
      \mathbb E\prod_{s=0}^{t-1} \underline {\mathcal A}_{\gamma_{i,s}} \prod_{s=t}^{\ell}  \mathcal A_{\gamma_{i,s}}\leq \left( \frac{p_{\max}}{\binom{n}{q-1}}\right)^{h}.\notag
  \end{align}
  We set  
\begin{align}
    m=\left\lfloor \frac{\log n}{25 \log(\log n)}\right\rfloor .\notag
\end{align}
Let $S$ be the right-hand side of \eqref{eq:MkE}. We have  $S$ is bounded by
\begin{align*}
    \sum_{h=1}^{m(2\ell-t+2)}\sum_{v=1}^{m+e/2}\sum_{e=v+h-1}^{2m(2t-1)}p_{\max}^{2m} (8\ell m)^{12m(e-v-h+1)+8m}  n^v \left(\frac{q-1}{n-q+2}\right)^{e-h} \binom{n}{q-1}^h \left( \frac{p_{\max}}{\binom{n}{q-1}}\right)^{h}
\end{align*}

Then following every step in Section \ref{sec:proof_Rkl}, we have
\begin{align*}
    S&= o(\log^{42m}n)p_{\max}^{2m}[(q-1)p_{\max}]^{(2\ell-t)m}.
 \end{align*}
Hence  from Markov's inequality, with probability at least $1-n^{-c}$,
\begin{align}\label{eq:boundL3}
     \|M\|\leq  (\log n)^{25} [(q-1)p_{\max}]^{\ell-t/2}.
\end{align}

Finally, using \eqref{eq:defoverlineD} and \eqref{eq:defoverlineDp}, the term involving $L^{(5)}$ can be written as 
 \begin{align}
    \Delta^{(t-1)}(\overline{D}-\overline{D''})B^{(\ell-t-1)}= \Delta^{(t-1)}\left(  \sum_{k\in [r]} (\mu_k J \chi_k \chi_k^*-\mu_k''J \chi_k'' \chi_k''^*)\right)B^{(\ell-t-1)}. \notag 
 \end{align} 
 Following the proof in \ref{sec:Delta_kPchi}, the same estimate \eqref{eq:Delta_chi} also holds for $\|\Delta^{(t-1)}J \chi_k''\|$.
Then with probability at least $1-2n^{-c}$,
\begin{align}\label{eq:boundL5}
    \norm{ \Delta^{(t-1)}L^{(5)}B^{(\ell-t-1)}}&\leq C \sum_{k\in [r]}(\|\Delta^{(t-1)}J \chi_k\| +\|\Delta^{(t-1)}J \chi_k''\|)\|B^{\ell-t-1}\| \notag \\
    &=O( \sqrt{n}(\log n)^{15} [(q-1)p_{\max}]^{\ell-t/2}), 
\end{align}
where the last inequality is due to  \eqref{eq:Delta_chi} and \eqref{eq:KBk}. 

With \eqref{eq:boundL1}, \eqref{eq:boundL2}, \eqref{eq:boundL4},  \eqref{eq:boundL3}, and \eqref{eq:boundL5}, for sufficiently large $n$,
\[ \|S_{k}^{(\ell)}\|\leq  (\log  n)^{30} n^{\kappa+1/2} [(q-1)p_{\max}]^{-k/2}
\]with probability at least $1-7n^{-c}$.
This finishes the proof of \eqref{eq:Sk}.

\section{Miscellaneous computations}\label{sec:Misc}

\subsection{Proof of Lemma \ref{lem:Ihara-Bass}}\label{sec:IHBass}
  Recall the definition of $J$ from \eqref{eq:def_P}.  Fix any $z\in \mathbb C$ such that $zI+J$ is invertible. We have 
\begin{align}
    \det (B-zI)&=\det (-zI-J+T^*S)\notag \\
    &=\det (-J-zI)\det (I-S(zI+J)^{-1}T^*), \label{eq:detB-z}
\end{align}
where in the equation above, we use the determinant formula that when  $X$ is an invertible square matrix,
\[\det(X+YZ)=\det(X)\det(I+YX^{-1}Z).
\] 
From the identity \begin{align}\label{eq:appendix_P}
    J^2=(q-2)J+(q-1)I,
\end{align}
we find for $z\not\in \{-(q-1),1\}$,
\begin{align}\label{eq:inv_P}
    (zI+J)^{-1}=\frac{-1}{(z+q-1)(z-1)}J+\frac{z+q-2}{(z+q-1)(z-1)}I.
\end{align}
Therefore from \eqref{eq:detB-z} and \eqref{eq:inv_P},
\begin{align}
     &\det (B-zI)  \notag\\
     =&\det (-J-zI) [(z+q-1)(z-1)]^{-n}  
       \det ((z+q-1)(z-1)+SJT^*+(z+q-2)ST^*) \notag\\
     =&\det (-J-zI) [(z+q-1)(z-1)]^{-n}\det ((z^2+(q-2)z)I-zA+(q-1)(D-I)),\label{eq:BBBBBB}
\end{align}
where in the last line, we use the identities in \eqref{eq:T_S_P_relations}. 
Now it remains to calculate $\det (-J-zI)$, and we can find all the eigenvectors of $J$.
Let \[v_e=\sum_{x\in e} \mathbf{1}_{x\to e},\] where $\mathbf{1}_{x\to e}\in \mathbb C^{\vec H}$ is a vector taking $1$ at $x\to e$ and $0$ elsewhere. We find 
\[Jv_{e}=\sum_{x\in e} J\mathbf{1}_{x\to e}=(q-1)v_{e}.\]
Since all $\{v_{e}, e\in H\}$ are linearly independent, this gives $m$ linearly independent eigenvectors of $-J$ with respect to $-(q-1)$.
On the other hand, let \[v_{x\to e}=(q-1)\mathbf{1}_{x\to e}- J\mathbf{1}_{x\to e}=(q-1)\mathbf{1}_{x\to e}-\sum_{y\in e,y\not=x} (q-1)\mathbf{1}_{x\to e}.\]
Then from the identity \eqref{eq:appendix_P},
\[Jv_{x\to e}=(q-1)J\mathbf{1}_{x\to e}-J^2\mathbf{1}_{x\to e}=-v_{x\to e}.
\]
Then $v_{x\to e}$ are all eigenvectors of $-J$ with respect to $1$. Since \[ \sum_{x\in e} v_{x\to e}=0,
\]
we see the linear space spanned by  eigenvectors $\{v_{x\to e}, x\in e \}$ is of dimension  $(q-1)$ and for $e'\not=e$, $v_{x\to e}, v_{y\to e'}$ are all linearly independent. Then, we have found $(q-1)m$ many independent eigenvectors of $-J$ for the eigenvalue $1$. This gives us the factorization 
\[ \det (-J-zI)=(z-1)^{(q-1)m}(z+q-1)^m.
\]
With \eqref{eq:BBBBBB},  we have for any $z\not\in \{1,-(q-1)\}$, 
\begin{align}\label{eq:detBformulaB}
\det (B-zI)
=(z-1)^{(q-1)m-n}(z+q-1)^{m-n}\det ((z^2+(q-2)z)I-zA+(q-1)(D-I)).
\end{align}
Let $\lambda\not=1, -(q-1)$ be an eigenvalue for $B$. Then from \eqref{eq:detBformulaB}, 
\begin{align}\label{eq:BH_appendix}
  \det ((\lambda^2+(q-2)\lambda)I-\lambda A+(q-1)(D-I))=0.  
\end{align}

Using the block determinant formula that if $ZW=WZ$,
\[ \det \begin{pmatrix}
X & Y\\
Z & W
\end{pmatrix} =\det( XW-YZ),
\]
\eqref{eq:BH_appendix} can be written as  
\begin{align}\label{eq:det_formula_block}
\det  \begin{pmatrix} -\lambda I &  (D-I)\\  -(q-1)I & A-(q-2)I-\lambda I
\end{pmatrix}=\det(\tilde B-\lambda I)=0.
\end{align}
On the other hand, if $\lambda\not\in \{1,-(q-1)\}$ is an eigenvalue of $\tilde{B}$, from \eqref{eq:det_formula_block} and \eqref{eq:detBformulaB}, $\det(B-zI)=0$ therefore $\lambda$ is also an eigenvalue for $B$.
This finishes the proof for the first claim. Since for $m\geq n$, \eqref{eq:detBformulaB} is a polynomial equation in $z$, it must  holds for all $z\in \mathbb C$.
This finishes the proof of the second claim when $m\geq n$.

  \subsection{Proof of Lemma \ref{lem:formula_P}}
  We have
  \begin{align*}
     (J^2)_{(x\to e), (y\to f)}&=\sum_{w\to g} \mathbf{1} \{ e=g, x\not=w\}\mathbf{1} \{ f=g, y\not=w\}\\
     &=\sum_{w\to g} \mathbf{1} \{ e=f=g, x\not=w,y\not=w\}\\
     &=\sum_{w\to e} \mathbf{1} \{ e=f, x\not=w,y\not=w\}\\
     &=\sum_{w\to e} \mathbf{1} \{ e=f, x\not=y, w\not=x,w\not=y\}+\sum_{w\to e} \mathbf{1} \{ e=f, x=y, w\not=x\}\\
     &=(q-2)\mathbf{1}\left\{ e=f, x\not=y\right\}+(q-1)\mathbf{1} \{ e=f, x=y\}.
  \end{align*}
  The entrywise identity implies $J^2=(q-2)J+(q-1)I.$

 \subsection{Proof of Lemma \ref{lem:BPPB}}
  Note that
  \begin{align*}
      (BJ)_{(x\to e),(y\to f)}&=\sum_{s\to w} B_{(x\to e),(s\to w)} J_{(s\to w), (y\to f)}\\
      &=\sum_{s\to w}\mathbf{1}\left\{ s\in e, s\not=x,e\not=w , w=f,s\not=y \right\}\\
      &=\sum_{s\to f}\mathbf{1}\left\{  s\in e,s\not=x, e\not=f, s\not=y\right\}= |\{s: s\in e\cap f,  s\not=x, y\}  |\cdot \mathbf{1} \{ f\not=e\},
  \end{align*}
  and 
  \begin{align*}
      (JB^*)_{(x\to e),(y\to f)}&=\sum_{s\to w} J_{(x\to e),(s\to w)} B_{(y\to f), (s\to w)}\\
      &=\sum_{s\to w} \mathbf{1}\{ e=w, x\not=s, s\in f, s\not=y, w\not=f \}\\
      &=\sum_{s\to e} \mathbf{1} \{ s\not=x, s\in f, s\not=y, e\not=f\}=|\{s: s\in e\cap f, s\not=x, y \} | \cdot \mathbf{1}\{ f\not=e\}.
  \end{align*}
  So $BJ=JB^*$. Then $B^kJ=B^{k-1} JB^*=\cdots =J{B^{*}}^k$. 
  
\subsection{Proof of Lemma \ref{lem:bprime_spec}}

With the notations in Section \ref{sec:IHBass}, we have $v_{\mathrm{in}} = SJ^{-1}v$ and $v_{\mathrm{out}} = Sv$.
 Now, from the equation $Bv = \lambda v$, we have
\begin{equation}\label{eq:expansion_bx}
    T^*Sv - Jv = \lambda  v.
\end{equation}
We first multiply \eqref{eq:expansion_bx} by $SJ^{-1}$, and find
\begin{equation}\label{eq:eigen_vin}
\begin{split}
\lambda v_{\mathrm{in}} &= SJ^{-1}T^* Sv - Sv 
=  SJ^{-1} J S^* v_{\mathrm{out}} - v_{\mathrm{out}} = (D - I) v_{\mathrm{out}}
\end{split}
\end{equation}
through applications of \eqref{eq:T_S_P_relations}.  Now, we multiply \eqref{eq:expansion_bx} by $S$ instead, so
\[ \lambda v_{\mathrm{out}} = S T^* S v - SJv.  \]
Hence
\begin{equation}\label{eq:eigen_vout}
\begin{split}
    \lambda v_{\mathrm{out}} &= ST^*Sv - S((q-2)I + (q-1)J^{-1})v \\
    &= A v_{\mathrm{out}} - (q-2)v_{\mathrm{out}} - (q-1)v_{\mathrm{in}}.
\end{split}
\end{equation}

Combining \eqref{eq:eigen_vin} and \eqref{eq:eigen_vout}, we find
\[ \begin{pmatrix} 
    0 & D - I \\
    -(q-1)I & A - (q-2)I \end{pmatrix} \begin{pmatrix}v_{\mathrm{in}}\\v_{\mathrm{out}} \end{pmatrix} =\lambda \begin{pmatrix}v_{\mathrm{in}}\\v_{\mathrm{out}} \end{pmatrix}, \]
    so $\tilde B v'=\lambda v'$.
    If $v_{\mathrm{out}} = 0$ then $v \in\mathrm{Ker}(S)$, hence \eqref{eq:expansion_bx} reduces to
$-Jv = \lambda v.$
This implies  $\lambda = 1$ or $-(q-1)$, a contradiction. Therefore $v'\not=0$ and $(v',\lambda)$ is an eigenpair for $\tilde B$.
 
To show the "moreover" part of Lemma \ref{lem:bprime_spec}, it suffices to show that for a basis $v^{(1)},\dots, v^{(t)}$ of the eigenspace of $B$ with eigenvalue $\lambda$, the corresponding eigenvectors $v^{(i)'}=\begin{pmatrix} v^{(i)}_{\mathrm{in}}\\ v^{(i)}_{\mathrm{out}}\end{pmatrix}$ are linearly independent. 
Then, by counting dimensions,  the dimension of the eigenspace of $\lambda \not\in \{1,-(q-1)\}$ for $\tilde B$ is the same  as the dimension of the corresponding eigenspace  for $\tilde B$, and $v^{(i)'}, 1\leq i\leq t$ form an eigenbasis of $\tilde B$ with $\lambda$. This implies all the eigenvectors $v'$ of $\tilde{B}$ corresponding to $\lambda$ can be obtained in the form of $v'=\begin{pmatrix} v_{\mathrm{in}}\\ v_{\mathrm{out}}\end{pmatrix}$.

Suppose $w_1 v^{(1)'}+ \cdots w_t v^{(t)'}=0$ for certain weights $w_1,\dots, w_t$. Let $v=w_1v^{(1)}+\cdots w_t v^{(t)}$. Then
\begin{align}
   Sv&= S(w_1v^{(1)}+\cdots w_t v^{(t)})=w_1v_{\mathrm{out}}^{(1)}+\cdots w_t v_{\mathrm{out}}^{(t)}=0, \notag 
\end{align}
which implies $v\in \ker(S)$.
 Then from \eqref{eq:expansion_bx}, $-Jv=\lambda v$.
Since $-J$ only has eigenvalues $1$ and $-(q-1)$, we must have $v=0$. By the linear independence of $v^{(1)},\dots, v^{(t)}$, we must have $w_1=\cdots=w_t=0$. Therefore $v^{(1)'},\dots, v^{(t)'}$ are linearly independent. This finishes the proof of all the statements in  Lemma \ref{lem:bprime_spec}.

\subsection{Proof of Lemma \ref{lem:eigen_reduction}}\label{sec:lemma6}
    We have when $G$ is tangle-free (which happens with probability at least $1-O(n^{-c'})$ from Lemma \ref{lem:tangle-free}),
    \begin{align}\label{eq:SBlh}
        [SB^\ell J \chi_i](x) &= \sum_{e \ni x} [B^\ell J \chi_i](x \to e)= \sum_{x_0 = x,\dots, x_{\ell+1}} \phi_i(\sigma(x_{\ell+1})) = h_{\phi_i, \ell+1}(G, x),
    \end{align}
    where $h_{\phi_i, t}$ was defined in \eqref{eq:defh}. As a result, if we define
    \[ f_1(g, o) = \phi_j(\sigma(o)) h_{\phi_i, \ell+1}(g, o) \quand f_2(g, o) = h_{\phi_i, \ell+1}(g, o)h_{\phi_j, \ell+1}(g, o), \]
    then similarly to Lemmas \ref{lem:scal_u_v} and \ref{lem:scal_uu_vv}, we have
    \[ \langle SB^\ell J \chi_i, \tilde \phi_j \rangle = n\langle \pi, \overline{f_1}  \rangle  + O(n^{1-c}), \quad \langle SB^\ell J \chi_i, SB^\ell J \chi_j\rangle  = n\langle \pi, \overline{f_2} \rangle   + O(n^{1-c}) .\]
    For the first computation, from Lemma \ref{lem:non_backtracking_functional} and \eqref{eq:fphit},
    \[ \overline{f_1}(k) =[(q-1)\mu_i]^{\ell+1}\phi_i(k)\phi_j(k), \]
    hence
    \[ \langle \pi, \overline{f_1} \rangle = [(q-1)\mu_i]^{\ell+1}\delta_{ij}. \]
     The second case is already treated in \eqref{eq:pi_eigen_vec_product}, which shows
     \[ \langle \pi, \overline{f_2} \rangle = [(q-1)\mu_i]^{2\ell+2} \gamma_i^{(\ell)}\delta_{ij}. \]

\subsection{Proof of  Lemma \ref{lem:tvar_poisson_binomial}}

We shall make use of the following classical bounds (see, e.g. \cite{barbour_introduction_2005}):
\begin{lemma}\label{lem:stein_chen}
 Let $n$ be any integer, and $\lambda, \lambda' > 0$. Then
 \[ d_\mathrm{var}\left(\mathrm{Poi}(\lambda), \mathrm{Bin} (n, \lambda/n)\right) \leq \frac{\lambda}{n} \quad \text{and} \quad d_\mathrm{var}(\mathrm{Poi}(\lambda), \mathrm{Poi}(\lambda')) \leq |\lambda - \lambda'|. \]
\end{lemma}

Define the intermediary distribution

\[ \mathbb{P}'_{\underline \tau} = \mathrm{Poi}\left(  \frac{p_{i, \underline j} \cdot \prod_{k=1}^r \dbinom{n_k(t)}{\tau_k}}{\dbinom{n}{q-1}}\right),\]
an immediate application of Lemma \ref{lem:stein_chen} implies
\[ d_\mathrm{var}(\mathbb{P}_{\underline \tau}, \mathbb{P}'_{\underline \tau}) \leq \frac{p_{\max}}{\dbinom{n}{q-1}}. \]
Another application of the same lemma allows us to bound the distance between $\mathbb{P}'_{\underline \tau}$ and $\mathbb{Q}_{\underline \tau}$:

\begin{align*}
d_\mathrm{var}(\mathbb{P}'_{\underline \tau}, \mathbb{Q}_{\underline \tau}) &\leq \left| \frac{p_{i, \underline j} \cdot \prod_{k=1}^r \dbinom{n_k(t)}{\tau_k}}{\dbinom{n}{q-1}} - p_{i, \underline j} \cdot \dbinom{q-1}{\tau_1, \dots, \tau_r} \prod_{k=1}^r \pi_k^{\tau_k} \right| \\
&\leq c \dbinom{q-1}{\tau_1, \dots, \tau_k} \left| \frac{\prod_{k=1}^r \frac{n_k(t)!}{(n_k(t)-\tau_r)!}}{\frac{n!}{(q-1)!}} - \prod_{k=1}^r \pi_k ^{\tau_k} \right|
\end{align*}
The rightmost term is the difference between two products of $q-1$ terms; all terms of the left product are of the form $(n_k(t)-a)/(n - b)$ with $a, b \leq q-1$, and can be paired with a term $\pi_k$ on the right. Further, we have

\begin{align*}
    \left| \frac{n_k(t)-a}{n-b} - \pi_k \right| &\leq \left| \frac{n_k(t)-a}{n-b} - \frac{n_k(t)}{n} \right| + \left| \frac{n_k(t)}{n} - \pi_k \right| \\
    &\leq \left | \frac{an - bn_k(t)}{n(n-b)} \right| + c_1 \frac{d^\ell}n \leq c_2 n^{\kappa-1},
\end{align*}
Since all terms of both products are less than one, we finally find

\[ d_\mathrm{var}(\mathbb{P}_{\underline \tau}, \mathbb{Q}_{\underline \tau}) \leq c_3  p_{\max} r(q-1)!\, n^{\kappa-1}. \]

\subsection{Proof of Lemma \ref{lem:gw_transform_product}}
For the first statement \eqref{eq:general_vec_product}, we have
  \begin{align*}
      \langle \pi, Q^{(3)}(\phi \otimes \phi') \rangle &= \sum_{i \in [r]} \pi_i \sum_{j, k \in [r]} \sum_{\underline \ell \in [r]^{q-3}} p_{ijk, \underline \ell} \left(\prod_{m \in \underline \ell} \pi_m \right) \pi_j \pi_k \phi_j \phi_k' \\
      &= \sum_{j, k \in [r]} \pi_j \phi_j \cdot \left(\sum_{\ell \in [r]^{q-2}} p_{jk, \underline \ell} \prod_{m \in \underline \ell} \pi_m\right) \cdot \pi_k \phi_k'\\
      &= \phi^* \Pi D^{(2)} \Pi \phi'=\phi^* \Pi Q\phi',
  \end{align*}
  where we used the permutation invariance of the tensor $\mathbf P$ in the second line, and the definition  $Q=D\Pi$ from \eqref{eq:Q=DPi} in the last line.
   Now let $\phi_i,\phi_j$ be two eigenvectors of $Q$.  Using the definition of 
  $\langle \cdot, \cdot \rangle_\pi$ in \eqref{eq:pi_orthogonal},
 \begin{align}\label{eq:149}
       \langle \pi, Q^{(3)}(\phi_i \otimes \phi_j)\rangle =\phi_i^* \Pi Q\phi_j=\mu_j\phi_i^* \Pi \phi_j=\mu_j \langle \phi_i, \phi_j\rangle_{\pi}=\mu_i\delta_{ij},
  \end{align}
therefore \eqref{eq:eigen_vec_product} holds. For \eqref{eq:pi_eigen_vec_product}, recall  from \eqref{eq:ffffphit},
\[ \overline{f_{\phi_i, t}f_{\phi_j, t}} = [(q-1)^2\mu_i\mu_j]^t \left( \phi_i\circ \phi_j + \sum_{s=0}^{t-1} \frac{Q^s y^{(\phi_i, \phi_j)}}{[(q-1)\mu_i \mu_j]^{s+1}} \right),
\]
and from \eqref{eq:def_y_phi_phip},
 \begin{align}\label{eq:yyyyphi}
    y^{(\phi_i, \phi_j)} = Q(\phi_i \circ \phi_j) + (q-2) Q^{(3)}(\phi_i \otimes \phi_j), 
 \end{align} 
We have 
\begin{align}\label{eq:pifff}
     \langle \pi, \overline{f_{\phi_i, t}f_{\phi_j, t}} \rangle =  [(q-1)^2\mu_i\mu_j]^t\langle \pi, \phi_i \circ \phi_j\rangle+\sum_{s=0}^{t-1} \frac{[(q-1)^2\mu_i\mu_j]^t}{[(q-1)\mu_i \mu_j]^{s+1}} \langle \pi,Q^s y^{(\phi_i, \phi_j)}\rangle.
\end{align}
  For the first term in \eqref{eq:pifff}, we have
  \[ [(q-1)^2\mu_i\mu_j]^t\langle \pi, \phi_i \circ \phi_j \rangle = [(q-1)^2\mu_i\mu_j]^t \langle \phi_i, \phi_j \rangle_\pi=[(q-1)\mu_i]^{2t}\delta_{ij}. \]
  For the second one, since $Q=D^{(2)}\Pi$, $\pi$ is a left eigenvector of $Q$ with eigenvalue $d$, we have $\pi^* Q^s = d^s \pi^*$ for any $s \geq 0$. Hence, according to \eqref{eq:yyyyphi} and \eqref{eq:149},
  \begin{align*}
      \langle \pi, Q^s y^{(\phi_i, \phi_j)} \rangle &= d^{s+1} \langle \pi, \phi_i \circ \phi_j \rangle + (q-2) d^s \langle \pi, Q^{(3)}(\phi_i\otimes \phi_j) \rangle\\
      &=d^{s+1}\delta_{ij}+(q-2)d^s\mu_i\delta_{ij}.
  \end{align*}
 Then the second statement in Lemma \ref{lem:gw_transform_product} holds.    Recall $\tau_i=\frac{d}{(q-1)\mu_i^2}$. Therefore \eqref{eq:pifff} can be simplified as
 \begin{align*}
    \langle \pi, \overline{f_{\phi_i, t}f_{\phi_j, t}} \rangle &=  [(q-1)\mu_i]^{2t}\delta_{ij}+[(q-1)\mu_i]^{2t}\delta_{ij}\sum_{s=0}^{t-1}[(q-1)\mu_i^2]^{-s-1}(d^{s+1}+(q-2)d^s\mu_i)\\
    &=[(q-1)\mu_i]^{2t}\delta_{ij} \left( 1+\sum_{s=0}^{t-1}\left(\tau_i^{s+1}+\frac{(q-2)\mu_i}{d} \tau_i^{s+1}\right)\right)\\
    &=[(q-1)\mu_i]^{2t}\delta_{ij}\left(1+ \tau_i \frac{1-\tau_i^t}{1-\tau_i}+\frac{(q-2)}{(q-1)\mu_i}\frac{1-\tau_i^t}{1-\tau_i}\right)=[(q-1)\mu_i]^{2t}\delta_{ij}\gamma_{i}^{(t)},
 \end{align*}
 where $\gamma_{i}^{(t)}$ is given in \eqref{eq:def_gamma_i_t}.
This completes the proof.

  \subsection{Proof of Lemma \ref{lem:bulk_B_left}}
From the telescoping sum formula
\begin{align}
    \prod_{s=0}^{\ell} a_s=\prod_{s=0}^{\ell} b_s+\sum_{t=0}^{\ell} \prod_{s=0}^{t-1}b_s(a_t-b_t)\prod_{s=t+1}^{\ell} a_s, \notag
\end{align} 
we can decompose $\Delta^{(\ell)}$ as 
\begin{align} \label{eq:new_decompose}
    \Delta^{(\ell)}_{(x\to e), (y\to f)}=B^{(\ell)}_{(x\to e), (y\to f)}- \sum_{\gamma \in F_{(x\to e), (y\to f)}^{\ell}}\sum_{t=0}^{\ell} \prod_{s=0}^{t-1}{A}_{\gamma_s}\frac{p_{\underline\sigma(\gamma_t)}}{\binom{n}{q-1}}\prod_{s=t+1}^{\ell} \underline{\mathcal A}_{\gamma_s}.
\end{align}
For $0\leq t\leq \ell$, define ${R'_t}^{(\ell)}$ as
\begin{align}\label{eq:new_decompose2}
    ({R'_t}^{(\ell)})_{(x\to e), (y\to f)}=\sum_{\gamma \in F_{t,(x\to e), (y\to f)}^{\ell}}\prod_{s=0}^{t-1}  \mathcal A_{\gamma_s} p_{\underline\sigma(\gamma_t)}\prod_{s=t+1}^{\ell}  \underline{\mathcal A}_{\gamma_s}.
\end{align}

Compared to Lemma \ref{lem:expansionBl}, \eqref{eq:new_decompose} and \eqref{eq:new_decompose2} imply a  similar expansion of $B^{(\ell)}$ as
\begin{align*}
  B^{(\ell)}= & \Delta^{(\ell)}+\frac{1}{\binom{n}{q-1}}K\Delta^{(\ell-1)}+\frac{q-1}{n-q+2}\sum_{t=1}^{\ell-1} B^{(t-1)}K^{(2)} \Delta^{(\ell-t-1)}
    +\frac{1}{\binom{n}{q-1}} B^{(\ell-1)}K-\frac{1}{\binom{n}{q-1}}\sum_{t=0}^{\ell} R_t'^{(\ell)}.
\end{align*}
  Similar to Lemma \ref{lem:Blwnorm}, using the definition  \[L=K^{(2)}-\sum_{k=1}^r \mu_k(J \chi_k) \chi_k^*,  \quad 
    {S_t'}^{(\ell)}=B^{(t-1)}L\Delta^{(\ell-t-1)},
\] we find  $ \|w^*B^{(\ell)}\|$ is bounded by 
    \begin{align*}
     &\| \Delta^{(\ell)}\| +\frac{1}{\binom{n}{q-1}}\|K\Delta^{(\ell-1)}\|+\frac{q-1}{n-q+2} \sum_{j=1}^r  \sum_{t=1}^{\ell-1} \mu_j |\langle  w, B^{(\ell-t-1)}J \chi_j\rangle | \|(\Delta^{(t-1)})^*\chi_j\|  \notag \\
      &+\frac{q-1}{n-q+2}\sum_{t=1}^{\ell-1} \| {S_t'}^{(\ell)}\| +\frac{1}{\binom{n}{q-1}} \|B^{(\ell-1)}K\| +\frac{1}{\binom{n}{q-1}}\sum_{t=0}^{\ell} \|{R_t'}^{(\ell)}\|.
  \end{align*}

 By the same arguments as in the proof of Proposition \ref{prop:trace}, we find with probability at least $1-n^{-c}$, for all $1\leq j\leq r$, $1\leq k\leq \ell$.
  \begin{align}
      \| (\Delta^{(k)})^* \chi_j\| &\leq  \sqrt n(\log n)^5 [(q-1)d]^{k/2},  \label{eq:154} \\
      \|{R_k'}^{(\ell)}\| &\leq \binom{n}{q-2}(\log n)^{25} n^{\kappa}[(q-1)p_{\max}]^{-(\ell-k)/2},  \\
      \| B^{(\ell-1)} K\| &\leq  \binom{n}{q-1}^{1/2} (\log n)^{10} [(q-1)p_{\max}]^{\ell}, 
  \end{align}
   And for all $1\leq k\leq \ell-1$, 
  \begin{align}
      \|{S_{k}'}^{(\ell)}\|\leq  (\log n)^{30} n^{\kappa+1/2}[(q-1)p_{\max}]^{-(\ell-k)/2}. \label{eq:157}
  \end{align}

 Now let $w$ be any unit vector orthogonal to all $B^{\ell} J \chi_i, i\in [r_0]$. With probability at least $1-n^{-c}$, $G$ is $\ell$-tangle free and the bounds in \eqref{eq:154}-\eqref{eq:157} hold. We find  
 \begin{align}\label{eq:158}
     \|w^* B^{(\ell)}\| \leq c_1(\log n)^{10} n^{\kappa/2}+c_1n^{-1/2}(\log n)^5\sum_{i=1}^r \sum_{t=1}^{\ell-1}[(q-1)d]^{t/2} |\langle B^{\ell-t-1}J \chi_i, w\rangle|.
 \end{align}

For $i\in [r_0]$,
\begin{align}
   [ (q-1)\mu_i]^{-t}\langle B^{t}J \chi_i, w\rangle = [ (q-1)\mu_i]^{-t}\langle B^{t}J \chi_i, w\rangle-[ (q-1)\mu_i]^{-\ell}\langle B^{\ell}J \chi_i, w\rangle. \notag 
\end{align}
Hence
\begin{align}\label{eq:sumBB}
&[ (q-1)\mu_i]^{-t}| \langle B^{t}J \chi_i, w\rangle|\\
\leq  &\sum_{s=t}^{\ell-1} [ (q-1)\mu_i]^{-(s+1)}|\langle B^{s+1}J \chi_i,w\rangle -[(q-1)\mu_i]\langle B^{s}J \chi_i,w\rangle| \notag \\
\leq & \sum_{s=t}^{\ell-1}[ (q-1)\mu_i]^{-(s+1)}\left((q-1)^{1/2}[(q-1)d]^{(s+3)/2}\sqrt{n}+c_1 \log(n)^{9/4}n^{3\kappa+1/4}\right), \notag 
\end{align}
where  the last inequality is from \eqref{eq:telescopediff2}.
Since $(q-1)\mu_i>\sqrt{(q-1)d}$ for $i\in [r_0]$,  from \eqref{eq:sumBB}, 
\begin{align}\label{eq:161}
    |\langle B^{t}J \chi_i, w\rangle  |\leq  c_2(q-1)^{1/2} [(q-1)d]^{t/2+1}\sqrt{n}+c_3\log^3(n) n^{3\kappa+1/4}.
\end{align}

For $j>r_0$ in \eqref{eq:lasttermrobound}, we have  $(q-1)\mu_j\leq \sqrt{(q-1)d}$. From \eqref{eq:scal_uuvv1} and the triangle inequality,
\begin{align}\label{eq:162}
     |\langle  B^{\ell-t-1}J \chi_j, w\rangle|\leq \|B^{\ell-t-1}J \chi_j\|\leq  c_4 n^{\kappa/2+1/2}[(q-1)d]^{-t/2}
+c_4 \log^{9/4}n^{2\kappa-1/4}.
\end{align}

From \eqref{eq:158}, \eqref{eq:161}, and \eqref{eq:162},  with probability $1-n^{-c}$,
\begin{align*}
    \| (B^\ell)^{*}w\|\leq  (\log n)^{9}n^{7\kappa/2-1/4}.
\end{align*}
This finishes the proof.

\section{Proof of Theorem \ref{thm:weakreconstruction}}\label{sec:weakreconstruction}
We reproduce the sketch of proof from \cite{stephan2019robustness}; all minutiae can be found in \cite[Appendix B]{stephan2019robustness}. Up to rearranging the eigenspaces of $Q$ (this reduction was already treated in the proof of Theorem \ref{th:main_reduced}), \eqref{eq:approx_tilde_u} in the proof of Theorem \ref{th:main_reduced} and Lemma \ref{lem:eigen_reduction} imply that with high probability, 
\[ \|\tilde u -\bar u_2\| = O(n^{1/2-c}), \quad \text{with} \quad \bar u_2 = \frac{SB^\ell J\chi_2}{[(q-1)\mu_2]^{\ell+1}\sqrt{\gamma_2}},\]
 where $\gamma_2$ is defined in \eqref{eq:def_gamma_i} and $\tilde{u}$ is defined in Algorithm \ref{alg:provable_algorithm}.
 From \eqref{eq:SBlh},  with  probability $1-O(n^{-c'})$, \[\bar u_2(x) =  \frac{1}{[(q-1)\mu_2]^{\ell+1}\sqrt{\gamma_2}} h_{\phi_2, \ell+1}(G, x).\]

The vector $\bar u_2$ then satisfies the weak convergence estimates of Proposition \ref{prop:functional_concentration_final}, so Lemma 16 from \cite{stephan2019robustness} directly translates to our setting but with stronger convergence estimates:
\begin{lemma}
    For any $i \in [r]$, there exists a random variable $X_i$ such that for any $K \geq 0$ that is a continuity point of $X_i$, with probability at least $1-n^{-c}$,
    \[\frac1n\sum_{x\in[n]} \ind_{\sigma(x)=i}\, \tilde u(x) \ind_{|\tilde u(x)| \leq K} = \pi_i \,\E*{X_i \ind_{|X_i| \leq K}}+O(n^{-c'}).\]
    Furthermore, we have
    \[ \sum_{i \in [r]} \pi_i \E{X_i} = 0, \quad  \sum_{i \in [r]} \pi_i \E{X_i^2} = 1 \quand \sum_{i \in [r]} \pi_i \E{X_i}^2 = \gamma_2^{-1}. \]
\end{lemma}

By the concentration bound in Proposition \ref{prop:functional_concentration_final}, for all $i$, with probability at least $1-n^{-c}$, 
\begin{align}\label{eq:mgconvergence}
\frac1n \sum_{x \in [n]} \ind_{\sigma(x) = i} \ind_{x \in V^+} =\pi_i \left( \frac12 + \frac{\E*{X_i \ind_{|X_i| \leq K}}}{2K} \right) +O(n^{-c'}). 
\end{align}
Denote $\tilde{p_i}=\frac{\E*{X_i \ind_{|X_i| \leq K}}}{2K}$.
Assume now that $\pi_i = 1/r, \forall i\in [r]$.  For any $i \in [r]$,
\begin{align*}
    \left| \E{X_i} - \E*{X_i \ind_{|X_i| \leq K}} \right| &= |\E*{X_i \ind_{|X_i| > K}} |\leq \sqrt{\E{X_i^2}\,\Pb{|X_i| > K}} \leq \frac{\E*{X_i^2}}K,
\end{align*}
having used both the Cauchy-Schwarz and Chebyshev inequalities. On the other hand, since
\[ \sum_{i \in [r]}  \E{X_i} = 0 \quand  \sum_{i \in [r]}  \E{X_i}^2 = r\, \gamma_2^{-1}, \]
we have
\[ \sum_{i, j\in [r]} (\E{X_i} - \E{X_j})^2 = 2r^2 \gamma_2^{-1}. \]
As a result, there exists $i, j \in [r]$ such that
\[ |\E{X_i} - \E{X_j}| \geq \sqrt{2} \gamma_2^{-1/2}. \]
Choosing $K = 2\sqrt{2}r\gamma_2^{1/2}$, and using that $\E{X_i^2} \leq r$,
\[ |\tilde p_i - \tilde p_j| \geq \frac1{2K}\left(\sqrt{2} \gamma_2^{-1/2} - \frac{2r}K\right) = \frac1{8r\gamma_2}. \]
Assume without generality that $\tilde p_i > \tilde p_j$. Let $\tau$ be a permutation on $[r]$ such that $\tau(i) = 1$, $\tau(j) = 2$. Assigning label $1$ to vertices in $V^+$ and label $2$ to vertices in $V^-$, from \eqref{eq:def_ov} and \eqref{eq:mgconvergence},  the overlap satisfies 
\begin{align*} \mathrm{ov}(\hat \sigma, \sigma) &\geq \frac1n \sum_{x \in [n]} \ind_{\sigma(x) = i} ~\ind_{x \in V^+} +\frac1n \sum_{x \in [n]} \ind_{\sigma(x) = j} ~\ind_{x \in V^-} \\
&= \frac{\tilde p_1}{r} + \frac{1-\tilde p_2}{r} + O(n^{-c'}) \geq \frac1r +  \frac1{8r\gamma_2} + O(n^{-c'}). 
\end{align*}
 This finishes the proof of Theorem \ref{thm:weakreconstruction}.

\end{document}